\documentclass{article}

\usepackage{PRIMEarxiv}

\usepackage[utf8]{inputenc} 
\usepackage[T1]{fontenc}    
\usepackage{hyperref}       
\usepackage{url}            
\usepackage{booktabs}       
\usepackage{amsfonts}       
\usepackage{nicefrac}       
\usepackage{microtype}      
\usepackage{xcolor}         
\usepackage{makecell}

\usepackage[english]{babel}
\usepackage{amsmath,amsfonts,amssymb}
\usepackage[export]{adjustbox}
\usepackage{xcolor}
\usepackage{soul}

\newcommand{\N}{\mathbb{N}}
\newcommand{\R}{\mathbb{R}}
\newcommand{\E}{\mathbb{E}}

\usepackage{bm}
\usepackage{multirow}
\usepackage{amsthm}
\usepackage{hyperref}  
\usepackage{thm-restate}

\usepackage{lipsum}		
\usepackage{graphicx}
\usepackage[numbers]{natbib}
\usepackage{algorithm} 
\usepackage{algorithmic}
\hypersetup{colorlinks, linkcolor=blue, citecolor=blue, urlcolor=magenta, linktocpage, plainpages=false}
\usepackage{enumitem}

\usepackage{subfigure}

\usepackage{caption}

\usepackage[final]{changes}

\newcommand{\mA}{\mathbf{A}}
\newcommand{\mB}{\mathbf{B}}

\newcommand{\mI}{\mathbf{I}}

\newcommand{\mV}{\mathbf{V}}
\newcommand{\mW}{\mathbf{W}}
\newcommand{\mX}{\mathbf{X}}
\newcommand{\mY}{\mathbf{Y}}
\newcommand{\mZ}{\mathbf{Z}}
\newcommand{\mU}{\mathbf{U}}

\newcommand{\mDelta}{\mathbf{\Delta}}
\newcommand{\mH}{\mathbf{H}}
\newcommand{\mTheta}{\mathbf{\Theta}}

\newcommand{\cD}{\mathcal{D}}
\newcommand{\cL}{\mathcal{L}}

\newcommand{\cF}{\mathcal{F}}
\newcommand{\cG}{\mathcal{G}}
\newcommand{\cO}{\mathcal{O}}
\newcommand{\T}{^\top}
\newcommand{\avein}{\frac{1}{n}\sum_{i=1}^n}

\newcommand{\bfone}{\mathbf{1}}
\newcommand{\bfonet}{\mathbf{1}^{\top}}
\newcommand{\norm}[1]{\left\|#1\right\|}

\def\<#1,#2>{\left\langle #1,#2\right\rangle}
\DeclareMathOperator*{\argmax}{arg\,max\,}
\DeclareMathOperator*{\argmin}{arg\,min\,}

\theoremstyle{plain}
\newtheorem{theorem}{Theorem}[section]

\newtheorem{lemma}[theorem]{Lemma}
\newtheorem{corollary}[theorem]{Corollary}
\theoremstyle{definition}
\newtheorem{definition}{Definition}
\newtheorem{assumption}{Assumption}
\newtheorem{remark}{Remark}
\allowdisplaybreaks

\usepackage{thm-restate}

\title{Fully First-Order Methods for Decentralized Bilevel Optimization}

\author{Xiaoyu Wang$^{*}$$^1$, \, Xuxing Chen\thanks{denotes equal contributions.}\, $^2$, \, Shiqian Ma$^3$, \, and Tong Zhang$^4$ \\
$^1$The Hong Kong University of Science and Technology \\ $^2$University of California Davis \\$^3$Rice University, $^4$University of Illinois Urbana-Champaign \\  \texttt{maxywang@ust.hk}, \texttt{xuxchen@ucdavis.edu}, \texttt{sqma@rice.edu}, \\\texttt{tongzhang@tongzhang-ml.org}}

\begin{document}

\maketitle

\begin{abstract}
    This paper focuses on decentralized stochastic bilevel optimization (DSBO) where agents only communicate with their neighbors. We propose Decentralized Stochastic Gradient Descent and Ascent with Gradient Tracking (DSGDA-GT), a novel algorithm that only requires first-order oracles that are much cheaper than second-order oracles widely adopted in existing works. We further provide a finite-time convergence analysis showing that for $n$ agents collaboratively solving the DSBO problem, the sample complexity of finding an $\epsilon$-stationary point in our algorithm is $\cO(n^{-1}\epsilon^{-7})$, which matches the currently best-known results of the single-agent counterpart with linear speedup. The numerical experiments demonstrate both the communication and training efficiency of our algorithm.
\end{abstract}

\section{Introduction}
\label{intro}
Bilevel optimization (BO) has recently gained growing attention in the machine learning community due to its effectiveness in various applications such as hyperparameter optimization~\citep{pmlr-v22-domke12,maclaurin2015gradient, iterative_der,lorraine2020optimizing}, meta-learning~\citep{andrychowicz2016learning,franceschi2018bilevel,rajeswaran2019meta}, reinforcement learning~\citep{NEURIPS2019_9713faa2, hong2023two}, and many others~\citep{sinha2017review}. Mathematically, the bilevel optimization problem can be formulated as follows
\begin{align}
 \min_{x \in  \R^p} & \quad \Phi(x) = f(x, y^{\ast}(x)), \quad 
 \text{s.t.}\, \, y^{\ast}(x) = \arg\min_{y\in \R^q} g(x,y)  \label{eq: bo_opt}
\end{align}
where $g$ is the lower-level (LL) function and is usually assumed to be strongly convex with respect
to $y$ for all $x$, and $f$ is the upper-level (UL) function which is possibly non-convex. A natural strategy to solve problem \eqref{eq: bo_opt} is to estimate $\nabla \Phi(x)$ (which we call hypergradient), and then perform hypergradient descent on $x$. Under certain smoothness and {strongly convexity assumptions}, the hypergradient exists and has the following closed-form expression by implicit function theorem~\citep{ghadimi2018approximation}:
\begin{align}
    \nabla \Phi(x) = \nabla_x f(x, y^*(x)) + \nabla y^*(x)^\top \nabla_y g(x, y^*(x)) \label{eq: hypergrad}
\end{align}
where we have
\begin{align}
    \nabla y^*(x)^\top = - \nabla_{xy}^2g(x,y^*(x))(\nabla_y^2g(x, y^*(x)))^{-1}. \label{eq: nabla_y_star}
\end{align}
Two major challenges are obvious from the hypergradient expression in \eqref{eq: hypergrad} -- one may not have direct access to $y^*(x)$ and it is usually expensive to directly invert a Hessian matrix $\nabla_y^2g(x, y^*(x))$, which may further require some approximation of the Hessian inverse. This suggests that one should carefully handle these two sources of large bias in estimating \eqref{eq: hypergrad}. State-of-the-art techniques to estimate \eqref{eq: hypergrad} include AID-based methods~\citep{domke2012generic, pedregosa2016hyperparameter, gould2016differentiating, ghadimi2018approximation, grazzi2020iteration, ji2021bilevel}, ITD-based methods~\citep{domke2012generic, maclaurin2015gradient, franceschi2018bilevel, grazzi2020iteration, ji2021bilevel}, Neumann-series-based methods~\citep{ghadimi2018approximation, chen2021closing,hong2023two,ji2021bilevel}, and SGD-based methods~\citep{arbel2022amortized, dagreou2022framework, chen2023optimal, hao2024bilevel}. Although the sample complexity of BO has been proven to match the lower bound under mild assumptions~\citep{chen2023optimal, hao2024bilevel}, it is worth noting all these works require Jacobian-vector product oracles, which largely restrict the applicability of such algorithms. {To mitigate this issue, a new paradigm of fully first-order penalty-based methods has been introduced, which reformulate the lower-level problem into the optimality constraint~\cite{liu2022bome,kwon2023fully, chen2023near, shen2023penalty}.}

To accelerate the optimization process of BO algorithms, there is a flurry of work extending the single-agent training setting to the multi-agent ones such as decentralized training~\citep{lu2022decentralized, chen2022decentralized, yang2022decentralized, gao2022stochastic, dong2023single, kong2024decentralized} and federated learning~\citep{tarzanagh2022fednest, huang2023achieving, yang2024simfbo}. {Decentralized training has the potential for efficient training of scalable foundation models and received a lot of attention in recent years~\cite{NEURIPS2022_a37d615b,gan2021bagua}.} Designing provably convergent and efficient algorithms for these types of problems is even harder, as we need to handle the heterogeneity from various sources of data and achieve consensus among different agents. Existing decentralized stochastic bilevel optimization (DSBO) algorithms mainly utilize second-order information to approximate the hypergradient, and then apply updates in a decentralized manner on top of it. This paper aims to propose and evaluate the fully first-order methods for DSBO problems. Our contributions can be summarized as follows.

\begin{table*}[t]
\label{table: compare}
\vskip 0.15in
\begin{small}
\caption{We compare our Algorithm \ref{alg:stochastic_minmax:onestage} with existing DSBO algorithms including DSBO-JHIP~\citep{chen2022decentralized}, GBDSBO~\citep{yang2022decentralized}, MA-DSBO~\citep{chen2022decentralizeddsbo}, and D-SOBA~\cite{kong2024decentralized}. ``Cost / Iter'' represents the per-iteration computational and communication cost. ``Complexity'' represents the oracle complexity as well as the communication rounds required to find an $\epsilon$-stationary point. ``Oracles'' represents the oracles needed in the algorithms. We use ``Jacobian'', ``JVP'', and ``Grad'' to denote oracles of Jacobian matrices, Jacobian-vector products, and gradients respectively. ``Heterogeneity'' corresponds to data heterogeneity, and ``Bounded'' indicates the requirement of an additional assumption that the data heterogeneity is bounded across agents, i.e., $\norm{\nabla f_i - \frac{1}{n}\sum_{i=1}^{n}\nabla f_i}$ is bounded uniformly for all $i$. In deep learning architectures, the computation of a Jacobian-vector product can take four times the time taken by computing a gradient and may require three times more memory than computing a gradient~\citep{dagréou2024howtocompute}.}
\end{small}
\centering
\begin{tabular}{ccccccc}
\toprule
\textbf{Algorithm}& \textbf{Cost / Iter}  & \textbf{Complexity} & \textbf{Oracles} & \textbf{Heterogeneity} \\ 
\midrule
\textbf{DSBO-JHIP}  &$\cO(d^2)$ & $\tilde\cO(\epsilon^{-6})$ & JVP, Grad & Bounded \\
\textbf{GBDSBO}   & $\cO(d^2)$ & $\tilde{\cO}(n^{-1}\epsilon^{-4})$ & Jacobian, Grad & Bounded \\ 
\textbf{MA-DSBO}   & $\cO(d)$ & $\tilde{\cO}(\epsilon^{-4}) $ & JVP, Grad & Bounded \\
\textbf{D-SOBA}  & $\cO(d)$ & ${\cO}(n^{-1}\epsilon^{-4}) $ & JVP, Grad & Bounded\\
\midrule
\textbf{DSGDA-GT}  & $\cO(d)$ & $\cO(n^{-1}\epsilon^{-7}) $ & Grad & Unbounded \\
\bottomrule
\end{tabular}
\end{table*}


%
\subsection{Our contributions}

\begin{itemize}[leftmargin=*]

    \item We propose Decentralized Stochastic Gradient Descent and Ascent with Gradient Tracking (DSGDA-GT), a fully first-order algorithm for solving the DSBO problem with a constant batch size and unbounded data heterogeneity {(i.e., $\norm{\nabla f_i - \frac{1}{n}\sum_{i=1}^{n}\nabla f_i}$ is unbounded across agents)}. Our algorithm greatly improves the per-iteration time and space complexity compared to existing works, which heavily depend on second-order information of the objectives.

    \item We provide a finite-time analysis, which indicates that our algorithm is capable of finding an $\epsilon$-stationary point within $\cO(n^{-1}\epsilon^{-7})$ first-order oracle complexity, which matches the current best-known result in the single-agent counterpart and achieves a linear speedup effect in the decentralized setting. In addition, our analysis of the double-loop and two-timescale decentralized optimization is of independent interest.
    
    \item We conduct experiments on both synthetic and real-world datasets, comparing the performance of our algorithm against existing state-of-the-art baselines. The empirical results demonstrate that our methods exhibit superior generalization performance and greater efficiency compared to the others.
\end{itemize}

\subsection{Related work}

{\bf Bilevel optimization.}\quad The study of bilevel optimization can be traced back to \cite{stackelberg1952theory}. Recently, there is a flurry of work proposing novel BO algorithms with provable convergence rates~\citep{ghadimi2018approximation, grazzi2020iteration, hong2023two, chen2021closing, dagreou2022framework} and implementing BO in large-scale problems in the machine learning community~\citep{pedregosa2016hyperparameter, lorraine2020optimizing}. It is gaining popularity due to its capability to handle different types of problems with a hierarchical structure. One line of theoretical work aims at settling the sample complexity of finding a stationary point in BO~\citep{ghadimi2018approximation, hong2023two, ji2020convergence, chen2021closing, arbel2022amortized, dagreou2022framework, chen2023optimal, hao2024bilevel} when second-order oracles like Jacobian-vector products are accessible. Despite the fact that the complexity of computing a matrix-vector product oracle is roughly the same as that of a gradient~\citep{pearlmutter1994fast}, such oracles are still time-consuming and difficult to implement, especially when it comes to neural network models, which require additional efforts in developing machine learning libraries to efficiently compute the hypergradient~\citep{grefenstette2019generalized, deleu2019torchmeta, arnold2020learn2learn}. {Motivated by this, some recent works propose novel algorithms to avoid accessing second-order information of the problem, such as fully first-order method~\cite{liu2022bome,kwon2023fully,shen2023penalty, chen2023near}}, which reformulates the bilevel problem as a single-level one treating the lower-level problem as a penalty term, zeroth-order method~\citep{Sow_2022_zeroth,yang2024achieving, aghasi2024fully}, which estimates the hypergradient via finite-difference approximation, etc.

{\bf Decentralized optimization.}\quad Decentralized optimization has been studied extensively in control community~\citep{xu2015augmented, di2016next}. When it comes to large-scale machine learning problems, the decentralized training was revealed to have its own advantages in terms of privacy protection, robustness, scalability, and linear speedup effect~\citep{lian2017can, tang2018d}. Theoretical investigations include analyzing the sample complexity~\citep{lian2017can, tang2018d}, effects of network topology~\citep{neglia2020decentralized}, compression techniques~\citep{tang2018communication, koloskova2019decentralized}, etc.

Decentralized stochastic bilevel optimization (DSBO) arises naturally when the data of a bilevel problem is distributed among different agents connected by a communication network. Extending BO from single-agent training to distributed training is non-trivial, as the hypergradient estimation involves Hessian inverse estimation, which requires the information of each local function pair $(f_i, g_i)$. Some efforts are trying to overcome this obstacle in the distributed setting, for example, decentralized setting~\citep{lu2022decentralized, chen2022decentralized, yang2022decentralized, gao2022stochastic, dong2023single, huang2024distributed,kong2024decentralized} and federated learning setting~\citep{tarzanagh2022fednest, huang2023achieving, yang2024simfbo}. However, all these works require access to matrix-vector products, i.e., second-order information, that are sometimes unavailable. {The recent parallel work~\cite{niu2024distributed} proposes a penalty-based distributed bilevel method in the deterministic setting and focuses on the bounded heterogeneity assumption. }

\section{Preliminaries}

{\bf Problem setup.}\quad In decentralized stochastic bilevel optimization (DSBO), we aim to solve the BO problem via multiple agents or devices in a distributed manner. Specifically, there are $n$ different agents communicating over a decentralized network, which can be represented by a graph whose vertices denote local agents and each edge indicates the neighboring relationship between end points of it. The formal description of the DSBO problem is 
\begin{align}
 \min_{x \in \R^p} & \quad \Phi(x) = \avein f_i(x, y^{\ast}(x)) \quad
 \text{s.t.}\, \, y^{\ast}(x) = \arg\min_{y} \avein g_i(x,y)   \label{opt: dsbo}
\end{align}
where the lower and upper functions $f_i(x,y) = \E_{\xi \sim \Xi_i}[F(x, y;\xi)]$ and $g_i(x,y)  = \E_{\psi \sim \Psi_i}[G(x, y;\psi)]$ are only accessible to the agent $i$. We assume that each agent only has access to stochastic gradient oracles of local functions $(f_i, g_i)$, and they can only communicate with their neighbors to exchange information so that they can collaboratively solve the problem. It is worth noting that according to the hypergradient expression in \eqref{eq: hypergrad} and \eqref{eq: nabla_y_star}, we can obtain
\begin{align*}
    \nabla \Phi(x) &= \left(\frac{1}{n}\sum_{i=1}^{n}\nabla_x f_i(x, y^*(x))\right) + \nabla y^*(x)^\top \left(\frac{1}{n}\sum_{i=1}^{n}\nabla_y f_i(x, y^*(x))\right) \\
    \nabla y^*(x)^\top &= -\left(\frac{1}{n}\sum_{i=1}^{n}\nabla_{xy}^2g_i(x,y^*(x))\right)\left(\frac{1}{n}\sum_{i=1}^{n}\nabla_y^2g_i(x, y^*(x))\right)^{-1}.
\end{align*}
We can clearly see that the main challenge of solving DSBO problems lies in estimating $\nabla y^*(x)^\top$, and there have been some efforts along this line~\citep{chen2022decentralized, chen2022decentralizeddsbo, yang2022decentralized, kong2024decentralized}. They all require access to Jacobian-vector products, which are not available in our setting. 

{\bf Notation.}\quad 
For convenience, we first introduce our notation conventions. $\bfone_n$ denotes the all-one vector in $\R^n$. $\norm{\cdot}$ represents $\ell^2$-norm for vectors and Frobenius norm for matrices. $\norm{\cdot}_2$ denotes the spectral norm for matrices. We use bar notation over a variable to represent the average of the variables of all agents. We use $\cO$ and $\Theta$ to denote big-O and big-Theta notation, i.e.,
\begin{align*}
    &f(x) = \cO(g(x)), \text{when } |f(x)| \leq C |g(x)| \text{ for some constant $C$ independent of $f, g$,} \\
    &f(x) = \Theta(g(x)), \text{when } C_1|g(x)| \leq |f(x)|\leq C_2 |g(x)| \text{ for some constants $C_1, C_2$ independent of $f, g$}.
\end{align*}
The notion of stationarity in this paper is defined as follows.
\begin{definition}
    Suppose we are given the output sequence $\{\bar x_1, \bar x_2, ..., \bar x_S\}$ of an algorithm for Problem \eqref{opt: dsbo}. We say it finds an $\epsilon$-stationary point, when 
    \[
        \min_{1\leq s\leq S}\E\left[\norm{\nabla \Phi(\bar x_s)}\right]\leq \epsilon.
    \]
\end{definition}

\subsection{Fully first-order hypergradient estimation}

To effectively approximate $(\nabla_y^2 g)^{-1}\nabla_y f$ in the expression of the hypergradient in \eqref{eq: hypergrad}, classical stochastic algorithms either require Neumann series methods~\citep{ghadimi2018approximation, hong2023two, chen2021closing}, or approximating the solution of a linear system via minimizing a quadratic function~\citep{arbel2022amortized, dagreou2022framework, chen2023optimal, hao2024bilevel}. All of them require Hessian-vector products. To avoid the computation of second-order information, we consider the following min-max formulation shown in~\citep{kwon2023fully, chen2023near} to design a fully first-order method for DBSO. 

{\bf Min-max reformulation.}\quad Note that in \eqref{opt: dsbo} the lower-level can be viewed as a constraint of the upper-level problem, and thus it is tempting to reformulate the DSBO problem as: 
\begin{align}\label{DSBO:bi-level_proxy}
    \min_{x \in \R^p, y\in \R^q} \,\,  &  \frac{1}{n}\sum_{i=1}^n f_i(x,y), \quad 
    \text{s.t.}  \,\,   \frac{1}{n}\sum_{i=1}^n  g_i(x, y) - \min_{z} \frac{1}{n}\sum_{i=1}^n g_i(x, z) = 0.   
\end{align}
In this formulation, we introduce an auxiliary variable $z$ to transform the lower problem $y^{\ast}(x) = \arg\min_{y} \frac{1}{n}\sum_{i=1}^ng_i(x, y)$ into the constraint $\frac{1}{n}\sum_{i=1}^ng_i(x, y) - \min_z \frac{1}{n}\sum_{i=1}^ng_i(x,z) = 0$, where $y$ serves as a proxy of $y^{\ast}(x)$. By adding the constraint in~\eqref{DSBO:bi-level_proxy} as a penalty term with a factor $\alpha$ to the upper-level function, the DSBO problem can be reformulated as follows:
\begin{align}\label{DBP:min:max}
    \min_{x \in \R^p, y\in \R^q}\max_{z}  \mathcal{L}^{\alpha}(x, y, z)
\end{align}
where
\begin{align}
\mathcal{L}^{\alpha}(x, y, z):= \frac{1}{n}\sum_{i=1}^n (f_i(x,y) + \alpha(g_i(x,y) - g_i(x,z)))
\end{align}
and $z \in \R^q$ is the lower variable whose optimum value is still $y^{\ast}(x)$, while $y \in \R^q$, $\alpha > 0$ is the multiplier. In this way, the approximation of both lower constraint and upper
optimum can be obtained during the same optimization process, and $\alpha$ controls the priority.  

{\bf Equivalence between Problems \eqref{opt: dsbo} and ~\eqref{DBP:min:max}.}\quad We overload the notation in~\eqref{DBP:min:max} and define
\begin{align}
     \Omega^{\alpha}(x,y) = \max_z \cL^{\alpha}(x,y,z), & \quad   z_*(x) := \argmax_z \cL^{\alpha}(x,y,z) = \argmin_z \frac{1}{n}\sum_{i=1}^ng_i(x,z), \notag \\
     \Gamma^{\alpha}(x) = \min_y \Omega^{\alpha}(x,y), &  \quad y^{\alpha}_*(x) := \argmin_y \Omega^{\alpha}(x,y). \notag 
\end{align}
Note that solving for $z$ does not require $\alpha$ to be present in the problem. The max part is essentially $\min_z g(x,z)$. The optimality metric of  Problem~\eqref{DBP:min:max} is defined as
\begin{align}\label{eq: minmax_metric}
    \norm{\nabla \Gamma^{\alpha}(x)} \leq \epsilon,
\end{align}
which is commonly used in non-convex strongly-concave (NCSC) min-max optimization \citep{lin2020gradient}. Moreover, we have the following Lemma~\ref{lem: gap_grad_bil_minmax} characterizing the relationship between the optimality of the min-max problem defined above and the first-order stationarity of problem \eqref{opt: dsbo}. We omit the proof and the details can be found in lemma 4.1 of \citep{chen2023near}).
\begin{lemma}\label{lem: gap_grad_bil_minmax}
Under Assumption~\ref{aspt: smoothness}, if $\alpha \geq 2\ell_{f,1}/\mu_g$, then
\begin{align}\label{eq:grad_bilevel_minimax}
 \text{(a.)}\, \left\| \nabla \Phi(x) - \nabla \Gamma^{\alpha}(x) \right\| \leq \mathcal{O}\left(\frac{\kappa^3}{\alpha}\right); \quad \text{(b.)}\, \left\|\nabla^2 \Gamma^{\alpha}(x) \right\| \leq \mathcal{O}(\kappa^3)
\end{align}
where $\kappa, \mu_g, \ell_{f,1}$ are defined in Section~\ref{sec:theoretical:result}.
\end{lemma}
 Lemma~\ref{lem: gap_grad_bil_minmax} (a.) implies that when $\alpha \sim 1/\epsilon$, the stationary point of Problem~\eqref{DBP:min:max} is also a stationary point of Problem~\eqref{opt: dsbo}. Note that Lemma~\ref{lem: gap_grad_bil_minmax} (b.) clarifies that the gradient Lipschitz constant of $\Gamma^{\alpha}(x)$ does not depend on the multiplier $\alpha$ when $\alpha$ is larger than a certain threshold.



\section{Algorithm}
In this section, we introduce the main ingredients of our algorithmic framework.

\subsection{Decentralized optimization with gradient tracking}
In decentralized optimization, the gradient tracking (GT) technique was proposed to improve the convergence rates of decentralized optimization algorithms~\citep{xu2015augmented, di2016next, nedic2017achieving, qu2017harnessing}. It was later shown, under mild assumptions, to have unique advantages in handling unbounded gradient similarity caused by data heterogeneity~\citep{zhang2019decentralized, lu2019gnsd, pu2021distributed, koloskova2021improved}. Thus, we will incorporate this technique into our algorithms to mitigate the data heterogeneity effect. It is worth noting that the implementation of Algorithm~\ref{alg:inner_loop} has one communication round in each iteration, and one can also adopt multi-consensus techniques such as FastMix~\citep{ye2023multi} and Chebyshev-type communication~\citep{song2023optimal} to enhance consensus among agents.

\subsection{Proposed algorithm}
To solve the equivalent decentralized min-max problem~\eqref{DBP:min:max}, we are ready to present our main Algorithm \ref{alg:stochastic_minmax:onestage} named decentralized stochastic gradient descent ascent with gradient tracking (DSGDA-GT). It adopts a double-loop structure widely used in bilevel optimization literature~\citep{ghadimi2018approximation, ji2020convergence, chen2021closing}.

We first perform the $T$-step inner-loop decentralized training with gradient tracking (in Algorithm~\ref{alg:inner_loop}) to update lower variables $y, z$. As shown in line 6 of Algorithm~\ref{alg:inner_loop}, we use $u_{t+1}^{(i)}$ to track the stochastic gradients of the local agent $i$, which provably achieves linear speedup without assuming data similarity assumption~\citep{pu2021distributed, koloskova2021improved}. Since the inner variables $y, z$ are independent of each other, the two $T$-step inner-loop updates can be performed synchronously.  In the inner-loop subroutines: when setting $T=1$, Algorithm \ref{alg:stochastic_minmax:onestage} immediately becomes a single-loop algorithm, while choosing large $T$ could potentially bring better convergence rates~\citep{ji2022will, chen2023near,kwon2024complexity}. Thus, this seemingly complex framework offers more flexibility than the single-loop counterpart.

In each outer iteration (indexed by $s$), we run stochastic gradient descent with gradient tracking specifically for the upper variable $x$. The gradient track update for agent $i$ is obtained in line 8 of Algorithm~\ref{alg:stochastic_minmax:onestage} utilizing additional variable set $v_{s+1}^{(i)}$. Note that Algorithm \ref{alg:stochastic_minmax:onestage} may involve unequal
stepsizes for $x$, $y$, and $z$ to accommodate their distinct objectives, as dictated by their theoretical properties. 

\begin{algorithm}
\caption{Decentralized stochastic gradient descent ascent with gradient tracking (DSGDA-GT)}
\label{alg:stochastic_minmax:onestage}
\begin{algorithmic}[1]
\STATE {\bfseries Input:}  $x_0, y_0, z_0, \alpha, \eta_{x}, \eta_{y}, \eta_{z}, S, T$.
\STATE {\bfseries Initialization:} $x_0^{(i)} = x_0, y_0^{(i)}=y_0, z_0^{(i)}=z_0, v_0^{(i)} = \delta_0^{(i)} = 0$ on node $i$.
\FOR{$s = 0:S-1$}
\FOR{$i=1:n$}
\STATE{$y_{s+1}^{(i)}, u_{s+1, y}^{(i)}, h_{s+1, y}^{(i)} = $ \texttt{Inner Loop}$(y_s^{(i)}, \eta_{y}, f_i(x_s^{(i)}, \cdot) + \alpha g_i(x_s^{(i)}, \cdot), u_{s, y}^{(i)}, h_{s, y}^{(i)}, T)$}
\STATE{$z_{s+1}^{(i)}, u_{s+1, z}^{(i)}, h_{s+1, z}^{(i)} = $ \texttt{Inner Loop}$(z_s^{(i)}, \eta_{z}, g_i(x_s^{(i)}, \cdot), u_{s, z}^{(i)}, h_{s, z}^{(i)}, T)$}
\STATE{$\delta_{s+1}^{(i)} =\nabla_{x} f_i(x_{s}^{(i)}, y_{s}^{(i)}; \xi_s^{(i)}) + \alpha \left(\nabla_{x} g_i(x_{s}^{(i)}, y_{s}^{(i)};\psi_s^{(i)}) - \nabla_{x} g_i(x_{s}^{(i)}, z_{s}^{(i)};\psi_s^{(i)})\right)$}
\STATE{$v_{s+1}^{(i)} = \sum_{j=1}^{n}w_{ij} v_{s}^{(j)} + \delta_{s+1}^{(i)} - \delta_{s}^{(i)}$}
\STATE{$x_{s+1}^{(i)} = \sum_{j=1}^{n}w_{ij}x_s^{(i)} - \eta_{x} v_{s+1}^{(i)}$}
\ENDFOR
\ENDFOR
\STATE{{\bf Output:} $x_{S}^{(i)}, y_{S}^{(i)}, z_{S}^{(i)}$ on each node.}
\end{algorithmic}
\end{algorithm}

\begin{algorithm}
\caption{\texttt{Inner Loop}\,($\theta_0, \gamma, \phi_i(x, \theta), u_0, h_0, T$) }
\label{alg:inner_loop}
\begin{algorithmic}[1]
\STATE {\bfseries Input: $\theta_0, \gamma, \phi_i(x, \theta), u_0, h_0, T$.}
\STATE {\bfseries Initialization: $u_0^{(i)}, h_0^{(i)}$ on node $i$ satisfying $\bar u_0 = \bar h_0$.}
\FOR{$t = 0:T-1$}
\FOR{$i=1:n$}
\STATE{$h_{t+1}^{(i)} = \nabla \phi_i(x^{(i)}, \theta_t^{(i)};\zeta_t^{(i)})$}
\STATE{$u_{t+1}^{(i)} = \sum_{j=1}^{n}w_{ij}u_t^{(i)} + h_{t+1}^{(i)} - h_t^{(i)}$}
\STATE{$\theta_{t+1}^{(i)} = \sum_{j=1}^{n}w_{ij}\theta_t^{(i)} - \gamma u_{t+1}^{(i)}$}
\ENDFOR
\ENDFOR
\STATE{{\bf Output:} $\theta_T^{(i)}, u_T^{(i)}, h_{T+1}^{(i)}$ on each node.}
\end{algorithmic}
\end{algorithm}

\section{Theoretical results}
\label{sec:theoretical:result}
In this section, we provide a convergence analysis of our algorithms. We first introduce the following assumptions, which are standard in both bilevel and distributed optimization literature, as follows.

\begin{assumption}(Smoothness)\label{aspt: smoothness}
The objectives $f_i$ and $g_i$ for each agent $i$ satisfy:
\begin{itemize}
    \item[(1)] The UL objective $f_i(x, y)$ is 
    $\ell_{f,0}$-Lipschitz continuous in $y$; $\ell_{f,1}$-gradient Lipschitz, and $\ell_{f,2}$-Hessian Lipschitz.
    \item[(2)] The LL objective $g_i(x, y)$ is $\ell_{g,1}$-gradient Lipschitz, $\ell_{g,2}$-Hessian Lipschitz, and $\mu_g$-strongly convex in $y$.  
\end{itemize}
\end{assumption}
 In this paper, we consider the well-conditioned bilevel problem which is sufficient under Assumption~\ref{aspt: smoothness}(2)~\citep{ghadimi2018approximation}. Here we define the condition number $\kappa = \max \left\lbrace \ell_{f,0}, \ell_{f,1},  \ell_{g,1}, \ell_{g,2}\right\rbrace / \mu_g$ which aligns with Definition 3.1 in~\citep{chen2023near}. 

Under Assumption~\ref{aspt: smoothness}, $f_i + \alpha g_i$ is $\mu_g\alpha/2$-strongly convex in $y$ if $\alpha \geq 2\ell_{f,1}/\mu_{g}$. The technical lemmas for functions $\cL^{\alpha}(x,y,z)$ and $\Gamma^{\alpha}(x)$ and their optimal functions $z_*(x)$ and $y_*^{\alpha}(x)$ in the nonconvex-(strongly-convex)-(strongly-concave) min-max setting can be found in Appendix~\ref{sec:appendix:functions}.

\begin{assumption}(Bounded variance)\label{aspt: so}
    Denote by $\cF_s$ the $\sigma$-algebra generated by all iterates with subscripts up to $s$. All stochastic oracles are unbiased with bounded variance. The stochastic oracles of iterates with subscript $s$ are independent under $\cF_s$.
\end{assumption}

\begin{remark}
The assumptions for objectives $f_i, g_i$ are similar to those of Theorem 4.1 in \citep{kwon2023fully}, except for the boundedness requirement on $\nabla g_i$ as stated in \citep{kwon2023fully}. In comparison to the assumptions made in \citep{chen2023near}, the Hessian Lipschitz condition of $f_i$ is required to ensure the smoothness of $y_*^{\alpha}(x)$ (see Lemma~\ref{lem:appendix:zy} in Appendix), which is necessary for the consensus analysis of $Y$ when the inner-loop step $T=1$. It is worth noting that this higher-order smoothness assumption in $f_i$ can be further relaxed by incorporating the moving-average technique used in~\citep{chen2023optimal, kong2024decentralized}.
\end{remark}

\begin{assumption}(Network topology)\label{aspt: sgap}
    $\mW = (w_{ij})\in\R^{n\times n}$ is symmetric and doubly stochastic, and its eigenvalues $\lambda_n\leq ...\leq \lambda_1 = 1$ satisfy $\rho:= \max\{|\lambda_2|, |\lambda_n|\} < 1$.
\end{assumption}
{Assumption~\ref{aspt: sgap} is a standard condition in the distributed optimization~\cite{lian2017can,chen2022dsbo}, as it characterizes the network topology. It is readily satisfied in certain fully connected networks, such as ring networks and perfect graphs.}
\begin{assumption}\label{aspt: bdd_grad}
    There exists a constant $c_{\delta}$ such that in Algorithm \ref{alg:stochastic_minmax:onestage} we have 
    \begin{align*}
        \E\left[\norm{\bar \delta_{s+1}}^2|\cF_s\right]\leq c_{\delta}\alpha^2.
    \end{align*}
     {where $\bar \delta_{s+1}$ is the average of $\delta_{s+1}^{(i)}$ (see Algorithm \ref{alg:stochastic_minmax:onestage}) across the agents.}
\end{assumption}
{Assumption~\ref{aspt: bdd_grad} is independent of the gradient heterogeneity assumption used in~\cite{lian2017can,chen2022dsbo}. } A similar assumption is also used in bilevel optimization literature (see Assumption 3.7 in~\citep{dagreou2022framework}). Note that Assumption \ref{aspt: bdd_grad} holds provided that Assumptions \ref{aspt: smoothness} and \ref{aspt: so} hold and $\norm{\nabla_x g_i(x,y)}$ is bounded since
\begin{align*}
    \E\left[\norm{\bar \delta_{s+1}}^2|\cF_s\right]= \norm{\E\left[\bar \delta_{s+1}|\cF_s\right]}^2 + \E\left[\norm{\bar \delta_{s+1} - \E\left[\bar \delta_{s+1}|\cF_s\right]}^2|\cF_s\right]
\end{align*}
which is of order $\cO(\alpha^2)$. 

Now we are ready to present the convergence results of our algorithms.
\begin{restatable}[]{theorem}{thmconverge} 
\label{thm: warm_start_rate}
    Suppose Assumptions \ref{aspt: smoothness}, \ref{aspt: so}, \ref{aspt: sgap}, and \ref{aspt: bdd_grad} hold, and parameters $\alpha$ and step sizes are chosen such that 
    \begin{align*}
\alpha = \Theta\left((1-\rho^2)^{1/2}(nS)^{1/7} \right), \eta_x = \Theta \left( \frac{(1-\rho^2)n^{2/7}}{S^{5/7}}\right), \eta_y = \Theta \left( \frac{(1-\rho^2)n^{2/7}}{S^{5/7}}\right), \eta_z = \Theta \left(\frac{(1-\rho^2)^{3/2}n^{3/7}}{S^{4/7}} \right)
    \end{align*}
    and further assume a warm-start for variables $y, z$ such that 
    \begin{align}\label{eq: warmup_yz}
\max\left(\norm{\bar y_0 - y_{*, 0}^{\alpha}}^2, \norm{\bar z_0 - z_{*, 0}}^2\right) = \cO\left(1/\alpha\right)    
    \end{align}
    Consider Algorithm~\ref{alg:stochastic_minmax:onestage} with $T=1$ and $S \geq n^{4/3}$, we have 
    \begin{align*}
     \min_{0\leq s \leq S-1} \E\left[\norm{\nabla \Phi(\bar x_s)}\right] &  \leq \cO\left(\frac{(1-\rho^2)^{-1/2}}{(nS)^{1/7}}\right),\ \min_{0 \leq s \leq S-1} \frac{\E\left[ \norm{\mX_s - \bar{x}_s \bfone_n}\right]}{n}\leq \cO\left(\frac{n^{-1/14}}{S^{4/7}} + \frac{n^{1/7}(1-\rho^2)^{-3/2}}{S^{6/7}}\right).
    \end{align*}
\end{restatable}
{Theorem \ref{thm: warm_start_rate} explicitly characterizes the relationship between the convergence and the network connectivity parameter $\rho$. As $\rho$ increases, the convergence rate gets worse.} As a byproduct of Theorem \ref{thm: warm_start_rate}, we have the following Corollary that gives the sample complexity of finding an $\epsilon$-stationary point.
\begin{corollary}
    Under the same conditions of Theorem \ref{thm: warm_start_rate}, the stochastic first-order oracles needed in Algorithm \ref{alg:stochastic_minmax:onestage} for finding an $\epsilon$-stationary point is $\cO(n^{-1}\epsilon^{-7})$.
\end{corollary}
We highlight that the warm-start condition \eqref{eq: warmup_yz} can be satisfied via running Algorithm \ref{alg:inner_loop} as another subroutine. Note that the sample complexity (per node) of achieving \eqref{eq: warmup_yz} is $\cO(n^{-1}\alpha) = \cO(n^{-6/7}S^{1/7} )$ according to Lemma \ref{lem: inner_error}, and for $S = \cO(n^{-1}\epsilon^{-7})$ we know this requires $\cO(n^{-1}\epsilon^{-1})$ additional stochastic oracles, which do not affect the final sample complexity. Note that we also obtain the linear speedup effect in the sample complexity bound, i.e., the samples required on each node is $\cO(n^{-1}\epsilon^{-7})$.

\begin{remark}
When considering $\cO(1)$ batch size setting, if we set $n=1$, which represents the single-agent training scenario, then the sample complexity of finding an $\epsilon$-stationary point of Algorithm \ref{alg:stochastic_minmax:onestage} matches that of \cite{kwon2023fully}. It is worth noting that the large-batch and inner-loop $T \gg 1$ settings can also be covered by our analysis, however, it does not yield the desired improvement by a simple extension of \citep{chen2023near} and \citep{kwon2024complexity} due to the consensus error in the upper variable $x$. With stronger assumptions such as mean-squared smoothness~\citep{kwon2024complexity, yang2024achieving} and large batch sizes~\citep{chen2023near} imposed, we anticipate the sample complexity can be further improved, and we leave this as an interesting future work.
\end{remark}

\subsection{Proof sketch}
In this section, we highlight the main steps of analyzing the proposed algorithms and the novelty of our analysis as compared to the existing ones. 

By the smoothness of $\Gamma^{\alpha}(x)$ in Lemma~\ref{lem: gap_grad_bil_minmax}, we first get the descent inequality over the variable $x$:
\begin{align*}
& \E\left[\Gamma^{\alpha}(\bar x_{s+1})\middle|\cF_s\right]-\Gamma^{\alpha}(\bar x_s) \notag \\
        \leq &  -\frac{\eta_{x}}{2}\norm{\nabla \Gamma^{\alpha}(\bar x_s)}^2 - \left(\frac{\eta_{x}}{2} - \frac{\eta_{x}^2\ell_{\Gamma}}{2}\right)\norm{\E\left[\bar v_{s+1}\middle|\cF_s\right]}^2 + \frac{\ell_{\Gamma}\eta_{x}^2\sigma_x^2}{2n} \notag \\
       + &\underbrace{\frac{3\eta_x\ell_{x,1}^2}{2n}\norm{\mX_s - \bar x_s\bfonet_n}^2}_{\text{outer-loop error}} + \underbrace{\frac{3\eta_x\ell_{y,1}^2}{2n}\norm{\mY_s - {y_{\ast}^{\alpha}}(\bar x_s)\bfonet_n}^2 +  \frac{3\eta_x\alpha^2\ell_{z,1}^2}{2n}\norm{\mZ_s - z_{\ast}(\bar x_s)\bfonet_n}^2}_{\text{inner-loop error}}.
    \end{align*}
This, together with \eqref{lem: gap_grad_bil_minmax}, indicates that to theoretically bound $\norm{\nabla \Phi(\bar x_s)}$, we need to carefully estimate the error induced by the inner-loop variables $y, z$ and the outer-loop variable $x$.

\textbf{Inner-loop error.}\quad Take $y$ for example, motivated by the decomposition
\begin{align*}
\norm{\mY_s - {y_{\ast}^{\alpha}}(\bar x_s)\bfonet_n}^2 \leq \underbrace{\left\| \mY_s - \bar{y}_s \bfonet_n\right\|^2}_{\text{Consensus error}} + \underbrace{n\left\|\bar{y}_s -y_{\ast}^{\alpha}(\bar x_s)\right\|^2}_{\text{Convergence error}},
\end{align*}
we separately analyze the consensus and convergence of inner variables $y, z$ in Section \ref{sec:convergence_app}. { The corresponding inner-loop errors are also provided in Lemma \ref{lem: yz_convergence}.}

\textbf{Outer-loop error.}\quad Note that due to the double-loop and two-timescale nature of our algorithm, the analysis of the inner-loop error, which gives a recursive relation between $\norm{\mY_{s+1} - {y_{\ast}^{\alpha}}(\bar x_{s+1})\bfonet_n}$ and $\norm{\mY_s - {y_{\ast}^{\alpha}}(\bar x_s)\bfonet_n}$ (see Lemma \ref{lem: inner_error}, same for $z$), cannot be directly incorporated into the outer-loop analysis. We provide a novel analysis to balance these two sources of error in Section \ref{sec:consensus_app}.

We highlight that different from classical analysis of decentralized stochastic gradient tracking techniques for optimizing strongly convex functions~\citep{pu2021distributed} which only requires all stepsizes to have the same order of magnitude in terms of $S$ (i.e., single-timescale), our convergence analysis requires careful design of stepsize choices for $\eta_x, \eta_y, \eta_z$ to handle the consensus error and convergence error induced by both the inner and outer loops. Different from the existing analysis of double-loop DSBO algorithm~\cite{chen2022decentralizeddsbo}, we provide a fine-grained analysis in Section \ref{sec:consensus_app} that is of independent interest.

\section{Experiments}\label{sec:exp}

In this section, we investigate the empirical performance of Algorithm \ref{alg:stochastic_minmax:onestage}. Following the basic experimental setup in existing works~\citep{pedregosa2016hyperparameter, grazzi2020iteration, ji2020convergence, chen2022decentralizeddsbo, kong2024decentralized}, we consider the following hyperparameter optimization problem under the decentralized setting.
\begin{align}\label{opt: ho_exp}
    \min_{\lambda\in\R^p}\  \frac{1}{n}\sum_{i=1}^{n}\cL_{\text{val}}^{(i)}(\lambda,\omega^*(\lambda)),\quad \text{s.t.}\ \  \omega^*(\lambda) =\argmin_{w\in\R^q} \frac{1}{n}\sum_{i=1}^{n}\cL_{\text{train}}^{(i)}(\lambda,\omega).
\end{align}
Here, agent $i$ has access to validation dataset $\cD_{\text{val}}^{(i)}$ and the training dataset $\cD_{\text{train}}^{(i)}$, that are used to evaluate $\cL_{\text{val}}$ and $\cL_{\text{train}}^{(i)}$ respectively. We aim at learning the best hyperparameters $\lambda$, under the constraint that the model parameters $\omega$ are optimal. All experiments are conducted on a computer with Intel Core i7-11370H Processor. We use 8 cores to simulate 8 agents ($n=8$), and the communication steps are conducted with mpi4py~\citep{dalcin2021mpi4py} module. We compare our Algorithm \ref{alg:stochastic_minmax:onestage} with MA-DSBO~\citep{chen2022decentralizeddsbo} and D-SOBA~\citep{kong2024decentralized}, two DSBO algorithms that only require first-order oracles and matrix-vector product oracles. We note that both DSBO-JHIP~\citep{chen2022decentralized} and Gossip-DSBO~\citep{yang2022decentralized} require computing and communicating Jacobian matrices, and are inefficient \citep{chen2022decentralizeddsbo} as reported by \cite{chen2022decentralizeddsbo}. Hence we do not include them as baseline algorithms.

We would like to highlight that for hyperparameter optimization problems, the validation datasets that produce the upper-level functions $f_i$ are relatively much smaller than the training datasets for the lower-level functions $g_i$. It is thus more reasonable to update the hyperparameters less frequently than the model parameters, which indicates that our double-loop DSBO Algorithm \ref{alg:stochastic_minmax:onestage} offers more flexibility in this type of problem than single-loop ones.

\subsection{Synthetic data}
To validate the efficiency of Algorithm \ref{alg:stochastic_minmax:onestage}, we first consider a simple binary 
classification problem with synthetic data. Specifically, we consider problem~\eqref{opt: ho_exp}, with functions $(\cL_{\text{val}}^{(i)}, \cL_{\text{train}}^{(i)})$ as follows.
\begin{align*}
    \cL_{\text{val}}^{(i)}(\lambda, \omega) &= \frac{1}{|\cD_{\text{val}}^{(i)}|}\sum_{(x_e,y_e)\in\cD_{\text{val}}^{(i)}}\psi(y_ex_e^{\top}\omega),\\
    \cL_{\text{train}}^{(i)}(\lambda, \omega) &= \frac{1}{|\cD_{\text{train}}^{(i)}|}\sum_{(x_e,y_e)\in\cD_{\text{train}}^{(i)}}\psi(y_ex_e^{\top}\omega) + \frac{1}{2}\sum_{i=1}^{d}e^{\lambda_i}\omega_i^2,
\end{align*}
where $\psi(x) = \log(1 + e^{-x})$. We have $x_e\sim \mathcal N(0, i^2I_d)$ and $y_e = \text{sgn}(x_e^\top\omega + 0.1\cdot z)$, where $\text{sgn}(\cdot)$ is the sign function that outputs $1$ for a positive input and $0$ otherwise. $z$ is the noise vector generated from standard normal distribution. This gives a regularized logistic regression problem, which is widely used in bilevel optimization literature~\citep{pedregosa2016hyperparameter, grazzi2020iteration}. We plot the training loss and test accuracy over wall-clock time in Figures~\ref{figure:comp_loss} and~\ref{figure:comp_time}, from which we can observe that our methods achieve the lowest training loss and best accuracy in a relatively short amount of time. Interestingly, when all curves stabilize, the test accuracy of our Algorithm is better than the ones that require second-order information. This may indicate fully first-order methods have better generalization performance than second-order ones.
\begin{figure*}[ht]
	\centering
        \subfigure[]{\label{figure:comp_loss}\includegraphics[width=0.49\textwidth]{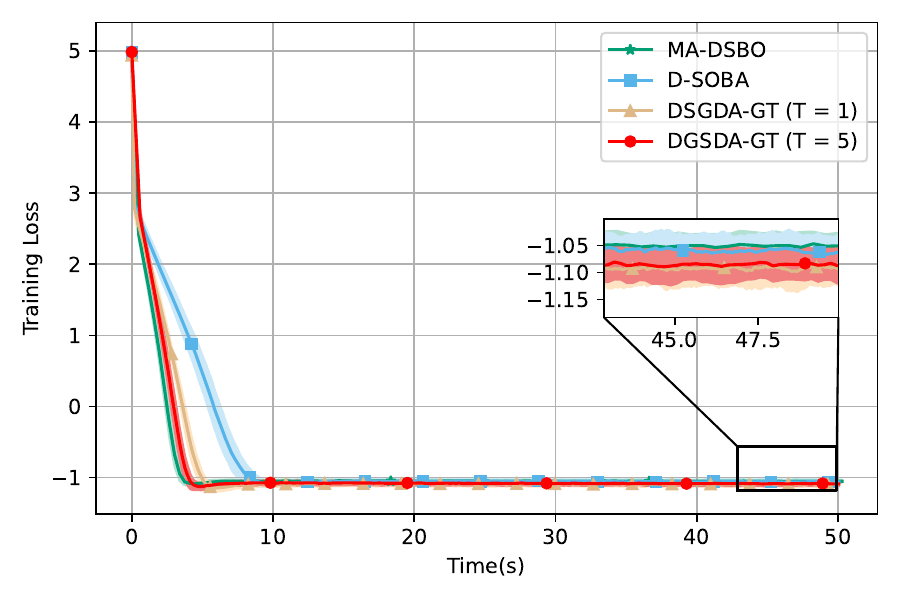}}
        \subfigure[]{\label{figure:comp_time}\includegraphics[width=0.49\textwidth]{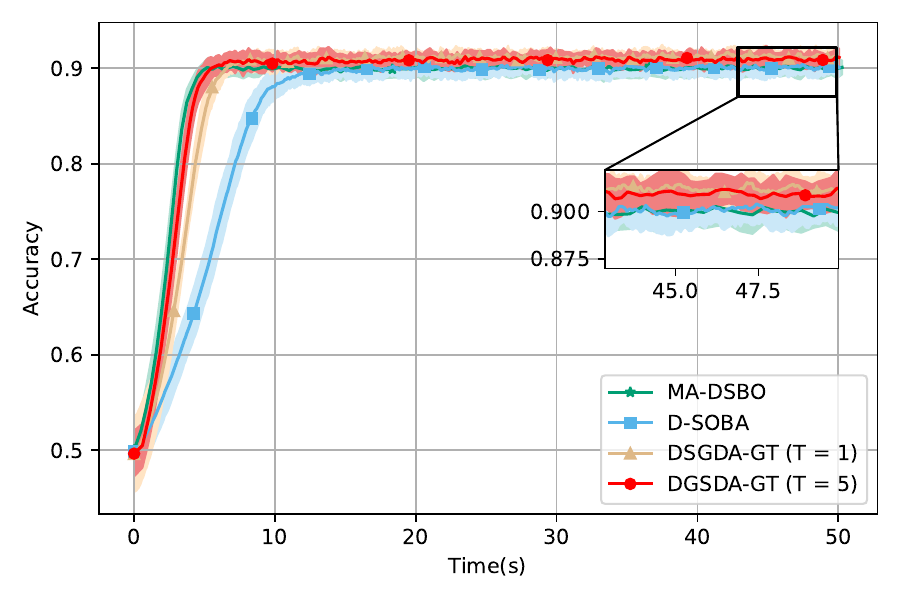}}
	\vspace{-0.2cm}
	\caption{Training loss and test accuracy of $\ell^2$-regularized logistic regression on synthetic data. The vertical axis of Figure \ref{figure:comp_loss} is in log scale.}\label{Figure: syn_acc}
	\vspace{-0.3cm}
\end{figure*}

\subsection{Real-world data}
We then test the performance of our algorithm on real-world data -- MNIST~\citep{lecun1998gradient}, with functions $(\cL_{\text{val}}^{(i)}, \cL_{\text{train}}^{(i)})$ defined as
\begin{align*}
    \cL_{\text{val}}^{(i)}(\lambda, \omega) &= \frac{1}{|\cD_{\text{val}}^{(i)}|}\sum_{(x_e,y_e)\in\cD_{\text{val}}^{(i)}}L(x_e^{\top}\omega, y_e),\\
    \cL_{\text{train}}^{(i)}(\lambda, \omega) &= \frac{1}{|\cD_{\text{train}}^{(i)}|}\sum_{(x_e,y_e)\in\cD_{\text{train}}^{(i)}}L(x_e^{\top}\omega, y_e) + \frac{1}{cd}\sum_{i=1}^{c}\sum_{j=1}^{d}e^{\lambda_j}\omega_{ij}^2,
\end{align*}
where we denote by $L$ the cross-entropy loss, and $(c,d)=(10, 784)$ represent the number of classes and number of features.  {More details of the parameters in the experiments are provided in Appendix~\ref{numerical:appendix}.} We plot the training loss and test accuracy with respect to training time in Figure~\ref{Figure: MNIST_loss}.  Our Algorithm \ref{alg:stochastic_minmax:onestage} with different settings is consistently better than existing ones in terms of training loss and accuracy. Moreover, we can observe better generalization performance of the fully first-order algorithm over the second-order algorithms under the same training time. Our Algorithm also provides more flexibility, in the sense that we can set the number of inner-loop iterations $T$ to be greater than $1$, which gives a double-loop algorithm, which has been proven beneficial over the fully single-loop ones both theoretically~\citep{chen2021closing, ji2022will} and also empirically in our Figures~\ref{figure:comp_time_mnist_loss} and~\ref{figure:comp_time_mnist_acc}.

\begin{figure*}[ht]
\centering  
    \subfigure[]{\label{figure:comp_time_mnist_loss}\includegraphics[width=0.49\textwidth]{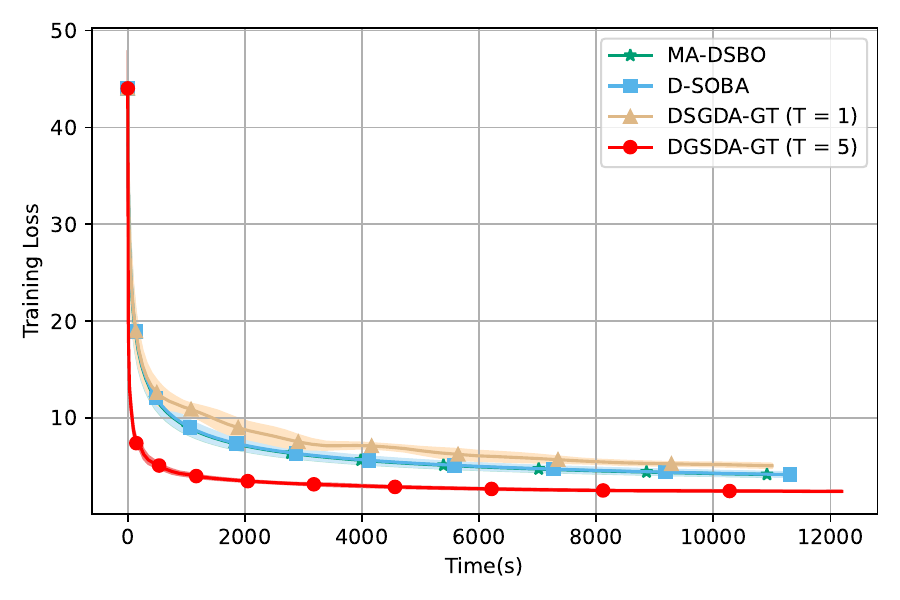}}
    \subfigure[]{\label{figure:comp_time_mnist_acc}\includegraphics[width=0.49\textwidth]{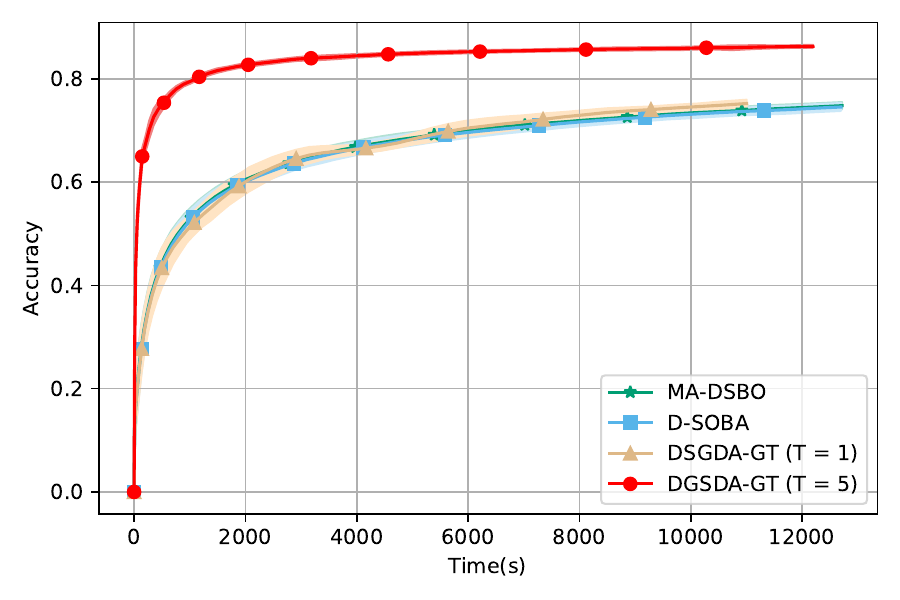}}
\vspace{-0.2cm}
\caption{$\ell^2$-regularized logistic regression on MNIST.}\label{Figure: MNIST_loss}
\vspace{-0.3cm}
\end{figure*}



\section{Conclusion}
In this paper, we propose a novel algorithm called Decentralized Stochastic Gradient Descent Ascent with Gradient Tracking (DSGDA-GT) for solving decentralized stochastic bilevel optimization problems. The proposed algorithm only requires the first-order gradient oracle, making it more efficient compared to the existing methods that involve second-order oracles. We provide the first-order oracle complexity $\cO(n^{-1}\epsilon^{-7})$ to find an $\epsilon$ stationary point, which matches the well-known result in the single agent method~\cite{kwon2023fully}. In the future, it will be interesting to improve the convergence rate of the fully first-order methods under stronger assumptions and large-batch settings. Moreover, investigating the fundamental limits and analyzing the lower bound of such problems is an area of independent interest. 


\bibliographystyle{unsrt}
\bibliography{ref}

\newpage


\appendix

\tableofcontents

\newpage
\section{Appendix / Auxiliary lemmas for theoretical results}\label{sec: aux_lemma_app}

In this section, we analyze the convergence of Algorithm \ref{alg:stochastic_minmax:onestage}. For convenience, we first introduce our notational conventions. $\bfone_n$ denotes the all-one vector in $\R^n$. $\norm{\cdot}$ represents $\ell^2$-norm for vectors and Frobenius norm for matrices. $\norm{\cdot}_2$ denotes the spectral norm for matrices.
\begin{align*}
    &\mX_s = \left(x_s^{(1)}, ..., x_s^{(n)}\right), \mY_s = \left(y_s^{(1)}, ..., y_s^{(n)}\right), \mZ_s = \left(z_s^{(1)}, ..., z_s^{(n)}\right), \notag \\
    & \mV_{s} = \left(v_{s}^{(1)}, ..., v_{s}^{(n)}\right), \mDelta_s = \left(\delta_s^{(1)}, ..., \delta_s^{(n)}\right). \\
    &\bar x_s = \frac{1}{n}\mX_s\bfone_n = \avein x_s^{(i)},\ \bar y_s = \frac{1}{n}\mY_s\bfone_n = \avein y_s^{(i)},\ \bar z_s = \frac{1}{n}\mZ_s\bfone_n = \avein z_s^{(i)},\\
    &\bar v_s = \frac{1}{n}\mV_{s}\bfone_n = \avein v_{s}^{(i)},\ \bar \delta_s = \frac{1}{n}\mDelta_s\bfone_n = \avein \delta_s^{(i)}.\\
    & y^{\alpha}_*(x) := \argmin_y \Omega^{\alpha}(x,y),\ z_*(x) := \argmin_z g(x,z), \notag \\
    & \ y_{*, s}^{\alpha} := \argmin_{y}\Omega^{\alpha}(\bar x_s, y),\ z_{*, s} := \argmin_z g(\bar x_s, z). \\
    & \cF_s = \sigma\left(\bigcup_{i=1}^{n}\left\{x_0^{(i)}, y_0^{(i)}, z_0^{(i)}, v_0^{(i)}, ..., x_s^{(i)}, y_s^{(i)}, z_s^{(i)}, v_s^{(i)}\right\}\right).
\end{align*}
In the following analysis, the symbol $\lesssim$ indicates that there exists an absolute constant $C$ such that LHS $\leq$ $C$ RHS, and for simplicity, omitting $C$ does not affect the order of RHS. 

Note that suppose that $f_i$ and $g_i$ for each agent $i$ satisfy the variance bounded condition in Assumption~\ref{aspt: so}, we might let 
\begin{align}
\E\left[\norm{\nabla F_i(x,y;\xi) - \nabla f_i(x, y)}^2 \right] \leq \sigma_f^2; \quad  \E\left[\norm{\nabla G_i(x,y;\psi) - \nabla g_i(x, y)}^2 \right] \leq \sigma_g^2.
\end{align}

The following technical lemmas are standard.
\begin{lemma}\label{lem: cs}
    For any $m, n\in \N_+$ and matrices $\mA, \mB\in \R^{m\times n}$ and $c>0$, we have:
    \[
        \norm{\mA + \mB}^2\leq (1+c)\norm{\mA}^2 + (1 + c^{-1})\norm{\mB}^2.
    \]
\end{lemma}
\begin{lemma}\label{lem: F_norm_ineq}
    For any $p, q, r\in \N_+$ and matrices $\mA\in \R^{p\times q}, \mB\in \R^{q\times r}$, we have:
    \[
        \|\mA\mB\| \leq \min\left(\|\mA\|_2\cdot \|\mB\|, \|\mA\|\cdot \|\mB\T\|_2\right).
    \]
\end{lemma}

\begin{lemma}\label{lem: basic_ineq}
    For three sequences $\{a_n\}_{n=0}^{\infty}, \{b_n\}_{n=0}^{\infty},\ \{\tau_n\}_{n=-1}^{\infty},$ and a constant $r$ satisfying
    \begin{equation}\label{ineq: conditions_cons_decay}
        a_{k+1}\leq r a_k + b_{k},\ a_k\geq 0,\ b_k\geq 0,\  0 = \tau_{-1}\leq \tau_{k+1}\leq \tau_k\leq 1,\ 0< r< 1,
    \end{equation}
    for all $k\geq 0$. Then for any $K > 0$, we have
    \begin{align}
        a_k &\leq r^ka_0 + \sum_{i=0}^{k-1}r^{k-1-i}b_{i}, \label{ineq: ak}\\
        \sum_{k=0}^{K}\tau_ka_k &\leq \frac{1}{1-r}\left(\tau_0a_0 + \sum_{k=0}^{K}\tau_kb_k\right).\label{ineq: aksum}
    \end{align}
\end{lemma}
\begin{proof}(of Lemma~\ref{lem: basic_ineq})
    To prove \eqref{ineq: ak}, notice that we have $\frac{a_i}{r^i}\leq \frac{a_{i-1}}{r^{i-1}} + \frac{b_{i-1}}{r^i}$, and thus taking summation for $1\leq i\leq k$ on both sides completes the proof. To prove \eqref{ineq: aksum}, note that we have
    \begin{align*}
        (1-r)\sum_{k=0}^{K}\tau_ka_k 
        \leq &\sum_{k=0}^{K}\tau_k(a_k - a_{k+1} + b_k) = \sum_{k=0}^{K}(\tau_k - \tau_{k-1})a_k - \tau_Ka_{K+1} + \sum_{k=0}^{K}\tau_kb_k\leq \tau_0a_0 + \sum_{k=0}^{K}\tau_kb_k,
    \end{align*}
    where the inequalities use \eqref{ineq: conditions_cons_decay}, and the equality uses summation by parts.
\end{proof}
\begin{lemma}\label{lem:double:sum}
For the sequence $\left\lbrace x_n \right\rbrace_{n=1}^{N}$ and constant $r \in (0,1)$, then
\begin{align*}
\sum_{s=0}^{N}\sum_{n=0}^{s}r^{s-n}x_n = \sum_{n=0}^{N}\sum_{s=n}^{N} r^{s-n}x_n \leq \frac{1}{1-r} \sum_{n=0}^{N} x_n.
\end{align*}
\end{lemma}
\begin{lemma}\label{lem: W_m}
    Suppose Assumption \ref{aspt: sgap} holds. For any $m\in \N+$, we have
    \[
        \norm{\mW^m - \frac{\bfone_n\bfonet_n}{n}}_2\leq \rho^m.
    \]
\end{lemma}

\begin{lemma}\label{lem: gd_decrease}
    Suppose $f(x)$ is $\mu$-strongly convex and $\ell$-smooth. For any $x$ and $\gamma<\frac{2}{\mu + \ell}$, define $x^+ = x - \gamma\nabla f(x),\ x^*=\argmin_x f(x)$. Then we have
    \[
        \|x^+ - x^*\| \leq (1-\gamma\mu)\|x-x^*\|.
    \]
\end{lemma}
\begin{proof}
    See, e.g., Lemma 10 in \cite{qu2017harnessing}.
\end{proof}

\begin{lemma}\label{lem: gt}
    Suppose Assumption \ref{aspt: sgap} holds. We have for all $0\leq s\leq S-1$ that 
    \begin{align*}
        \bar v_s = \bar \delta_s.
    \end{align*}
\end{lemma}
\begin{proof}(of Lemma~\ref{lem: gt})
    We first note that each $v_{s+1}^{(i)}$ is introduced in the gradient tracking step of Algorithm \ref{alg:stochastic_minmax:onestage}, i.e.,
    \begin{align*}
        v_{s+1}^{(i)} = \sum_{j=1}^{n}w_{ij} v_{s}^{(j)} + \delta_{s+1}^{(i)} - \delta_{s}^{(i)}
    \end{align*}
    Computing the average on both sides and using the fact that $\mW$ is doubly stochastic, we have
    \begin{align*}
        \bar v_{s+1} = \bar v_s + \bar \delta_{s+1} - \bar \delta_s.
    \end{align*}
    Hence, $\bar v_s = \bar \delta_s$ given the initialization $\bar v_0 = \bar \delta_0$. 
\end{proof}

\subsection{Properties of min-max functions and its optimal functions}\label{sec:appendix:functions}
Suppose Assumption~\ref{aspt: smoothness} hold, the functions $\cL^{\alpha}(x,y,z)$ and  $\Gamma^{\alpha}(x)$ satisfy the following properties. 
\begin{restatable}[]{lemma}{lemsmoothminmaxprob} 
\label{lem: smoothness_minmax_prob}
Under Assumption~\ref{aspt: smoothness}, the followings hold:
\begin{itemize}
\item[(i)] $\cL^{\alpha}(x,y,z)$ is $\mu_g\alpha$-strongly concave w.r.t. $z$; 
\item[(ii)]  $\cL^{\alpha}(x,y,z)$ is $\mu_g\alpha/2$-strongly convex w.r.t. $y$ if $\alpha > 2\ell_{f,1}/\mu_g$.
\end{itemize}
\end{restatable}
The results of Lemma~\ref{lem: smoothness_minmax_prob} can be found in~\citep{kwon2023fully} and Lemma B.1 of \citep{chen2023near}. From Lemma B.7 in \citep{chen2023near}, the following result holds for $\Gamma^{\alpha}(x)$:
\begin{lemma}
Under Assumption~\ref{aspt: smoothness}, if $\alpha > 2\ell_{f,1}/\mu_g$, then $\Gamma^{\alpha}(x)$ is $\ell_{\Gamma}$-smooth, where $\ell_{\Gamma} = \mathcal{O}( \kappa^3)$ is a constant that is independent on $\alpha$.
\end{lemma}
Moreover, the functions $y_*^{\alpha}(x)$ and $z_*(x)$ satisfy the following properties.
\begin{lemma}\label{lem:y:ast}
Under Assumption~\ref{aspt: smoothness}, we have
\begin{align*}
 \left\| y_{\ast}^{\alpha}(x) - y^{\ast}(x)\right\| \leq \frac{C_0}{\alpha}
\end{align*}
where $C_0 = \ell_{f,0}/\mu_{g}$.
\end{lemma}
The result in Lemma~\ref{lem:y:ast} follows from Lemma B.2 of \cite{chen2023near}.
\begin{lemma}\label{lem:appendix:zy}
Under Assumption~\ref{aspt: smoothness}, if $\alpha > 2\ell_{f,1}/\mu_g$,  then we have
\begin{itemize}
\item[(i)] $z_{*}(x)$ is $\kappa$-Lipschitz continuous;
\item[(ii)] $y_{*}^{\alpha}(x)$ is $\ell_{y_{*},0}$-Lipschitz continuous where $\ell_{y_*,0}=3\kappa$. 
\end{itemize}
\end{lemma}
Claim (i) in Lemma~\ref{lem:appendix:zy} can be found in Lemma 2.2 of \citep{ghadimi2018approximation} and Claim (ii) implies from Lemma 3.2 (setting $\lambda_1=\lambda_2$) of \citep{kwon2023fully}. 

\begin{restatable}[]{lemma}{lemasptsmoothyz} 
\label{lem_aspt: smoothness:yz}
Under Assumption~\ref{aspt: smoothness}, if $\alpha > 2\ell_{f,1}/\mu_g$,
\begin{itemize}
\item[(i)] $y_{*}^{\alpha}(x)$ is $\ell_{\nabla y_*}$-smooth where $\ell_{\nabla y_*} = \mathcal{O}\left(\frac{\kappa^2}{\mu_g}\left(\frac{\ell_{f,2}}{\alpha} + \ell_{g,2}\right)\right)$
\item[(ii)] $z_{*}(x)$ is $\ell_{\nabla z_*}$-smooth where $\ell_{\nabla z_*} = \mathcal{O}\left(\frac{\kappa^2}{\mu_g}\left(\ell_{g,1} + 1\right) \right)$
\end{itemize}
\end{restatable}
Following Lemma A.3 of \citep{kwon2023fully} and recalling the Lipschitz continuous property of $y_*^{\alpha}(x)$,  we have the first claim (i) is correct. Note that to ensure the smoothness of $y_*^{\alpha}(x)$, we need to assume the Hessian-Lipschitz of $f$. Similarly,  recalling the Lipschitz continuity of $z_*(x)$ from Lemma~\ref{lem:appendix:zy}, the function $z_{*}(x)$ is gradient Lipschitz, that is Claim (ii) holds.

\section{Appendix / Analysis of Algorithm~\ref{alg:inner_loop}}\label{sec:convergence_app}

In Algorithm~\ref{alg:stochastic_minmax:onestage}, the updates for $y_{s}^{(i)}$ and $z_s^{(i)}$ are essentially $T$-step decentralized stochastic gradient descent with gradient tracking (see Algorithm~\ref{alg:inner_loop}). Hence, their convergence and the consensus can be analyzed through the following technical lemma.
\begin{restatable}[]{lemma}{leminnerconverge} 
\label{lem: dsgt}
    Suppose $\phi_i(x, \theta)$ in Algorithm \ref{alg:inner_loop} is $\ell$-smooth and $\mu$-strongly convex. The stochastic oracle $h_{t+1}^{(i)} = \nabla_{\theta}\phi_i(x^{(i)}, \theta_t^{(i)}; \zeta_t^{(i)})$ is unbiased with variance bounded by $\sigma^2$, and is independent of $h_{t+1}^{(j)}$ conditioning on all iterates with subscripts up to $t$. Define 
    \begin{align*}
        &\mTheta_t = \left(\theta_t^{(1)}, ..., \theta_t^{(n)}\right),\ \mH_t = \left(h_t^{(1)}, ..., h_t^{(n)}\right),\ \bar x = \avein x^{(i)},\ \phi(x, \theta) = \avein \phi_i(x, \theta)\\
        &\theta_* = \argmin_{\theta} \avein \phi_i(\bar x, \theta),\ \cG_t = \sigma\left(\bigcup_{i=1}^{n}\{\theta_0^{(i)}, h_0^{(i)}, ..., \theta_t^{(i)}, h_t^{(i)}, x^{(i)}\}\right).
    \end{align*}
    If $\gamma < \frac{1}{\ell} \leq \frac{2}{\mu + \ell}$, we have
    \begin{subequations}
    \begin{align}
        &\E\left[\norm{\bar \theta_{t+1} - \theta_*}^2\mid \cG_t\right] \notag\\
        \leq &(1-\gamma\mu)\norm{\bar \theta_t - \theta_*}^2 + \frac{2\gamma\ell^2}{\mu n} \left(\norm{\mX - \bar x\bfonet_n}^2 + \norm{\mTheta_t - \bar \theta_t\bfonet_n}^2 \right) + \frac{\gamma^2\sigma^2}{n}, \label{ineq: dsgt:a}\\
        &\norm{\mTheta_{t+1} - \bar \theta_{t+1}\bfonet_n}^2\leq \frac{1+\rho^2}{2}\norm{\mTheta_t - \bar \theta_t\bfonet_n}^2 + \frac{(1+\rho^2)\gamma^2}{1-\rho^2}\norm{\mU_{t+1} - \bar u_{t+1}\bfonet_n}^2, \label{ineq: dsgt:b} \\
        &\E\left[\norm{\mU_{t+1} - \bar u_{t+1}\bfonet_n}^2\right] \notag \\
        \leq &\left(\frac{1+\rho^2}{2} + \frac{6\ell^2\gamma^2(1+\rho^2)}{1-\rho^2}\right)\E\left[\norm{\mU_t - \bar u_t\bfonet_n}^2\right] + \frac{36(1+\rho^2)\ell^2}{1-\rho^2} \E\left[\norm{\mTheta_{t-1} - \bar \theta_{t-1}\bfonet}^2\right]  \notag \\
    + &\frac{12(1+\rho^2)\ell^4\gamma^2}{1-\rho^2}\E\left[\norm{\mX - \bar x\bfonet_n}^2\right] + \frac{12n(1+\rho^2)\ell^4\gamma^2}{1-\rho^2}\E\left[\norm{\bar \theta_{t-1} - \theta_*}^2\right] + \frac{12n(1+\rho^2)\sigma^2}{1-\rho^2}.  \label{ineq: dsgt:c}
    \end{align}
    \end{subequations}
\end{restatable}
\begin{proof}(of Lemma~\ref{lem: dsgt})
   At each step, we have
    \begin{align}\label{eq: dsgt_update}
        \mU_{t+1} = \mU_t\mW + \mH_{t+1} - \mH_t,
        \mTheta_{t+1} = \mTheta_t\mW - \gamma \mU_{t+1},\ \bar \theta_{t+1}  = \bar \theta_t - \gamma\bar u_{t+1} =  \bar \theta_t - \gamma\bar h_{t+1}.
    \end{align}
    To prove the first inequality  \eqref{ineq: dsgt:a}, we have 
 \begin{align*}
 &\bar \theta_{t+1} -  \theta_*  = \bar \theta_t  - \gamma\bar h_{t+1} - \theta_* \notag \\
  &  = \, \bar \theta_t - \theta_* - \gamma\nabla_{\theta} \phi(\bar x, \bar \theta_t) - \gamma \left(\E[\bar h_{t+1} \mid \cG_t] - \nabla_{\theta} \phi(\bar x, \bar \theta_t) \right) - \gamma \left(\bar h_{t+1} - \E[\bar h_{t+1} \mid \cG_t]\right).
 \end{align*}
This implies
\begin{align}
&\E\left[\norm{\bar \theta_{t+1} - \theta_*}^2\mid \cG_t \right]  \notag \\
= & \norm{\bar \theta_t - \theta_* - \gamma\nabla_{\theta} \phi(\bar x, \bar \theta_t) - \gamma \left(\E[\bar h_{t+1} \mid \cG_t] - \nabla_{\theta} \phi(\bar x, \bar \theta_t) \right) }^2 + \gamma^2\E\left[ \norm{\bar h_{t+1} - \E[\bar h_{t+1} \mid \cG_t}^2\mid \cG_t\right] \notag \\
\leq& (1 + \gamma\mu)\norm{\bar \theta_t - \theta_* - \gamma\nabla_{\theta} \phi(\bar x, \bar \theta_t) }^2 + \left(1 + \frac{1}{\gamma\mu}\right)\gamma^2\norm{\E[\bar h_{t+1} \mid \cG_t] - \nabla_{\theta} \phi(\bar x, \bar \theta_t) }^2 \notag \\
& \quad + \gamma^2\E\left[ \norm{\bar h_{t+1} - \E[\bar h_{t+1} \mid \cG_t]}^2\mid \cG_t\right] \label{eq: theta_conv},
\end{align}
where the first equality holds by the unbiasedness of $h_{t+1}^{(i)}$. We use Lemma \ref{lem: gd_decrease} to estimate the first term of \eqref{eq: theta_conv}:
\begin{align*}
\norm{\bar \theta_t - \theta_* - \gamma\nabla_{\theta} \phi(\bar x, \bar \theta_t) }^2\leq (1-\gamma\mu)^2 \norm{\bar \theta_t - \theta_* }^2.
\end{align*}
Then we focus on the second term of \eqref{eq: theta_conv}:
\begin{align*}
\norm{\E[\bar h_{t+1} \mid \cG_t] - \nabla_{\theta} \phi(\bar x, \bar \theta_t) }^2  & = \norm{\frac{1}{n}\sum_{i=1}^n \nabla_{\theta} \phi_i(x^{(i)}, \theta_t^{(i)}) - \nabla_{\theta} \phi_i(\bar x, \bar \theta_t) }^2 \notag \\
& \leq \frac{1}{n}\sum_{i=1}^n \norm{\nabla_{\theta} \phi_i(x^{(i)}, \theta_t^{(i)}) - \nabla_{\theta} \phi_i(\bar x, \bar \theta_t)  }^2  \notag \\
& \leq  \frac{\ell^2}{n}\sum_{i=1}^n \left(\norm{x^{(i)} - \bar x}^2 + \norm{\theta_t^{(i)} - \bar \theta_t }^2 \right)
\end{align*}
where the last inequality follows from the Lipschitz smoothness of each $\phi_i$. Next, we estimate the third term of \eqref{eq: theta_conv}:
\begin{align}
 &\E\left[ \norm{\bar h_{t+1} - \E[\bar h_{t+1} \mid \cG_t]}^2\mid \cG_t\right]  = \E\left[ \norm{ \frac{1}{n}\sum_{i=1}^n \left(h_{t+1}^{(i)} - \E\left[h_{t+1}^{(i)}\mid \cG_t\right]\right)}^2\middle| \cG_t\right]  \notag \\
 =&  \frac{1}{n^2}\sum_{i=1}^n \E\left[ \norm{ h_{t+1}^{(i)} -  \E\left[h_{t+1}^{(i)}\mid \cG_t\right]}^2\middle|\cG_t\right] + \frac{1}{n^2} \sum_{j \neq i}\E\left[ \left\langle  h_{t+1}^{(i)} -  \E\left[h_{t+1}^{(i)}\mid \cG_t\right], h_{t+1}^{(j)} -  \E\left[h_{t+1}^{(j)} \mid \cG_t\right]\right\rangle \middle| \cG_t\right] \notag \\
 & 
 \leq   \frac{\sigma^2}{n}
\end{align}
where the inequality uses the bounded variance, unbiasedness, and the independence of different stochastic oracles. Substituting the above results into \eqref{eq: theta_conv}, we have
\begin{align*}
\E\left[\norm{\bar \theta_{t+1} - \theta_*}^2 \mid \cG_t\right] 
\leq &(1-\gamma\mu)\norm{\bar \theta_t - \theta_*}^2 + \left(1 + \frac{1}{\gamma\mu}\right) \frac{\gamma^2\ell^2}{n} \left(\norm{\mX - \bar x\bfonet_n}^2 + \norm{\mTheta_t - \bar \theta_t\bfonet_n}^2 \right) + \frac{\sigma^2\gamma^2}{n}.
\end{align*}
The first inequality \eqref{ineq: dsgt:a} holds due to the step-size $\gamma < \frac{1}{\ell} \leq \frac{1}{\mu}$. Now for the second inequality \eqref{ineq: dsgt:b}, by \eqref{eq: dsgt_update} we have
\begin{align}
    \mTheta_{t+1} - \bar \theta_{t+1}\bfonet_n & =  \mTheta_t\mW - \gamma\mU_{t+1} - (\bar \theta_t - \gamma \bar u_{t+1})\bfonet_n \notag \\
    & = \left(\mTheta_t - \bar \theta_t\bfonet_n\right)\left(\mW - \frac{\bfone_n\bfonet_n}{n}\right) - \gamma\left(\mU_{t+1} - \bar u_{t+1}\bfonet_n\right).
\end{align}
By Lemmas \ref{lem: cs} and \ref{lem: W_m} we know for any $c>0$,
\begin{align*}
    \norm{\mTheta_{t+1} - \bar \theta_{t+1}\bfonet_n}^2 & \leq (1 + c)\rho^2\norm{\mTheta_t - \bar \theta_t\bfonet_n}^2 + (1 + c^{-1})\gamma^2\norm{\mU_{t+1} - \bar u_{t+1}\bfonet_n}^2.
\end{align*}
We set $c = \frac{1-\rho^2}{2\rho^2}$ and obtain the second inequality \eqref{ineq: dsgt:b}. Finally, for the third inequality  \eqref{ineq: dsgt:c}, we have from \eqref{eq: dsgt_update} that
\begin{align}
    \mU_{t+1} - \bar u_{t+1}\bfonet_n 
    &= \mU_t\mW + \mH_{t+1} - \mH_t - (\bar u_t + \bar h_{t+1} - \bar h_t)\bfonet_n \notag\\
    &= \left(\mU_t - \bar u_t\bfonet_n\right)\left(\mW - \frac{\bfone_n\bfonet_n}{n}\right) + \left(\mH_{t+1} - \mH_t\right)\left(\mI_n - \frac{\bfone_n\bfonet_n}{n}\right). \label{eq: U_decompose}
\end{align}
which, together with Lemmas \ref{lem: cs}, \ref{lem: F_norm_ineq} and \ref{lem: W_m}, and $\norm{\mI_n - \frac{\bfone_n\bfonet_n}{n}}_2\leq 1$, implies
\begin{align*}
    \norm{\mU_{t+1} - \bar u_{t+1}\bfonet_n}^2\leq \frac{1+\rho^2}{2}\norm{\mU_t - \bar u_t\bfonet_n}^2 + \frac{1+\rho^2}{1-\rho^2}\norm{\mH_{t+1} - \mH_t}^2.
\end{align*}
To bound $\norm{\mH_{t+1} - \mH_t}$, we have
\begin{align}\label{eq: H_diff}
    \mH_{t+1} - \mH_t = \mH_{t+1} - \E\left[\mH_{t+1}\mid \cG_t\right] - (\mH_t - \E\left[\mH_t\mid \cG_{t-1}\right])+\E\left[\mH_{t+1}\mid \cG_t\right]- \E\left[\mH_t\mid \cG_{t-1}\right]
\end{align}
and thus
\begin{align}\label{ineq: H_diff_exp}
   \E\left[\norm{\mH_{t+1} - \mH_t}^2\right] 
   &\leq 3\E\left[\norm{\mH_{t+1} - \E\left[\mH_{t+1}\mid \cG_t\right]}^2 + \norm{\mH_t - \E\left[\mH_t\mid \cG_{t-1}\right]}^2 + \norm{\E\left[\mH_{t+1}\mid \cG_t\right]- \E\left[\mH_t\mid \cG_{t-1}\right]}^2\right] \notag \\
    &\leq 6n\sigma^2 + 3\E\left[\norm{\E\left[\mH_{t+1}\mid \cG_t\right]- \E\left[\mH_t\mid \cG_{t-1}\right]}^2\right]
\end{align}
in which we bound $\norm{\E\left[\mH_{t+1}\mid \cG_t\right]- \E\left[\mH_t\mid \cG_{t-1}\right]}$ via the following inequalities:
\begin{align*}
    \norm{\E\left[\mH_{t+1}\mid \cG_t\right]- \E\left[\mH_t\mid \cG_{t-1}\right]}^2 &= \sum_{i=1}^{n}\norm{\nabla_{\theta}\phi_i(x^{(i)}, \theta_t^{(i)}) - \nabla_{\theta}\phi_i(x^{(i)}, \theta_{t-1}^{(i)})}^2\leq \ell^2\norm{\mTheta_t - \mTheta_{t-1}}^2. 
    \end{align*}
    \begin{align*}
    \norm{\mTheta_{t+1} - \mTheta_t}^2 &= \norm{\left(\mTheta_t - \bar \theta_t\bfonet_n\right)\left(\mW - \mI\right) - \gamma \mU_{t+1}}^2\leq 2\norm{\left(\mTheta_t - \bar \theta_t\bfonet_n\right)\left(\mW - \mI\right)}^2 + 2\gamma^2\norm{\mU_{t+1}}^2 \\
    & \leq 8\norm{\mTheta_t - \bar \theta_t\bfonet}^2 + 2\gamma^2\norm{\mU_{t+1} - \bar u_{t+1}\bfonet_n}^2 + 2\gamma^2\norm{\bar u_{t+1}\bfonet_n}^2.
    \end{align*}
    \begin{align*}
    \E\left[\norm{\bar u_{t+1}}^2\mid \cG_t\right] & = \E\left[\norm{\bar h_{t+1} - \E\left[\bar h_{t+1}\mid \cG_t\right] }^2\mid \cG_t\right] + \norm{\E\left[\bar h_{t+1}\mid \cG_t\right]}^2\leq \frac{\sigma^2}{n} + \norm{\E\left[\bar h_{t+1}\mid \cG_t\right]}^2. 
     \end{align*}
    \begin{align*}
     \norm{\E\left[\bar h_{t+1}\mid \cG_t\right]}^2 & = \norm{\E\left[\bar h_{t+1}\mid \cG_t\right] - \nabla_{\theta}\phi(\bar x, \bar \theta_t) + \nabla_{\theta}\phi(\bar x, \bar \theta_t) - \nabla_{\theta}\phi(\bar x, \theta_*)}^2\\
    &\leq \frac{2\ell^2}{n}\left(\norm{\mX - \bar x\bfonet_n}^2 + \norm{\mTheta_t - \bar \theta_t\bfonet_n}^2\right) + 2\ell^2\norm{\bar \theta_t - \theta_*}^2.
\end{align*}
Combining all the inequalities above, we obtain
\begin{align*}
    &\E\left[\norm{\E\left[\mH_{t+1}\mid \cG_t\right]- \E\left[\mH_t\mid \cG_{t-1}\right]}^2\right]\\
    \leq &\ell^2 \E\left[8\norm{\mTheta_{t-1} - \bar \theta_{t-1}\bfonet}^2 + 2\gamma^2\norm{\mU_t - \bar u_t\bfonet_n}^2 + 2\gamma^2\norm{\bar u_t\bfonet_n}^2\right] \\
    \leq &8\ell^2\E\left[\norm{\mTheta_{t-1} - \bar \theta_{t-1}\bfonet}^2\right] + 2\ell^2\gamma^2\E\left[\norm{\mU_t - \bar u_t\bfonet_n}^2\right] + 2\ell^2\gamma^2\left(\sigma^2 + n\E\left[\norm{\E\left[\bar h_t\mid \cG_t\right]}^2\right]\right) \\
    \leq &(8\ell^2 + 4\ell^4\gamma^2)\E\left[\norm{\mTheta_{t-1} - \bar \theta_{t-1}\bfonet}^2\right] + 2\ell^2\gamma^2\E\left[\norm{\mU_t - \bar u_t\bfonet_n}^2\right] +  4\ell^4\gamma^2\E\left[\norm{\mX - \bar x\bfonet_n}^2\right]  \\
    &+ 4\ell^4\gamma^2n\E\left[\norm{\bar \theta_{t-1} - \theta_*}^2\right] + 2\ell^2\gamma^2\sigma^2.
\end{align*}
and thus
\begin{align*}
    &\E\left[\norm{\mU_{t+1} - \bar u_{t+1}\bfonet_n}^2\right]\\
    \leq &\frac{1+\rho^2}{2}\E\left[\norm{\mU_t - \bar u_t\bfonet_n}^2\right] + \frac{1+\rho^2}{1-\rho^2}\E\left[\norm{\mH_{t+1} - \mH_t}^2\right] \\
    \leq &\frac{1+\rho^2}{2}\E\left[\norm{\mU_t - \bar u_t\bfonet_n}^2\right] + \frac{1+\rho^2}{1-\rho^2}\bigg(6n\sigma^2 + 3\bigg\{(8\ell^2 + 4\ell^4\gamma^2)\E\left[\norm{\mTheta_{t-1} - \bar \theta_{t-1}\bfonet}^2\right] \\
    &+ 2\ell^2\gamma^2\E\left[\norm{\mU_t - \bar u_t\bfonet_n}^2\right] +  4\ell^4\gamma^2\E\left[\norm{\mX - \bar x\bfonet_n}^2\right]
    + 4\ell^4\gamma^2n\E\left[\norm{\bar \theta_{t-1} - \theta_*}^2\right] + 2\ell^2\gamma^2\sigma^2
    \bigg\} \bigg) \\
    =&\left(\frac{1+\rho^2}{2} + \frac{6\ell^2\gamma^2(1+\rho^2)}{1-\rho^2}\right)\E\left[\norm{\mU_t - \bar u_t\bfonet_n}^2\right] + \frac{3(1+\rho^2)(8\ell^2 + 4\ell^4\gamma^2)}{1-\rho^2} \E\left[\norm{\mTheta_{t-1} - \bar \theta_{t-1}\bfonet}^2\right] \\
    & +\frac{1+\rho^2}{1-\rho^2}\left(12\ell^4\gamma^2\E\left[\norm{\mX - \bar x\bfonet_n}^2\right] + 12n\ell^4\gamma^2\E\left[\norm{\bar \theta_{t-1} - \theta_*}^2\right] + 6(n+\ell^2\gamma^2)\sigma^2\right).
\end{align*}
The third inequality \eqref{ineq: dsgt:c} holds by noticing that $\gamma < \frac{1}{\ell}$. We have completed the proof.
\end{proof}

Based on Lemma~\ref{lem: dsgt}, after $T$-steps, Algorithm~\ref{alg:inner_loop} achieves the following result.
\begin{restatable}[]{lemma}{lemsinnererrorconverge} 
\label{lem: inner_error}
   Under the same conditions as~Lemma~\ref{lem: dsgt}. Suppose the stepsize $\gamma$ satisfies 
    \begin{align}
      \gamma \leq \mathcal{O}\left(\min  \left\lbrace\frac{1-\rho^2}{\ell}, \frac{(1-\rho^2)\sqrt{\mu/\ell}}{\ell}, \frac{(1-\rho^2)^2}{\ell} \right\rbrace\right).
    \end{align}
    Define the constants
    \begin{align*}
    & e_{\theta}  = 1-\gamma\mu, \, e_{\rho,1} = \frac{\rho^2+1}{2},\, e_{\rho,2} = \frac{\rho^2+3}{4}\notag \\ 
   &  C_{x,1} = \left(\frac{\ell^2}{\mu} + \frac{\gamma^4\ell^6}{\mu(1-\rho^2)^4}\right),  C_{x,2}  = \frac{\ell^4}{(1-\rho^2)^4}\left(\frac{\ell^2}{\mu^2} + 1\right), C_{x,3}  = \frac{\left(\frac{\ell^2}{\mu^2} + 1 \right)\ell^4}{(1-\rho^2)^2},\\
   & C_{\sigma, 1}  =\left( \frac{\gamma\ell^2 n}{\mu \left(1-\rho^2 \right)^4} + 1\right), C_{\sigma,2} = \frac{\gamma^3\ell^4}{\mu(1-\rho^2)^4} + \frac{n}{(1-\rho^2)^4}, C_{\sigma,3} = \frac{1}{(1-\rho^2)^2}\left(\frac{\gamma^3\ell^4}{n \mu} +1 \right)
    \end{align*}
    then consider Algorithm \ref{alg:inner_loop}, for any $T \geq 1$,  we have
    \begin{align*}
\E\left[\norm{\bar \theta_{T} - \theta_*}^2\right]  
&\leq e_{\theta}^T \left( 1+ \frac{\gamma^4\ell^6}{\mu^2(1-\rho^2)^4} \right)\E\left[\norm{\bar \theta_0 - \theta_*}^2\right] + \frac{e_{\theta}^{T-1}\gamma\ell^2}{\mu(1-\rho^2)}\frac{1}{n}\E\left[\norm{\mTheta_0 - \bar \theta_0\bfonet_n}^2\right]  \notag \\  &  + \frac{e_{\theta}^{T-1}\gamma^3\ell^2}{\mu(1-\rho^2)^3n}\E\left[\norm{\mU_1 - \bar u_1\bfonet_n}^2\right]
+   \min \left(T, \frac{1}{\mu\gamma} \right) \left( \frac{C_{x,1}\gamma}{n}\E\left[\norm{\mX - \bar{x}\bfonet_n}^2\right] + \frac{C_{\sigma,1}}{n} \gamma^2 \sigma^2   \right);
    \end{align*}
    \begin{align}
  \frac{1}{n}\E\left[\norm{\mTheta_T - \bar \theta_T\bfonet_n}^2\right]  
& \leq  e_{\rho,1}^{T}\left(1 + \frac{\gamma^2\ell^2}{(1-\rho^2)^4}\right)\frac{1}{n}\E\left[\norm{\mTheta_0 - \bar \theta_0\bfonet_n}^2\right] 
 + \frac{e_{\rho,1}^{T-1} \gamma^2}{(1-\rho^2)^2}  \frac{1}{n}\E\left[\norm{\mU_{1} - \bar u_{1}\bfonet_n}^2\right] \notag \\
 & + \frac{e_{\rho,1}^{T-1} \gamma^3\ell^4}{\mu(1-\rho^2)^3}\E\left[\norm{\bar \theta_{0} - \theta_*}^2\right] 
  + \frac{C_{x,2}\gamma^4}{n}\E\left[\norm{X - \bar{x}\bfonet_n}^2\right] + \frac{C_{\sigma,2}\gamma^2\sigma^2}{n};\notag 
\end{align}
and 
\begin{align*}
\frac{1}{n}\E\left[\norm{\mU_{T+1} - \bar u_{T+1}\bfonet_n}^2\right] 
  & \leq e_{\rho,2}^T \left(1 + \frac{\ell^2\gamma^2}{(1-\rho^2)^4}\right)\frac{1}{n}\E\left[\norm{\mU_{1} - \bar u_{1}\bfonet_n}^2\right] + \frac{\ell^2 e_{\rho,2}^{T-1}}{(1-\rho^2)^3}\frac{1}{n}\E\left[\norm{\mTheta_{0} - \bar \theta_{0}\bfonet_n}^2\right] \\   & +\frac{e_{\rho,2}^{T-1}\ell^4\gamma}{\mu(1-\rho^2)}\E\left[\norm{\bar \theta_{0} - \theta_*}^2\right]  
+  \frac{C_{x,3}\gamma^2}{n}\E\left[\norm{\mX - \bar{x}\bfonet_n}^2\right] +C_{\sigma, 3}\sigma^2. 
\end{align*}
\end{restatable}
\begin{proof}(of Lemma~\ref{lem: inner_error})
We define the vector function $\Omega_{t}$
\begin{align}
\Omega_{t}= \left(\E\left[\norm{\bar \theta_{t} - \theta_*}^2\right],  \frac{1}{n}\E\left[\norm{\mTheta_{t} - \bar \theta_{t}\bfonet_n}^2\right],  \frac{1}{n}\E\left[\norm{\mU_{t+1} - \bar u_{t+1}\bfonet_n}^2\right] \right)^\top
\end{align}
and an $3\times 3$ matrix $M$
\begin{equation}
M = 
\begin{pmatrix}
M_{11} & M_{12} & M_{13}\\
M_{21} & M_{22} & M_{23}\\
M_{31} & M_{32} & M_{33}\\
\end{pmatrix}
\end{equation}
where
\begin{align}
   &  M_{11} = 1- \gamma \mu; \quad M_{12} =\frac{\gamma\ell^2}{\mu}; \quad M_{13} = 0  \notag \\
    & M_{21} = 0;\quad M_{22} = \frac{1+\rho^2}{2}; \quad M_{23} = \frac{\gamma^2}{1-\rho^2} \notag \\
    & M_{31}= \frac{\ell^4\gamma^2}{(1-\rho^2)}; \quad M_{32} = \frac{\ell^2}{(1-\rho^2)}; \quad M_{33} = \frac{1+\rho^2}{2} + \frac{6\ell^2\gamma^2(1+\rho^2)}{1-\rho^2}.
\end{align}
By the results of Lemma \ref{lem: dsgt}, for any $t$, we have 
\begin{align}\label{ineq: Omega_recursion:1}
    \Omega_{t+1} \leq M \Omega_{t} + \tilde{C}
\end{align}
where 
\begin{align}
\tilde{C} = \left(\frac{\gamma\ell^2}{\mu n}\E\left[\norm{X - \bar{x}\bfonet_n}^2\right] + \frac{\gamma^2\sigma^2}{n}, 0, \frac{\ell^4\gamma^2}{(1-\rho^2)n}\E\left[\norm{\mX - \bar x\bfonet_n}^2\right] + \frac{\sigma^2}{(1-\rho^2)}\right)^{T}.
\end{align}
Note that we omit the constant factor to simplify the definitions of matrix $M$ and $\tilde{C}$. For sufficient small stepsize $\gamma \leq \mathcal{O}\left(\frac{1-\rho^2}{\ell} \right)$, we have $M_{33} \leq \frac{3+\rho^2}{4}$. For simplicity, we overload the notation and set $\Omega_t = (a_t, b_t, c_t)^\top$ and $\tilde{C} = (d_1, d_2, d_3)^\top$. 
Note that we have
    \begin{align*}
        a_{t+1}&\leq M_{11}a_t + M_{12}b_t + d_1 \\
        b_{t+1}&\leq M_{22}b_t + M_{23}c_t\\
        c_{t+1}&\leq M_{31}a_t + M_{32}b_t + M_{33}c_t + d_3
    \end{align*}
    and thus we apply Lemma \ref{lem: basic_ineq} (\eqref{ineq: ak} to $a_t$ and \eqref{ineq: aksum} to $a_t, b_t, c_t$) to get 
    \begin{align}
        a_{t+1} &\leq M_{11}^{t+1}a_0 + M_{12}M_{11}^t\sum_{i=0}^{t}\frac{b_i}{M_{11}^i} + M_{11}^t\sum_{i=0}^{t}\frac{d_1}{M_{11}^i} \label{inequ:a:final}\tag{a*}\\
        \sum_{i=0}^{t}\frac{a_i}{M_{11}^i}&\leq \frac{1}{1-M_{11}}\left(a_0 + M_{12}\left(\sum_{i=0}^{t}\frac{b_i}{M_{11}^i}\right) + \left(\sum_{i=0}^{t}\frac{d_1}{M_{11}^i}\right) \right) \label{inequ:a:sum}\tag{a}\\
        \sum_{i=0}^{t}\frac{b_i}{M_{11}^i}&\leq \frac{1}{1-M_{22}}\left(b_0 +  M_{23}\left(\sum_{i=0}^{t}\frac{c_i}{M_{11}^i}\right)\right) \label{inequ:b:sum}\tag{b}\\
        \sum_{i=0}^{t}\frac{c_i}{M_{11}^i}&\leq \frac{1}{1-M_{33}}\left(c_0 +  M_{31}\left(\sum_{i=0}^{t}\frac{a_i}{M_{11}^i}\right) + M_{32}\left(\sum_{i=0}^{t}\frac{b_i}{M_{11}^i}\right) + \left(\sum_{i=0}^{t}\frac{d_3}{M_{11}^i}\right) \right) \label{inequ:c:sum}\tag{c} \\
            \sum_{i=0}^{t}\frac{c_i}{M_{11}^i}&\leq \frac{c_0}{1-M_{33}} +  \frac{M_{31}}{1-M_{33}}\frac{a_0}{1-M_{11}} + \frac{M_{31}}{1-M_{33}}\frac{M_{12}}{1-M_{11}}\left(\sum_{i=0}^{t}\frac{b_i}{M_{11}^i}\right) \notag \\& + \frac{M_{32}}{1-M_{33}}\left(\sum_{i=0}^{t}\frac{b_i}{M_{11}^i}\right) 
             + \frac{M_{31}}{1-M_{33}}\frac{1}{1-M_{11}}\left(\sum_{i=0}^{t}\frac{d_1}{M_{11}^i}\right) + \frac{1}{1-M_{33}}\left(\sum_{i=0}^{t}\frac{d_3}{M_{11}^i}\right) \label{inequ:c:sum:2}\tag{$\tilde{c}$}    
    \end{align}
   Incorporating \eqref{inequ:a:sum} into \eqref{inequ:c:sum} gives \eqref{inequ:c:sum:2}, the coefficient of $\sum_{i=0}^{t}\frac{b_i}{M_{11}^i}$ in \eqref{inequ:c:sum:2} is denoted by $R_0$
    \begin{align}
R_0 = \frac{M_{31}}{1-M_{33}}\frac{M_{12}}{1-M_{11}} + \frac{M_{32}}{1-M_{33}} \sim \Theta \left(\frac{\ell^6 \gamma^2}{(1-\rho^2)^2\mu^2} + \frac{\gamma^2}{(1-\rho^2)^2} \right) \sim \Theta \left(\frac{\ell^6 \gamma^2}{(1-\rho^2)^2\mu^2}\right), 
\end{align}
   and then doing the operations on the two inequalities $\frac{(1-M_{22})}{M_{23}} \times \eqref{inequ:b:sum} + \eqref{inequ:c:sum:2}$ gives
    \begin{align}\label{inequ:b:avg:1}
    \left(\frac{1-M_{22}}{M_{23}} -  R_0 \right)\sum_{i=0}^{t}\frac{b_i}{M_{11}^i}  & \leq \frac{b_0}{M_{23}}   + \frac{c_0}{1-M_{33}} +  \frac{M_{31}}{1-M_{33}}\frac{a_0}{1-M_{11}} \notag \\   
&+   \frac{M_{31}}{1-M_{33}}\frac{1}{1-M_{11}}\left(\sum_{i=0}^{t}\frac{d_1}{M_{11}^i}\right) +  \frac{1}{1-M_{33}}\left(\sum_{i=0}^{t}\frac{d_3}{M_{11}^i}\right).
    \end{align}
    Let 
    \begin{align}
    R_1 = \frac{1-M_{22}}{M_{23}} -  R_0 = \Theta \left(\frac{(1-\rho^2)^2}{\gamma^2} -  \frac{\ell^6 \gamma^2}{(1-\rho^2)^2\mu^2}\right).
    \end{align}
    For sufficient small stepsize $\gamma \leq (1-\rho^2) \sqrt{\mu/\ell}/\ell$, 
    we have $R_1 \geq \frac{(1-\rho^2)^2}{2\gamma^2}$. Then 
    \begin{align}\label{inequ:b:avg:2}
    \sum_{i=0}^{t}\frac{b_i}{M_{11}^i} &\leq  \frac{1}{R_1}\left(\frac{b_0}{M_{23}} + \frac{c_0}{1-M_{33}}  +   \frac{M_{31}}{1-M_{33}}\frac{a_0}{1-M_{11}} \right) \notag \\   &+  \frac{1}{R_1}\left(\frac{M_{31}}{1-M_{33}}\frac{1}{1-M_{11}}\left(\sum_{i=0}^{t}\frac{d_1}{M_{11}^i}\right) +  \frac{1}{1-M_{33}}\left(\sum_{i=0}^{t}\frac{d_3}{M_{11}^i}\right)\right).
    \end{align}
    Then incorporating \eqref{inequ:b:avg:2} into \eqref{inequ:a:final}, then
    \begin{align}
    a_{t+1} \leq & M_{11}^{t+1}a_0 + M_{11}^t\frac{M_{12}}{R_1}\left(\frac{M_{31}}{1-M_{33}}\frac{1}{1-M_{11}}\left(\sum_{i=0}^{t}\frac{d_1}{M_{11}^i}\right) + \frac{1}{1-M_{33}}\left(\sum_{i=0}^{t}\frac{d_3}{M_{11}^i}\right)\right) \notag \\
    &  + M_{11}^t\frac{M_{12}}{R_1}\left(\frac{b_0}{M_{23}} + \frac{c_0}{1-M_{33}}   +  \frac{M_{31}}{1-M_{33}}\frac{a_0}{1-M_{11}} \right) + M_{11}^t\sum_{i=0}^{t}\frac{d_1}{M_{11}^i}.
\end{align}
Incorporating the definitions of $a_t$, $M$ and $d_1, d_3$, we have
\begin{align}
& \E\left[\norm{\bar \theta_{T} - \theta_*}^2\right] \notag \\
& \leq \left(1-\mu\gamma\right)^T \E\left[\norm{\bar \theta_0 - \theta_*}^2\right] + \min \left(T, \frac{1}{\mu\gamma} \right) \left(\frac{\gamma\ell^2}{\mu n}\E\left[\norm{X - \bar{x}\bfonet_n}^2\right] + 
\frac{\gamma^2\sigma^2}{n} \right) \notag \\
& + \min \left(T, \frac{1}{\mu\gamma} \right)\frac{\gamma^3 \ell^2}{\mu \left(1-\rho^2 \right)^3} \left( \frac{\ell^4\gamma^2}{(1-\rho^2)n}\E\left[\norm{\mX - \bar x\bfonet_n}^2\right] + \frac{\sigma^2}{(1-\rho^2)}\right) \notag \\
& + \left(1-\mu\gamma\right)^{T-1}\left(\frac{\gamma \ell^2}{\mu (1-\rho^2)n}\E\left[\norm{\mTheta_0 - \bar \theta_0\bfonet_n}^2\right] +  \frac{\gamma^3\ell^2}{\mu(1-\rho^2)^3n} \E\left[\norm{\mU_1 - \bar u_1\bfonet_n}^2\right] \right) \notag \\
 & +\left(1-\mu\gamma\right)^{T-1} \frac{\gamma^4\ell^6}{\mu^2(1-\rho^2)^4}\E\left[\norm{\bar \theta_0 - \theta_*}^2\right]. 
\end{align}

Following the same process for sequence $a_t$, we may achieve the estimation for $b_t$. We apply Lemma \ref{lem: basic_ineq} (\eqref{ineq: ak} to $b_t$ and \eqref{ineq: aksum} to $a_t, c_t$) to get 
 \begin{align}
           b_{t+1} &\leq M_{22}^{t+1}b_0 +  M_{23}M_{22}^t\sum_{i=0}^t \frac{c_i}{M_{22}^i} \label{inequ:b:final} \tag{b*}\\
       \sum_{i=0}^{t}\frac{b_i}{M_{22}^i}
        & \leq \frac{1}{1-M_{22}}\left(b_0 + M_{23}\sum_{i=0}^t \frac{c_i}{M_{22}^i}\right) \label{inequ:b:avg} \tag{b'}\\
       \sum_{i=0}^{t}\frac{a_i}{M_{22}^i}
       & \leq \frac{1}{1-M_{11}}\left(a_0  + M_{12} \sum_{i=0}^{t}\frac{b_i}{M_{22}^i} + \sum_{i=0}^{t}\frac{d_1}{M_{22}^i} \right) \label{inequ:a:b} \tag{a'}\\
        \sum_{i=0}^{t}\frac{c_i}{M_{22}^i}&\leq \frac{1}{1-M_{33}}\left(c_0 +  M_{31}\left(\sum_{i=0}^{t}\frac{a_i}{M_{22}^i}\right) + M_{32}\left(\sum_{i=0}^{t}\frac{b_i}{M_{22}^i}\right) + \left(\sum_{i=0}^{t}\frac{d_3}{M_{22}^i}\right) \right). \label{inequ:b:c} \tag{c'}
    \end{align}
 Firstly, we incorporate \eqref{inequ:a:b} into \eqref{inequ:b:c} and get that 
 \begin{align}
  \sum_{i=0}^{t}\frac{c_i}{M_{22}^i} &\leq  \left(\frac{M_{31}}{1-M_{33}}\frac{a_0}{1-M_{11}}  + \frac{M_{31}}{1-M_{33}}\frac{M_{12}}{1-M_{11}} \sum_{i=0}^{t}\frac{b_i}{M_{22}^i} + \frac{M_{31}}{1-M_{33}}\frac{1}{1-M_{11}}\sum_{i=0}^{t}\frac{d_1}{M_{22}^i} \right)  \notag \\ & +\frac{c_0}{1-M_{33}}  + \frac{M_{32}}{1-M_{33}}\left(\sum_{i=0}^{t}\frac{b_i}{M_{22}^i}\right) 
   + \frac{1}{1-M_{33}}\left(\sum_{i=0}^{t}\frac{d_3}{M_{22}^i}\right). \label{inequ:b:c:2}
 \end{align}
 Let 
 \begin{align}
 R_3 = \frac{M_{31}}{1-M_{33}}\frac{M_{12}}{1-M_{11}} + \frac{M_{32}}{1-M_{33}} \sim \Theta \left( \frac{\ell^6\gamma^2}{\mu^2(1-\rho^2)^2} + \frac{\ell^2}{(1-\rho^2)^2}\right) \sim \Theta \left( \frac{\ell^2}{(1-\rho^2)^2}\right) ,
 \end{align} then
 we do the operations $R_3 \times \eqref{inequ:b:avg} + \eqref{inequ:b:c:2} $, we have
    \begin{align}\label{ineq: c:b:avg}
\left( 1- R_3\frac{M_{23}}{1-M_{22}}\right)\sum_{i=0}^{t}\frac{c_i}{M_{22}^i} & \leq    \left(\frac{M_{31}}{1-M_{33}}\frac{a_0}{1-M_{11}}   + \frac{M_{31}}{1-M_{33}}\frac{1}{1-M_{11}}\sum_{i=0}^{t}\frac{d_1}{M_{22}^i} \right)  \notag \\ &  
   + \frac{R_3b_0}{1-M_{22}} + \frac{c_0}{1-M_{33}} +\frac{1}{1-M_{33}}\left(\sum_{i=0}^{t}\frac{d_3}{M_{22}^i}\right).
    \end{align}
We can select sufficient small step-size $\gamma \leq (1-\rho^2)^2/\ell$ such that  the coefficient $1- R_3\frac{M_{23}}{1-M_{22}}\geq 1/2$.
    Incorporating the above inequality to \eqref{inequ:b:final} get that
    \begin{align}
   b_{t+1} 
   & \leq   M_{22}^{t+1}b_0 + 2M_{22}^{t} M_{23} \left(\frac{R_3b_0}{1-M_{22}} + \frac{c_0}{1-M_{33}} +  \frac{M_{31}}{1-M_{33}}\frac{1}{1-M_{11}}\left(a_0 +\sum_{i=0}^{t}\frac{d_1}{M_{22}^i} \right)\right) \notag \\
    & + \frac{2M_{22}^{t} M_{23}}{1-M_{33}}\left(\sum_{i=0}^{t}\frac{d_3}{M_{22}^i}\right),
    \end{align}
then we thus substitute the definitions $b_t$, $M$, $d_1, d_2$ into the above inequality:
\begin{align}
& \frac{1}{n}\E\left[\norm{\mTheta_T - \bar \theta_T\bfonet_n}^2\right] \notag \\
 \leq &\left(\frac{\rho^2+1}{2}\right)^{T} \frac{1}{n}\E\left[\norm{\mTheta_0 - \bar \theta_0\bfonet_n}^2\right] \notag\\
&+ \min\left(T, \frac{2}{1-\rho^2} \right) \frac{\gamma^3\ell^4}{\mu(1-\rho^2)^3} \left(\frac{\gamma\ell^2}{\mu n}\E\left[\norm{\mX - \bar{x}\bfonet_n}^2\right] + \frac{\gamma^2\sigma^2}{n} \right) \notag \\
& + \min \left(T, \frac{4}{1-\rho^2}\right)\frac{\gamma^2}{(1-\rho^2)^2}\left(\frac{\ell^4\gamma^2}{(1-\rho^2)n}\E\left[\norm{\mX - \bar x\bfonet_n}^2\right] + \frac{\sigma^2}{(1-\rho^2)}\right) \notag \\
& + \left(\frac{\rho^2+1}{2}\right)^{T} \left(\frac{\gamma^2\ell^2}{(1-\rho^2)^4n}\E\left[\norm{\mTheta_{0} - \bar \theta_{0}\bfonet_n}^2\right] + \frac{\gamma^2}{(1-\rho^2)^2n}\E\left[\norm{\mU_{1} - \bar u_{1}\bfonet_n}^2\right]\right) \notag \\
& + \left(\frac{\rho^2+1}{2}\right)^{T}\frac{\gamma^3\ell^4}{\mu(1-\rho^2)^3}\E\left[\norm{\bar \theta_{0} - \theta_*}^2\right]. \notag 
\end{align}

Since for the sequence $c_t$, we have the recursive formulation $c_{t+1} \leq M_{33}c_t + M_{31}a_t + M_{32}b_t + d_3$. Applying Lemma~\ref{lem: basic_ineq} (\eqref{ineq: ak} to $c_t$ and \eqref{ineq: aksum} to $a_t, b_t, c_t$), we have
\begin{align}
\sum_{i=0}^t\frac{a_i}{M_{33}^i} & \leq  \frac{1}{1-M_{11}}\left(a_0 + M_{12} \sum_{i=0}^t \frac{b_i}{M_{33}^i} + \sum_{i=0}^t \frac{d_1}{M_{33}^i} \right) \tag{a''} \label{inequ:a:avg:3}\\
\sum_{i=0}^t\frac{b_i}{M_{33}^i} & \leq  \frac{1}{1-M_{22}}\left(b_0 + M_{23} \sum_{i=0}^t \frac{c_i}{M_{33}^i} \right) \tag{b''} \label{inequ:b:avg:3}\\
\sum_{i=0}^t\frac{c_i}{M_{33}^i} & \leq  \frac{1}{1-M_{33}}\left(c_0 + M_{31}\sum_{i=0}^t \frac{a_i}{M_{33}^i} + M_{32} \sum_{i=0}^t \frac{b_i}{M_{33}^i} + \sum_{i=0}^t \frac{d_3}{M_{33}^i} \right) \tag{c''} \label{inequ:c:avg:3}\\
c_{t+1} & \leq M_{33}^{t+1}c_0 + M_{31}M_{33}^{t}\sum_{i=0}^t \frac{a_i}{M_{33}^i} + M_{32}M_{33}^{t}\sum_{i=0}^t \frac{b_i}{M_{33}^i} + M_{33}^{t}\sum_{i=0}^t \frac{d_3}{M_{33}^i}.  \tag{$c*$} \label{inequ:c:final:3}
\end{align}
Incorporating \eqref{inequ:a:avg:3} into \eqref{inequ:c:final:3} gives
\begin{align}
 c_{t+1} 
& \leq M_{33}^{t+1}c_0 + \frac{M_{31}M_{33}^{t}}{1-M_{11}}\sum_{i=0}^t\left(a_0 + \sum_{i=0}^t \frac{d_1}{M_{33}^i} \right) + \left(M_{32} + \frac{M_{31}M_{12}}{1-M_{11}} \right)M_{33}^{t}\sum_{i=0}^t \frac{b_i}{M_{33}^i} \notag \\
& + M_{33}^{t}\sum_{i=0}^t \frac{d_3}{M_{33}^i}.  \label{inequ:c:final:4}
\end{align}
Then incorporating \eqref{inequ:a:avg:3} into \eqref{inequ:c:avg:3}, we have
\begin{align}
\sum_{i=0}^t\frac{c_i}{M_{33}^i} & \leq \frac{1}{1-M_{33}}\left(c_0 + \frac{M_{31}a_0}{1-M_{11}} +  \left(\frac{M_{31}M_{12}}{1-M_{11}} + M_{32} \right)\sum_{i=0}^t \frac{b_i}{M_{33}^i} +  \frac{M_{31}}{1-M_{11}}\sum_{i=0}^t \frac{d_1}{M_{33}^i}\right) \notag \\
& + \frac{1}{1-M_{33}}\sum_{i=0}^t \frac{d_3}{M_{33}^i}. \label{inequ:ctob}
\end{align}
Combining \eqref{inequ:ctob} and \eqref{inequ:b:avg:3} and doing the operations $\frac{1-M_{22}}{M_{23}}\times \eqref{inequ:b:avg:3} + \eqref{inequ:ctob}$ gives:
\begin{align}
& \left(\frac{1-M_{22}}{M_{23}} - \frac{1}{1-M_{33}}\left(\frac{M_{31}M_{12}}{1-M_{11}} + M_{32} \right)\right) \sum_{i=0}^t\frac{b_i}{M_{33}^i} \notag \\
 \leq  &\frac{b_0}{M_{23}} + \frac{1}{1-M_{33}}\left(c_0 + \frac{M_{31}a_0}{1-M_{11}} +  \frac{M_{31}}{1-M_{11}}\sum_{i=0}^t \frac{d_1}{M_{33}^i}   + \sum_{i=0}^t \frac{d_3}{M_{33}^i} \right). 
\end{align}
Define $R_5, R_6$ and for sufficient small stepsize $\gamma \leq \min \left( (1-\rho^2)/\ell, (1-\rho^2)(\mu/\ell)^{1/3}/\ell\right)$
\begin{align*}
R_5 & =\frac{M_{31}M_{12}}{1-M_{11}} + M_{32} \sim \Theta \left(\frac{\ell^2}{1-\rho^2} \right), \notag \\
R_6 & = \frac{1-M_{22}}{M_{23}} - \frac{1}{1-M_{33}}\left(\frac{M_{31}M_{12}}{1-M_{11}} + M_{32} \right) \sim \Theta \left( \frac{(1-\rho^2)^2}{\gamma^2}\right).
\end{align*}
Thus
\begin{align*}
\sum_{i=0}^t\frac{b_i}{M_{33}^i}  \leq \frac{1}{R_6}\frac{b_0}{M_{23}} + \frac{1}{R_6}\frac{1}{1-M_{33}}\left(c_0 + \frac{M_{31}a_0}{1-M_{11}} +  \frac{M_{31}}{1-M_{11}}\sum_{i=0}^t \frac{d_1}{M_{33}^i}   + \sum_{i=0}^t \frac{d_3}{M_{33}^i} \right).
\end{align*}
Applying the above inequality into \eqref{inequ:c:final:4} gives
\begin{align}
&c_{t+1}  \notag\\
\leq &  M_{33}^{t+1}c_0 + \frac{M_{31}M_{33}^{t}}{1-M_{11}}\sum_{i=0}^t\left(a_0 + \sum_{i=0}^t \frac{d_1}{M_{33}^i} \right) + M_{33}^{t}\sum_{i=0}^t \frac{d_3}{M_{33}^i} \notag \\
+ & \left(M_{32} + \frac{M_{31}M_{12}}{1-M_{11}} \right)\frac{M_{33}^{t}}{R_6}\left(\frac{b_0}{M_{23}} + \frac{1}{1-M_{33}}\left(c_0 + \frac{M_{31}a_0}{1-M_{11}} +  \frac{M_{31}}{1-M_{11}}\sum_{i=0}^t \frac{d_1}{M_{33}^i}   + \sum_{i=0}^t \frac{d_3}{M_{33}^i} \right) \right) \notag \\
 \leq & M_{33}^{t+1}c_0 + \frac{M_{31}M_{33}^{t} a_0}{1-M_{11}} + R_5\frac{M_{33}^{t}}{R_6}\left(\frac{b_0}{M_{23}} + \frac{1}{1-M_{33}}\left(c_0 + \frac{M_{31}a_0}{1-M_{11}}  \right)\right) \notag \\
+ & M_{33}^{t}\left(\frac{M_{31}}{1-M_{11}} + \frac{R_5}{R_6}\frac{M_{31}}{1-M_{11}}\right)\sum_{i=0}^t \frac{d_1}{M_{33}^i} + \left(1 + \frac{R_5}{R_6}\frac{1}{1-M_{33}}\right)M_{33}^{t}\sum_{i=0}^t \frac{d_3}{M_{33}^i}. 
\end{align}
We thus substitute the value of $c_t, M, d_1, d_2$ and get the simplified result
\begin{align}
& \frac{1}{n}\E\left[\norm{\mU_{T+1} - \bar u_{T+1}\bfonet_n}^2\right] \notag \\
  & \leq e_{\rho,2}^T \left(1 + \frac{\ell^2\gamma^2}{(1-\rho^2)^4}\right)\frac{1}{n}\E\left[\norm{\mU_{1} - \bar u_{1}\bfonet_n}^2\right] + \frac{\ell^2 e_{\rho,2}^{T-1}}{(1-\rho^2)^3}\frac{1}{n}\E\left[\norm{\mTheta_{0} - \bar \theta_{0}\bfonet_n}^2\right] \notag \\
& + \frac{e_{\rho,2}^{T-1}\ell^4\gamma}{\mu(1-\rho^2)}\E\left[\norm{\bar \theta_{0} - \theta_*}^2\right]  + \left(\frac{\ell^2}{\mu^2} + 1 \right) \frac{\gamma^2\ell^4}{(1-\rho^2)^2n} \E\left[\norm{\mX - \bar{x}\bfonet_n}^2\right]  + \left(\frac{\gamma^3\ell^4}{n \mu} +1\right)\frac{\sigma^2}{(1-\rho^2)^2}. \notag 
\end{align}
The proof is complete.
\end{proof}

To get the recursive result of the $T$-step inner-loop, we need to carefully estimate $ \E\left[\norm{\mU_1 - \bar u_1\bfonet_n}^2\right]$.

{\bf Remark.} Note that by \eqref{eq: U_decompose} we have
\begin{align*}
    \E\left[\norm{\mU_1 - \bar u_1\bfonet_n}^2\right]\leq \E\left[\norm{\mU_0 - \bar u_0\bfonet_n}^2\right] + \frac{1}{1-\rho^2}\E\left[\norm{\mH_1 - \mH_0}^2\right].
\end{align*}
Following \eqref{eq: H_diff} and \eqref{ineq: H_diff_exp}, for 
\[
    z_{s+1}^{(i)}, u_{s+1, z}^{(i)}, h_{s+1, z}^{(i)} =  \texttt{Inner Loop}(z_s^{(i)}, \eta_{z}, g_i(x_s^{(i)}, \cdot), u_{s, z}^{(i)}, h_{s, z}^{(i)}, T)
\]
we know that $\mH_0$ is initialized by $\mH_{T+1}$, the output in the previous $T$-steps inner-loop update (see Algorithm~\ref{alg:inner_loop}). Hence, we know
\begin{align*}
    \E\left[\norm{\mH_1 - \mH_0}^2\right]\leq &6n\sigma_z^2 + 3\sum_{i=1}^{n}\E\left[\norm{\nabla_yg_i(x_s^{(i)}, z_s^{(i)}) - \nabla_y g_i(x_{s-1}^{(i)}, z_s^{(i)})}^2\right] \\
    \leq &6n\sigma_z^2 + 9\ell_{g,1}^2\E\left[\norm{\mX_s - \bar x_s\bfonet_n}^2 + \norm{\mX_{s-1} - \bar x_{s-1}\bfonet_n}^2 + n\norm{\bar x_s - \bar x_{s-1}}^2\right].
\end{align*}
Combining the above conclusions we know for $s \geq 1$
\begin{align}
   \E\left[\norm{\mU_1 - \bar u_1\bfonet_n}^2\right] 
     &  \leq\E\left[\norm{\mU_0 - \bar u_0\bfonet_n}^2\right] + \frac{6n\sigma_z^2}{1-\rho^2} \notag \\
     & + \frac{9\ell_{g,1}^2}{1-\rho^2} \E\left[\norm{\mX_s - \bar x_s\bfonet_n}^2 + \norm{\mX_{s-1} - \bar x_{s-1}\bfonet_n}^2 + n\norm{\bar x_s - \bar x_{s-1}}^2\right]. \label{ineq: U1}
\end{align}
For $s=0$, we provide a more careful estimation.
Following \eqref{eq: H_diff} and \eqref{ineq: H_diff_exp}, for 
\[
    z_{1}^{(i)}, u_{1, z}^{(i)}, h_{1, z}^{(i)} =  \texttt{Inner Loop}(z_0^{(i)}, \eta_{z}, g_i(x_0^{(i)}, \cdot), u_{0, z}^{(i)}, h_{0, z}^{(i)}, T)
\]
We know $H_0=0$, thus
\begin{align}\label{inequ: U1:initial}
    \E\left[\norm{\mH_1 - \mH_0}^2\right] &= \E\left[\norm{\mH_1}^2\right] 
   = \sum_{i=1}^{n}\E\left[\norm{\nabla_y g_i(x_0^{(i)}, z_0^{(i)};\xi_j)}^2\right] \notag \\
&     \leq  2n\sigma_z^2 +  2\sum_{i=1}^{n}\E\left[\norm{\nabla_y g_i(x_0^{(i)}, z_0^{(i)})}^2\right] := 2n (\sigma_z^2 + \ell_{f,0}^2).
\end{align}

\section{Appendix / Consensus and convergence analysis for \texorpdfstring{$Y, Z, X$}{Lg}}\label{sec:consensus_app}

As a direct result of Lemma \ref{lem: inner_error}, we first get the estimations for the consensus of $y$ and $z$.
\begin{restatable}[]{lemma}{lemsinneryzconverge}
\label{lem: yz_convergence}
    Suppose Assumptions \ref{aspt: smoothness}, \ref{aspt: so}, \ref{aspt: sgap}, and \ref{aspt: bdd_grad} hold, we have the following estimations for the consensus of $y$ and $z$:
    \begin{align}
        \frac{1}{2n}\sum_{s=0}^{S}\E\left[\norm{\mZ_{s} - z_{*,s}\bfonet_n}^2\right] & \leq  C_{z_*,0}\Delta_{z_*,0}+  C_{Z,0}\Delta_{Z,0}  + C_{U_z,0}\Delta_{U_z,0}  + C_{z,v} \sum_{s=0}^{S}\E\left[\norm{\E\left[\bar v_{s+1}\middle|\cF_s\right]}^2\right] \notag \\ & + C_{z,vs} \sum_{s=0}^{S}\E\left[\norm{\bar v_{s+1}}^2\right] 
 + C_{z,x} \frac{1}{n}\sum_{s=0}^{S}\E\left[\norm{\mX_{s} - \bar x_{s}\bfonet_n}^2\right] + S \cdot C_{z,\sigma}  \sigma_z^2 \notag \\
     \frac{1}{2n}\sum_{s=0}^{S}\E\left[\norm{\mY_{s} - y_{*,s}^{\alpha}\bfonet_n}^2\right] &   \leq  C_{y_*,0}\Delta_{y_*,0} +  C_{Y,0}\Delta_{Y,0} + C_{U_y,0}\Delta_{U_y,0}  + C_{y,v}\sum_{s=0}^{S}\E\left[\norm{\E\left[\bar v_{s+1}\middle|\cF_s\right]}^2\right] \notag \\ & + C_{y,vs} \sum_{s=0}^{S}\E\left[\norm{\bar v_{s+1}}^2\right]  
  + C_{y,x} \sum_{s=0}^{S}\frac{1}{n}\E\left[\norm{\mX_s - \bar{x}_s\bfonet_n}^2\right] + S C_{y,\sigma} \sigma_y^2  \notag 
    \end{align}
    where the constants $C_{z_*,0}, C_{Z,0}, C_{U_z,0}, \Delta_{z_*,0}, \Delta_{Z,0}, \Delta_{U_z,0}, C_{z,v}, C_{z,vs},C_{z,x},C_{z,\sigma}$ are defined in \eqref{inequ:z:const} and $C_{y_*,0}, C_{Y,0}, C_{U_y,0}, \Delta_{y_*,0}, \Delta_{Y,0}, \Delta_{U_y,0}, C_{y,v}, C_{y,vs}, C_{y,x}, C_{y,\sigma}$ are defined in \eqref{inequ:y:const}.
\end{restatable}
\begin{proof}(of Lemma~\ref{lem: yz_convergence})
We train the variables $y,z$ with  $T$-steps $b$-batch gradient descent for $b \geq 1$.
    Note that from Algorithm \ref{alg:stochastic_minmax:onestage} and by Lemma \ref{lem: gt} we know the updates of $\bar y_s, \bar z_s$ take the form
    \begin{align}
        &\bar y_{s+1} = \bar y_s - \eta_{y}\bar v_{s+1, y} = \bar y_s - \eta_{s,y}\bar \delta_{s+1, y}  \\
        &\bar z_{s+1} = \bar z_s - \eta_{z}\bar v_{s+1, z} = \bar z_s - \eta_{z}\bar \delta_{s+1, z}.
    \end{align}
The variable $z$ is to optimize the objective $g_i$ which is $\mu_g$-strongly convex and $\ell_{g,1}$-smooth. The stochastic gradient $h_{t+1,z}$ of updating $z$ is supposed to be variance-bounded by $\sigma_z^2 = \sigma_g^2$. By Lemma \ref{lem: inner_error} and Inequality \eqref{ineq: U1}, we know if the step-size $\eta_{z}$ satisfies that 
    \begin{align}
    \eta_{z} \leq \mathcal{O} \left(\min \left\lbrace {\frac{1-\rho^2}{\ell_{g,1}}}, \frac{(1-\rho^2)\sqrt{\mu_g}}{\ell_{g,1}\sqrt{\ell_{g,1}}}, \frac{(1-\rho^2)^2}{\ell_{g,1}}\right\rbrace \right)
    \end{align}
    then for $s\geq 1$
    \begin{align}
       &\E\left[\norm{\bar z_{s+1} - z_{*, s+1}}^2\right]  \notag \\
       & \leq  \left(1-\mu_g \eta_{z}\right)^T\left( 1+\frac{\eta_z^4\ell_{g,1}^6}{\mu_{g}^2(1-\rho^2)^4} \right) \E\left[\norm{\bar z_s - z_{*, s+1}}^2\right]   + \left(1-\mu_g \eta_{z}\right)^{T-1}\frac{\eta_z\ell_{g,1}^2}{\mu_g(1-\rho^2)}\frac{1}{n}\E\left[\norm{\mZ_{s} - \bar z_{s}\bfonet_n}^2\right] \notag \\  & +\min \left(T, \frac{1}{\mu_g\eta_{z}}\right)\left(\frac{C_{x,1}\eta_z}{n}\E\left[\norm{\mX_s - \bar x_s\bfonet_n}^2\right] + \frac{C_{\sigma,1}}{n}\eta_z^2\sigma_{z}^2  \right)   \notag \\
       & + \frac{\left(1-\mu_g \eta_{z}\right)^{T-1}\eta_z^3\ell_{g,1}^2}{\mu_g(1-\rho^2)^3n}\E\left[\norm{\mU_{s, z} - \bar u_{s,z}\bfonet_n}^2\right] +  \left(1-\mu_g \eta_{z}\right)^{T-1}\frac{\eta_z^3\ell_{g,1}^2}{\mu_g(1-\rho^2)^4}6\sigma_{z}^2\notag \\
        & + \frac{\left(1-\mu_g \eta_{z}\right)^{T-1}\eta_z^3\ell_{g,1}^2}{\mu_g(1-\rho^2)^4}\left( 9\ell_{g,1}^2\frac{1}{n}\E\left[\norm{\mX_s - \bar x_s\bfonet_n}^2 + \norm{\mX_{s-1} - \bar x_{s-1}\bfonet_n}^2 + n\norm{\bar x_s - \bar x_{s-1}}^2\right]\right) \notag \\
        &  \lesssim  e_z^T \E\left[\norm{\bar z_s - z_{*, s+1}}^2\right] + \frac{e_z^{T-1}\eta_z}{1-\rho^2} \frac{1}{n}\E\left[\norm{\mZ_{s} - \bar z_{s}\bfonet_n}^2\right]  +  \frac{e_z^{T-1}\eta_z^3}{(1-\rho^2)^3}\frac{1}{n}\E\left[\norm{\mU_{s, z} - \bar u_{s,z}\bfonet_n}^2\right] \notag \\
        & +  \min \left(T, \frac{1}{\mu_g\eta_{z}}\right)\left(\frac{C_{x,1}\eta_z}{n}\E\left[\norm{\mX_s - \bar x_s\bfonet_n}^2\right] + \frac{C_{\sigma,1}}{n}\eta_z^2\sigma_{z}^2  \right)  \notag \\
         & + \frac{e_z^{T-1}\eta_z^3}{(1-\rho^2)^4}\left(\frac{1}{n}\E\left[\norm{\mX_{s-1} - \bar x_{s-1}\bfonet_n}^2 + n\norm{\bar x_s - \bar x_{s-1}}^2\right]\right)
        \label{ineq: z_recur}
    \end{align}
    where  $e_z = 1-2\mu_g \eta_{z}/3$ and for simplicity we choose sufficient small $\eta_z$ and any $T \geq 1$ such that 
    \begin{align}
   \left(1-\mu_g \eta_{z}\right)^T\left( 1+\frac{\eta_z^4\ell_{g,1}^6}{\mu_{g}^2(1-\rho^2)^4} \right)  &\leq \left(1-\frac{2\mu_g\eta_z}{3} \right)^T \notag \\
  \left(1-\mu_g \eta_{z}\right)^{T-1} \frac{\eta_z^3\ell_{g,1}^4}{\mu_g(1-\rho^2)^4} & \leq  \frac{\eta_z^3\ell_{g,1}^4}{\mu_g(1-\rho^2)^4} \leq \min \left(T, \frac{1}{\mu_z\eta_z} \right)C_{x,1}\eta_z \notag \\
  \left(1-\mu_g \eta_{z}\right)^{T-1} \frac{\eta_z\ell_{g,1}^2  n}{\mu_g(1-\rho^2)^4} & \leq  \frac{\eta_z\ell_{g,1}^2  n}{\mu_g(1-\rho^2)^4} \leq C_{\sigma,1} 
    \end{align}
    with constant $C_{x, 1} = \mathcal{O}\left(1\right)$  and $C_{\sigma, 1} = \mathcal{O}\left(\eta_z n + 1\right)$ for variable $z$. 
    We also have
    \begin{align*}
        \norm{\bar z_s - z_{*, s+1}}^2
        = &\norm{\bar z_s - z_{*, s}}^2 + \norm{z_{*, s} - z_{*, s+1}}^2 + 2\<z_{*, s} - \bar z_s, z_{*, s+1} - z_{*, s}> \\
        \leq &\norm{\bar z_s - z_{*, s}}^2 + \eta_{x}^2\ell_{z_*}^2\norm{\bar v_{s+1}}^2 + 2\<z_{*, s} - \bar z_s, z_{*, s+1} - z_{*, s}>
    \end{align*}
    When we consider the convergence of the consensus convergence $\mZ$ and $\mU$, we only need the following inequality
    \begin{align}\label{inequ:alert:z1}
     \norm{\bar z_s - z_{*, s+1}}^2 \leq  2\norm{\bar z_s - z_{*, s}}^2 + 2\norm{z_{*, s} - z_{*, s+1}}^2 = 2\norm{\bar z_s - z_{*, s}}^2 + 2\eta_{x}^2\ell_{z_*}^2\norm{\bar v_{s+1}}^2.
    \end{align}
  To ensure convergence of variable $\bar z_s$, it is necessary to carefully estimate the cross-term. For any $a_1, a_2>0$, we have
    \begin{align*}
        &\E\left[2\<z_{*, s} - \bar z_s, z_{*, s+1} -  z_{*, s}>\middle|\cF_s\right] \notag \\
        & =  \E\left[2\<z_{*, s} - \bar z_s, \<\nabla z_*(\bar x_s), \bar x_{s+1} - \bar x_s>>\middle|\cF_s\right]  + \E\left[2\<z_{*, s} - \bar z_s, z_{*, s+1} -  z_{*, s} - \<\nabla z_*(\bar x_s), \bar x_{s+1} - \bar x_s>>\middle|\cF_s\right] \\
        \leq & 2\eta_{x}\ell_{z_*}\norm{\bar z_s - z_{*, s}} \norm{\E\left[\bar v_{s+1}\middle|\cF_s\right]}+ \E\left[2\norm{\bar z_s - z_{*, s}}\norm{z_{*, s+1} -  z_{*, s} - \<\nabla z_*(\bar x_s), \bar x_{s+1} - \bar x_s>}\middle|\cF_s\right]  \\
        \leq &  \eta_{x}\ell_{z_*}\left(a_1\norm{\bar z_s - z_{*, s}}^2 + \frac{1}{a_1}\norm{\E\left[\bar v_{s+1}\middle|\cF_s\right]}^2\right) + \ell_{\nabla z_*}\eta_{x}^2\E\left[\norm{\bar z_s - z_{*, s}}\norm{\bar v_{s+1}}^2\middle|\cF_s\right] \\ 
       \leq &  \eta_{x}\ell_{z_*}\left(a_1\norm{\bar z_s - z_{*, s}}^2 + \frac{1}{a_1}\norm{\E\left[\bar v_{s+1}\middle|\cF_s\right]}^2\right)   + \ell_{\nabla z_*}\eta_{x}^2\left(\frac{a_2}{2}\E\left[\norm{\bar z_s - z_{*, s}}^2\norm{\bar v_{s+1}}^2|\cF_s\right] + \frac{1}{2a_2}\E\left[\norm{\bar v_{s+1}}^2 \middle|\cF_s\right] \right) \notag \\
     \leq  & \eta_{x}\ell_{z_*}\left(a_1\norm{\bar z_s - z_{*, s}}^2 + \frac{1}{a_1}\norm{\E\left[\bar v_{s+1}\middle|\cF_s\right]}^2\right) + \ell_{\nabla z_*}\eta_{x}^2\left(\frac{a_2 c_{\delta} \alpha^2}{2}\norm{\bar z_s - z_{*, s}}^2  + \frac{1}{2a_2}\E\left[\norm{\bar v_{s+1}}^2 \middle|\cF_s\right] \right) 
    \end{align*}
    where the second inequality uses smoothness of $z^*(\cdot)$.
    Note that here we carefully analyze the cross term by using the method introduced in \cite{chen2021closing}. Note that this type of analysis utilizes Taylor expansion that leads to better error bound, and can be avoided by using the Moving-Average trick in \cite{chen2023optimal}. 
    Combining the above inequalities, we have
    \begin{align}
        \E\left[\norm{\bar z_s - z_{*, s+1}}^2\right]\leq &\left(1 + a_1\eta_{x}\ell_{z_*} + \frac{a_2\ell_{\nabla z_*}c_{\delta}\eta_{x}^2\alpha^2}{2}\right)\E\left[\norm{\bar z_s - z_{*,s}}^2\right] \notag\\
        &+ \frac{\eta_{x}\ell_{z^*}}{a_1}\E\left[\norm{\E\left[\bar v_{s+1}\middle|\cF_s\right]}^2\right]+ \eta_{x}^2 \left(\ell_{z_*}^2 + \frac{\ell_{\nabla z_*}}{2a_2} \right)\E\left[\norm{\bar v_{s+1}}^2\right] \label{ineq: alset_term}
    \end{align}
    Choosing $\eta_{x}, \eta_{z}$ and $T$ such that 
    \begin{align}\label{ineq: xzt}
        r_z^T = \left(1 + a_1\eta_{x}\ell_{z_*} + \frac{a_2\ell_{\nabla z_*}c_{\delta}\eta_{x}^2\alpha^2}{2}\right)\left(1-\frac{2\mu_g\eta_{z}}{3}\right)^T\leq \left(1 - \frac{\mu_g\eta_{z}}{3}\right)^T.
    \end{align}
    Combining \eqref{ineq: z_recur}, \eqref{ineq: alset_term} and \eqref{ineq: xzt} gives
    \begin{align}
& \E\left[\norm{\bar z_{s+1} - z_{*, s+1}}^2\right] \notag \\
\lesssim  & r_z^T \E\left[\norm{\bar z_s - z_{*, s}}^2\right]  +  \frac{e_z^{T-1}\eta_z}{1-\rho^2} \frac{1}{n}\E\left[\norm{\mZ_{s} - \bar z_{s}\bfonet_n}^2\right]  +  \frac{e_z^{T-1}\eta_z^3}{(1-\rho^2)^3}\frac{1}{n}\E\left[\norm{\mU_{s, z} - \bar u_{s,z}\bfonet_n}^2\right] \notag \\
+ &  e_z^T \frac{\eta_{x}\ell_{z^*}}{a_1}\E\left[\norm{\E\left[\bar v_{s+1}\middle|\cF_s\right]}^2\right] + e_z^T\eta_{x}^2 \left(\ell_{z_*}^2 + \frac{\ell_{\nabla z_*}}{2a_2} \right)\E\left[\norm{\bar v_{s+1}}^2\right]  + \frac{e_z^{T}\eta_z^3\eta_x^2}{(1-\rho^2)^4}\E\left[\norm{\bar v_s}^2\right]\notag 
\\
 + &  \min \left(T, \frac{1}{\mu_g\eta_{z}}\right)\left(\frac{C_{x,1}\eta_z}{n}\E\left[\norm{\mX_s - \bar x_s\bfonet_n}^2\right] + \frac{C_{\sigma,1}}{n}\eta_z^2\sigma_{z}^2  \right) 
        +    \frac{e_z^{T}\eta_z^3}{(1-\rho^2)^4}\frac{1}{n}\E\left[\norm{\mX_{s-1} - \bar x_{s-1}\bfonet_n}^2\right].\label{inequ:z:recur}
\end{align}
where $e_z = 1- 2\mu_g\eta_z/3$ and $r_z \leq 1-\mu_g\eta_z/3$. Recalling the result of Lemma~\ref{lem: inner_error} for the consensus of $z$ and incorporating Inequalities~\eqref{ineq: U1} and \eqref{inequ:alert:z1}, we have
\begin{align}
& \frac{1}{n}\E\left[\norm{\mZ_{s+1} - \bar z_{s+1}\bfonet_n}^2\right] \notag \\
\lesssim  &  \left(1 + \frac{\eta_z^2\ell_{g,1}^2}{(1-\rho^2)^4} \right)\frac{e_{\rho,1}^{T}}{n}\E\left[\norm{\mZ_{s} - \bar z_s\bfonet_n}^2\right] +  \frac{e_{\rho,1}^T\eta_z^3}{(1-\rho^2)^3}\E\left[\norm{\bar z_s - z_{*,s}}^2\right]   + \frac{e_{\rho,1}^T \eta_z^2}{(1-\rho^2)^2 n}\E\left[\norm{\mU_{s,z} - \bar u_{s,z}\bfonet_n}^2\right]   \notag \\
+ &  \frac{\eta_z^2}{(1-\rho^2)n}\left(C_{x,2}\eta_z^2+ \frac{e_{\rho,1}^T}{(1-\rho^2)^2} \right)\E\left[\norm{\mX_s - \bar{x}_s\bfonet_n}^2\right]  + \frac{e_{\rho,1}^T \eta_z^2}{(1-\rho^2)^3} \frac{1}{n}\E\left[\norm{\mX_{s-1} - \bar x_{s-1}\bfonet_n}^2\right] \notag \\
+ &    \eta_z^2\left(\frac{C_{\sigma,2}}{n}+\frac{e_{\rho,1}^T}{(1-\rho^2)^2} \right)\sigma_z^2 + \frac{e_{\rho,1}^T \eta_z^2}{(1-\rho^2)^3} \eta_x^2\E\left[\norm{\bar v_s}^2\right]  +   \frac{e_{\rho,1}^T\eta_z^3\eta_{x}^2 }{(1-\rho^2)^3}\E\left[\norm{\bar v_{s+1}}^2\right] \label{inequ: Z:recur}
\end{align}
for sufficient small $\eta_z \leq (1-\rho^2)^3/\ell_{g,1}$ such that $\left(1 + \frac{\eta_z^2\ell_{g,1}^2}{(1-\rho^2)^4} \right)e_{\rho,1} \leq \frac{3+\rho^2}{4}$, then $\left(1 + \frac{\eta_z^2\ell_{g,1}^2}{(1-\rho^2)^4} \right)e_{\rho,1}^T \leq e_{\rho,2}^T$ for any $T \geq 1$. Similarly, we recall the result of Lemma~\ref{lem: inner_error} for the consensus convergence $\mU$ of $z$  and incorporate Inequalities~\eqref{ineq: U1} and \eqref{inequ:alert:z1}
\begin{align}
& \frac{1}{n}\E\left[\norm{\mU_{s+1,z} - \bar u_{s+1,z}\bfonet_n}^2\right] \notag \\
\lesssim &   e_{\rho,2}^T \left(1 + \frac{\ell_{g,1}^2\eta_z^2}{(1-\rho^2)^4}\right)\frac{1}{n}\E\left[\norm{\mU_{s,z} - \bar u_{s,z}\bfonet_n}^2\right] + \frac{e_{\rho,2}^{T-1}\eta_z}{(1-\rho^2)}\E\left[\norm{\bar z_{s} - z_{*,s}}^2\right]  +  \frac{e_{\rho,2}^{T-1}}{(1-\rho^2)^3}\frac{1}{n}\E\left[\norm{\mZ_{s} - \bar z_{s}\bfonet_n}^2\right]  \notag \\ +  &   
 \left(C_{x,3}\eta_z^2  + \frac{e_{\rho,2}^T}{1-\rho^2}\right)\frac{1}{n}\E\left[\norm{\mX_s - \bar{x}_s\bfonet_n}^2\right] + \frac{e_{\rho,2}^T}{(1-\rho^2)n}\E\left[\norm{\mX_{s-1} - \bar x_{s-1}\bfonet_n}^2\right] \notag \\
 + &  \left(C_{\sigma, 3}  + \frac{e_{\rho,2}^T}{1-\rho^2}\right)\sigma_z^2 + \frac{e_{\rho,2}^T\eta_x^2}{1-\rho^2} \E\left[\norm{\bar v_s}^2\right] +\frac{e_{\rho,2}^{T-1}\eta_z\eta_{x}^2}{(1-\rho^2)}\E\left[\norm{\bar v_{s+1}}^2\right] \label{inequ: U_z:recur}
\end{align}
for sufficient small $\eta_z \leq (1-\rho^2)^3/\ell_{g,1}$ such that $r_{U,z} = e_{\rho,2} \left(1 + \frac{\ell_{g,1}^2\eta_z^2}{(1-\rho^2)^4}\right) \leq \frac{4+\rho^2}{5} < 1$ for any $T \geq 1$.
Here we use the same idea in Lemma~\ref{lem: inner_error} and define the vector function $\Omega_{Z,s}$:
\begin{align}
\Omega_{Z,s} = \left(\E\left[\norm{\bar z_s - z_{*, s}}^2\right],  \frac{1}{n}\E\left[\norm{\mZ_{s} - \bar z_s\bfonet_n}^2\right], \frac{1}{n}\E\left[\norm{\mU_{s,z} - \bar u_{s,z}\bfonet_n}^2\right]\right)
\end{align}
and an $3\times 3$ matrix $M_{Z}$
\begin{equation}
M_{Z} = 
\begin{pmatrix}
M_{11} & M_{12} & M_{13}\\
M_{21} & M_{22} & M_{23}\\
M_{31} & M_{32} & M_{33}\\
\end{pmatrix}
\end{equation}
where
\begin{align}
   &  M_{11} = r_z^T; \quad M_{12}=\frac{e_z^{T-1}\eta_z}{1-\rho^2}; \quad M_{13} = \frac{e_z^{T-1}\eta_z^3}{(1-\rho^2)^3} \notag \\
    & M_{21} = \frac{e_{\rho,1}^T\eta_z^3}{(1-\rho^2)^3};\quad M_{22} = e_{\rho,2}^T; \quad M_{23} = \frac{e_{\rho,1}^T \eta_z^2}{(1-\rho^2)^2} \notag \\
    & M_{31}= \frac{e_{\rho,2}^{T-1}\eta_z}{(1-\rho^2)}; \quad M_{32} = \frac{e_{\rho,2}^{T-1}}{(1-\rho^2)^3}; \quad M_{33} = r_{U,z}^T.
\end{align}
By the above inequalities \eqref{inequ:z:recur}, \eqref{inequ: Z:recur} and \eqref{inequ: U_z:recur}, we have
\begin{align}
\Omega_{Z,s+1} \leq M_{Z}\Omega_{Z,s} + \tilde{C}_{z,s}
\end{align}
where $\tilde{C}_{z,s} \in \R^3$ is defined by
\begin{align}
 \tilde{C}_{z,s}[1] 
& = \min \left(T, \frac{1}{\mu_g\eta_{z}}\right)\left(\frac{C_{x,1}\eta_z}{n}\E\left[\norm{\mX_s - \bar x_s\bfonet_n}^2\right] + \frac{C_{\sigma,1}}{n}\eta_z^2\sigma_{z}^2  \right) \notag \\
& + \frac{e_z^{T}\eta_z^3}{(1-\rho^2)^4}\frac{1}{n}\E\left[\norm{\mX_{s-1} - \bar x_{s-1}\bfonet_n}^2\right]   + e_z^T \frac{\eta_{x}\ell_{z^*}}{a_1}\E\left[\norm{\E\left[\bar v_{s+1}\middle|\cF_s\right]}^2\right] \notag \\
& + e_z^T\eta_{x}^2 \left(\ell_{z_*}^2 + \frac{\ell_{\nabla z_*}}{2a_2} \right)\E\left[\norm{\bar v_{s+1}}^2\right]  + \frac{e_z^{T}\eta_z^3\eta_x^2}{(1-\rho^2)^4}\E\left[\norm{\bar v_s}^2\right] \notag \\
\tilde{C}_{z,s}[2] 
& = \frac{\eta_z^2}{(1-\rho^2)n}\left(C_{x,2}\eta_z^2+ \frac{e_{\rho,1}^T}{(1-\rho^2)^2} \right)\E\left[\norm{\mX_s - \bar{x}_s\bfonet_n}^2\right]  \notag \\
& + \frac{e_{\rho,1}^T \eta_z^2}{(1-\rho^2)^3} \frac{1}{n}\E\left[\norm{\mX_{s-1} - \bar x_{s-1}\bfonet_n}^2\right] \notag \\
& + \eta_z^2\left(\frac{C_{\sigma,2}}{n}+\frac{e_{\rho,1}^T}{(1-\rho^2)^2} \right)\sigma_z^2 + \frac{e_{\rho,1}^T \eta_z^2}{(1-\rho^2)^3} \eta_x^2\E\left[\norm{\bar v_s}^2\right]  +   \frac{e_{\rho,1}^T\eta_z^3\eta_{x}^2 }{(1-\rho^2)^3}\E\left[\norm{\bar v_{s+1}}^2\right] \notag \\
 \tilde{C}_{z,s}[3] & = \left(C_{x,3}\eta_z^2  + \frac{e_{\rho,2}^T}{1-\rho^2}\right)\frac{1}{n}\E\left[\norm{\mX_s - \bar{x}_s\bfonet_n}^2\right] + \frac{e_{\rho,2}^T}{(1-\rho^2)n}\E\left[\norm{\mX_{s-1} - \bar x_{s-1}\bfonet_n}^2\right] \notag \\ &  + \left(C_{\sigma, 3}  + \frac{e_{\rho,2}^T}{1-\rho^2}\right)\sigma_z^2 
+ \frac{e_{\rho,2}^T\eta_x^2}{1-\rho^2} \E\left[\norm{\bar v_s}^2\right] +\frac{e_{\rho,2}^{T-1}\eta_z\eta_{x}^2}{(1-\rho^2)}\E\left[\norm{\bar v_{s+1}}^2\right]. \notag
\end{align}
For simplicity, we also overload the notation and set $\Omega_{Z,s} = (a_s, b_s, c_s)^\top$ and $\tilde{C}_{z,s} = (d_{1,s}, d_{2,s}, d_{3,s})^\top$.
Note that we have
\begin{align}
a_{s+1} & \leq M_{11}a_s + M_{12}b_s + M_{13}c_s + d_{1,s}  \notag \\
b_{s+1} & \leq M_{21}a_s + M_{22}b_s + M_{23}c_s + d_{2,s} \notag \\
c_{s+1} & \leq  M_{31}a_s + M_{32}b_s + M_{33}c_s + d_{3,2}, \notag 
\end{align}
thus we apply Lemma \ref{lem: basic_ineq} (\eqref{ineq: aksum} to $a_s, b_s, c_s$ and let $\tau_k = 1$)
\begin{align}
\sum_{i=0}^s a_i & \leq \frac{1}{1-M_{11}}\left( a_0 + M_{12}\sum_{i=0}^s b_i + M_{13}\sum_{i=0}^s c_i + \sum_{i=0}^s d_{1,i}\right) \tag{$z:a$} \label{inequ:z:a}\\
\sum_{i=0}^s b_i & \leq \frac{1}{1-M_{22}}\left( b_0 + M_{21}\sum_{i=0}^sa_i + M_{23}\sum_{i=0}^sc_i + \sum_{i=0}^sd_{2,i}\right) \tag{$z:b$} \label{inequ:z:b}\\
\sum_{i=0}^s c_i & \leq \frac{1}{1-M_{33}}\left( c_0 + M_{31}\sum_{i=0}^sa_i + M_{32}\sum_{i=0}^sb_i + \sum_{i=0}^sd_{3,i}\right).  \tag{$z:c$} \label{inequ:z:c}
\end{align}
Incorporating \eqref{inequ:z:c} into \eqref{inequ:z:a} and \eqref{inequ:z:b}, let $Q_1=\frac{M_{13}}{1-M_{11}}\frac{1}{1-M_{33}}$, we have
\begin{align}
 \left(1 - Q_1M_{31}\right)\sum_{i=0}^s a_i 
& \leq \frac{a_0}{1-M_{11}}  +  c_0 Q_1  + \left(Q_1M_{32} + \frac{M_{12}}{1-M_{11}}\right)\sum_{i=0}^sb_i \notag \\
& + Q_1\sum_{i=0}^sd_{3,i}  + \frac{1}{1-M_{11}}\sum_{i=0}^sd_{1,i}. \tag{$z:a'$} \label{inequ:z:a'}
\end{align}
Let $Q_2 = \frac{M_{23}}{1-M_{22}}\frac{1}{1-M_{33}}$, we have
\begin{align}
\left( 1- Q_2M_{32}\right) \sum_{i=0}^s b_i & \leq  \frac{b_0}{1-M_{22}} + Q_2c_0 + \left(Q_2M_{31} + \frac{M_{21}}{1-M_{22}}\right)\sum_{i=0}^sa_i  \notag \\
 & + Q_2\sum_{i=0}^sd_{3,i} + \frac{1}{1-M_{22}}\sum_{i=0}^sd_{2,i}. \tag{$z:b'$} \label{inequ:z:b'}
\end{align}
Then we make the operations on the sum of $a_s$ and $b_s$: that $\left(Q_2M_{31} + \frac{M_{21}}{1-M_{22}}\right) \times \eqref{inequ:z:a'} + \left(1 - Q_1M_{31}\right)\times \eqref{inequ:z:b'}$, then 
\begin{align}
& \left(\left(1 - Q_1M_{31}\right)\left( 1- Q_2M_{32}\right)-  \left(Q_2M_{31} + \frac{M_{21}}{1-M_{22}}\right)\left(Q_1M_{32} + \frac{M_{12}}{1-M_{11}}\right)\right) \sum_{i=0}^s b_i \notag \\
& \leq \left(Q_2M_{31} + \frac{M_{21}}{1-M_{22}}\right)\left(\frac{a_0}{1-M_{11}}  +  c_0 Q_1\right) + \left(1 - Q_1M_{31}\right)\left(\frac{b_0}{1-M_{22}} + Q_2c_0\right) \notag \\
& + \left(Q_1Q_2M_{31} + \frac{M_{21}Q_1}{1-M_{22}} +Q_2 \left(1 - Q_1M_{31}\right) \right)\sum_{i=0}^sd_{3,i} +  \frac{\left(1 - Q_1M_{31}\right)}{1-M_{22}} \sum_{i=0}^sd_{2,i} \notag \\
& + \frac{\left(Q_2M_{31} + \frac{M_{21}}{1-M_{22}}\right)}{1-M_{11}}\sum_{i=0}^sd_{1,i}.
\end{align}
Note that to simplify the calculations and also cover the two cases: $T=1$ and $T \gg 1 $, we set  
\begin{align*}
\frac{1}{1 - M_{11}} & = \frac{1}{1- r_z^T}  \sim \max\left(\frac{1}{\mu\gamma}, 2 \right); \notag \\
\frac{1}{1- M_{22}} & \sim \frac{1}{1-M_{33}} \sim \frac{1}{1 - \left(\frac{3 + \rho^2}{4}\right)^T} \sim \max \left(\frac{1}{1-\rho^2},2\right).
\end{align*}
Then 
\begin{align}
Q_1 & = \frac{M_{13}}{1-M_{11}}\frac{1}{1-M_{33}} \sim \Theta \left( \frac{e_{z}^T\eta_z^3}{(1-\rho^2)^3}\max\left( \frac{1}{\mu_g\eta_z(1-\rho^2)},4\right)\right), \notag \\
Q_2 & = \frac{M_{23}}{1-M_{22}}\frac{1}{1-M_{33}} \sim \Theta\left(\frac{e_{\rho,1}^T\eta_z^2}{(1-\rho^2)^2}\max\left( \frac{1}{(1-\rho^2)^2},4\right) \right). \notag 
\end{align}
For any $T\geq 1$ we choose sufficient small stepsize $\eta_z \leq \min\left((1-\rho^2)^{3.5}, (1-\rho^2)\mu_g\right)$ such that 
\begin{align}
 1 - Q_1M_{31} \sim  1- \Theta \left(\max\left(\frac{\eta_z^3}{\mu_g(1-\rho^2)^5}, \frac{\eta_z^4}{(1-\rho^2)^4}\right) \right)  & \geq \frac{2}{3},    \notag \\
 1- Q_2M_{23} \sim 1- \Theta\left( \max \left(\frac{e_{\rho,1}^{2T}\eta_z^2}{(1-\rho^2)^7}, \frac{e_{\rho,1}^{2T}\eta_z^2}{(1-\rho^2)^5}\right) \right) & \geq \frac{2}{3}, \notag \\
\left(\left(1 - Q_1M_{31}\right)\left( 1- Q_2M_{32}\right)-  \left(Q_2M_{31} + \frac{M_{21}}{1-M_{22}}\right)\left(Q_1M_{32} + \frac{M_{12}}{1-M_{11}}\right)\right)  & \geq \frac{1}{3}, \notag 
\end{align}
we have
\begin{align}
\sum_{i=0}^s b_i 
& \leq 3\left(Q_2M_{31} + \frac{M_{21}}{1-M_{22}}\right)\left(\frac{a_0}{1-M_{11}}  +  c_0 Q_1\right) + 3\left(1 - Q_1M_{31}\right)\left(\frac{b_0}{1-M_{22}} + Q_2c_0\right) \notag \\
& + 3\left(Q_1Q_2M_{31} + \frac{M_{21}Q_1}{1-M_{22}} +Q_2 \left(1 - Q_1M_{31}\right) \right)\sum_{i=0}^sd_{3,i} +  3\frac{1 - Q_1M_{31}}{1-M_{22}} \sum_{i=0}^sd_{2,i} \notag \\ 
& + 3\frac{\left(Q_2M_{31} + \frac{M_{21}}{1-M_{22}}\right)}{1-M_{11}}\sum_{i=0}^sd_{1,i} \notag \\
& \lesssim \frac{e_{\rho,1}^T\eta_z^2}{\mu_g(1-\rho^2)^4} a_0 + \max\left(\frac{1}{1-\rho^2},2 \right)b_0 + \frac{e_{\rho,1}^T\eta_z^2}{(1-\rho^2)^4} c_0 + \frac{e_{\rho,1}^T\eta_z^2}{(1-\rho^2)^4} \sum_{i=0}^sd_{3,i} \notag \\
& + \max\left(\frac{1}{1-\rho^2},2\right)\sum_{i=0}^sd_{2,i} + \left( \frac{1}{\mu_g} + \frac{e_{\rho,2}^T}{(1-\rho^2)}\right)\frac{e_{\rho,1}^T\eta_z^2}{(1-\rho^2)^4}\sum_{i=0}^sd_{1,i}.
\label{inequ:b:iter:sum:1}
\end{align}
Then we substitute the above result w.r.t. $b_i$ to Inequality~\eqref{inequ:z:a'}, then
\begin{align}
\sum_{i=0}^s a_i 
& \leq \frac{3}{2}\frac{a_0}{1-M_{11}}  +   \frac{3}{2}c_0 Q_1  +  \frac{3}{2}\left(Q_1M_{32} + \frac{M_{12}}{1-M_{11}}\right)\sum_{i=0}^sb_i \notag \\
& +  \frac{3}{2}Q_1\sum_{i=0}^sd_{3,i}  +  \frac{3}{2}\frac{1}{1-M_{11}}\sum_{i=0}^sd_{1,i} \notag \\
& \lesssim \frac{a_0}{1-M_{11}} +c_0 Q_1 + \left(Q_1M_{32} + \frac{M_{12}}{1-M_{11}}\right)\left(Q_2M_{31} + \frac{M_{21}}{1-M_{22}}\right)\left(\frac{a_0}{1-M_{11}}  +  c_0 Q_1\right) \notag \\
& + \left(Q_1M_{32} + \frac{M_{12}}{1-M_{11}}\right)\left(1 - Q_1M_{31}\right)\left(\frac{b_0}{1-M_{22}} + Q_2c_0\right) \notag \\
& + \left[\left(Q_1M_{32} + \frac{M_{12}}{1-M_{11}}\right)\left(Q_1Q_2M_{31} + \frac{M_{21}Q_1}{1-M_{22}} +Q_2 \left(1 - Q_1M_{31}\right) \right) + Q_1\right]\sum_{i=0}^sd_{3,i} \notag \\
& + \left(Q_1M_{32} + \frac{M_{12}}{1-M_{11}}\right)\frac{1 - Q_1M_{31}}{1-M_{22}} \sum_{i=0}^sd_{2,i} \notag \\
& + \left[\left(Q_1M_{32} + \frac{M_{12}}{1-M_{11}}\right)\frac{\left(Q_2M_{31} + \frac{M_{21}}{1-M_{22}}\right)}{1-M_{11}} + \frac{1}{1-M_{11}}\right]\sum_{i=0}^sd_{1,i} \notag \\
& \lesssim \max\left(\frac{1}{\mu_g\eta_z},2 \right)a_0 +  \frac{e_{z}^T}{\mu_g(1-\rho^2)}b_0 + \frac{e_z^T\eta_z^2}{\mu_g(1-\rho^2)^5} c_0 + \frac{e_{z}^T\eta_z^2}{\mu_g(1-\rho^2)^5}\sum_{i=0}^sd_{3,i} \notag \\
& + \frac{e_{z}^T}{\mu_g(1-\rho^2)^2} \sum_{i=0}^sd_{2,i} + \max\left(\frac{1}{\mu_g\eta_z}, 2 \right)\sum_{i=0}^sd_{1,i}.
\label{inequ:b:iter:final:1}
\end{align}
We then substitute the definitions of $a_t, b_t$ and the matrix $M_{Z}$ and $d_{1,s},d_{2,s},d_{3,s}$. Note that 
\begin{align}
\frac{1}{2n}\E\left[\norm{\mZ_{s} - z_{*,s}\bfonet_n}^2\right] \leq \E\left[\norm{\bar z_{s} - z_{*, s}}^2 + \frac{1}{n}\norm{\mZ_{s} - \bar z_{s}\bfonet_n}^2 \right]:= a_{s} + b_{s} 
\end{align}
Now we can give the estimation for the sum of $z$:
\begin{align*}
& \frac{1}{2n}\sum_{s=0}^{S}\E\left[\norm{\mZ_{s} - z_{*,s}\bfonet_n}^2\right] := \sum_{s=0}^{S} \left( a_{s} + b_{s} \right) \notag \\
 & \lesssim \frac{e_{\rho,1}^T\eta_z^2}{\mu_g(1-\rho^2)^4} a_0 + \max\left(\frac{1}{1-\rho^2},2 \right)b_0 + \frac{e_{\rho,1}^T\eta_z^2}{(1-\rho^2)^4} c_0 + \frac{e_{\rho,1}^T\eta_z^2}{(1-\rho^2)^4} \sum_{i=0}^sd_{3,i} \notag \\
& + \max\left(\frac{1}{1-\rho^2},2\right)\sum_{i=0}^sd_{2,i} + \left( \frac{1}{\mu_g} + \frac{e_{\rho,2}^T}{(1-\rho^2)}\right)\frac{e_{\rho,1}^T\eta_z^2}{(1-\rho^2)^4}\sum_{i=0}^sd_{1,i} \notag \\
& +\max\left(\frac{1}{\mu_g\eta_z},2 \right)a_0 + \frac{e_{z}^T}{\mu_g(1-\rho^2)}b_0 + \frac{e_z^T\eta_z^2}{\mu_g(1-\rho^2)^5} c_0 + \frac{e_{z}^T\eta_z^2}{\mu_g(1-\rho^2)^5}\sum_{i=0}^sd_{3,i} \notag \\
& + \frac{e_{z}^T}{\mu_g(1-\rho^2)^2} \sum_{i=0}^sd_{2,i} + \max\left(\frac{1}{\mu_g\eta_z}, 2 \right)\sum_{i=0}^sd_{1,i} \notag \\
& \lesssim \max \left(\frac{e_{\rho,1}^T\eta_z^2}{\mu_g(1-\rho^2)^4}, \frac{1}{\mu_g\eta_z}\right)a_0 + \max\left(\frac{1}{1-\rho^2},\frac{e_{z}^T}{\mu_g(1-\rho^2)} \right)b_0 \notag \\
& + \max \left(e_{\rho,1}^T,\frac{e_z^T}{\mu_g(1-\rho^2)} \right)\frac{\eta_z^2 c_0}{(1-\rho^2)^4}  + \max\left(e_{\rho,1}^T, \frac{e_z^T}{\mu_g(1-\rho^2)}\right)\frac{\eta_z^2}{(1-\rho^2)^4}\sum_{s=0}^Sd_{3,s} \notag \\
& + \max\left(1,\frac{e_{z}^T}{\mu_g} \right)\frac{\sum_{s=0}^Sd_{2,i}}{(1-\rho^2)^2} + \max\left( \frac{1}{\mu_g\eta_z}, 2\right)\sum_{s=0}^Sd_{1,s}. 
\end{align*}
We thus substitute the definition of $a_0, b_0, c_0, d_{1,s}, d_{2,s}, d_{3,s}$ then
\begin{align}
& \frac{1}{2n}\sum_{s=0}^{S}\E\left[\norm{\mZ_{s} - z_{*,s}\bfonet_n}^2\right] \notag \\
&  \lesssim \max \left(\frac{e_{\rho,1}^T\eta_z^2}{\mu_g(1-\rho^2)^4}, \frac{1}{\mu_g\eta_z}\right)\E\left[\norm{\bar z_0 - z_{*, 0}}^2\right] +  \max\left(\frac{1}{1-\rho^2},\frac{e_{z}^T}{\mu_g(1-\rho^2)} \right)\frac{1}{n}\E\left[\norm{\mZ_{0} - \bar z_0\bfonet_n}^2\right]   \notag \\
& + \max \left(e_{\rho,1}^T,\frac{e_z^T}{\mu_g(1-\rho^2)} \right)\frac{\eta_z^2}{(1-\rho^2)^4}\frac{1}{n}\E\left[\norm{\mU_{0,z} - \bar u_{0,z}\bfonet_n}^2\right] \notag \\
& + \max\left(e_{\rho,1}^T, \frac{e_z^T}{\mu_g(1-\rho^2)}\right)\frac{\eta_z^2}{(1-\rho^2)^4} \Biggl\{ \left(C_{x,3}\eta_z^2  + \frac{e_{\rho,2}^T}{1-\rho^2}\right)\sum_{s=0}^S\frac{1}{n}\E\left[\norm{\mX_s - \bar{x}_s\bfonet_n}^2\right] \notag \\
& + \frac{e_{\rho,2}^T}{(1-\rho^2)n}\sum_{s=0}^{S-1}\E\left[\norm{\mX_{s} - \bar x_{s}\bfonet_n}^2\right] + S\left(C_{\sigma, 3}  + \frac{e_{\rho,2}^T}{1-\rho^2}\right)\sigma_z^2 
 + \frac{e_{\rho,2}^T\eta_x^2}{1-\rho^2} \sum_{s=0}^{S-1}\E\left[\norm{\bar v_{s+1}}^2\right] \notag \\
 & +\frac{e_{\rho,2}^{T-1}\eta_z\eta_{x}^2}{(1-\rho^2)}\sum_{s=0}^S\E\left[\norm{\bar v_{s+1}}^2\right] \Biggl\} + \frac{e_{z}^T}{\mu_g(1-\rho^2)^2}\Biggl\{ \frac{\eta_z^2\left(C_{x,2}\eta_z^2+ \frac{e_{\rho,1}^T}{(1-\rho^2)^2} \right)}{(1-\rho^2)}\sum_{s=0}^{S}\frac{1}{n}\E\left[\norm{\mX_s - \bar{x}_s\bfonet_n}^2\right] \notag \\
 & + \frac{e_{\rho,1}^T \eta_z^2}{(1-\rho^2)^3} \frac{1}{n}\sum_{s=0}^{S-1}\E\left[\norm{\mX_{s} - \bar x_{s}\bfonet_n}^2\right]  + S\eta_z^2\left(\frac{C_{\sigma,2}}{n}+\frac{e_{\rho,1}^T}{(1-\rho^2)^2} \right)\sigma_z^2 \notag \\
&+ \frac{e_{\rho,1}^T \eta_z^2}{(1-\rho^2)^3} \eta_x^2\sum_{s=0}^{S-1}\E\left[\norm{\bar v_{s+1}}^2\right]  +   \frac{e_{\rho,1}^T\eta_z^3\eta_{x}^2 }{(1-\rho^2)^3}\sum_{s=0}^{S}\E\left[\norm{\bar v_{s+1}}^2\right] \Biggl\} \notag \\
& +\max\left( \frac{1}{\mu_g\eta_z}, 2\right) \Biggl\{ \min \left(T, \frac{1}{\mu_g\eta_{z}}\right)\left(\frac{C_{x,1}\eta_z}{n}\sum_{s=0}^{S}\E\left[\norm{\mX_s - \bar x_s\bfonet_n}^2\right] + S\frac{C_{\sigma,1}}{n}\eta_z^2\sigma_{z}^2  \right)  \notag \\
& + \frac{e_z^{T}\eta_z^3}{(1-\rho^2)^4}\sum_{s=0}^{S-1}\frac{1}{n}\E\left[\norm{\mX_{s} - \bar x_{s}\bfonet}^2\right] + e_z^T \frac{\eta_{x}\ell_{z^*}}{a_1}\sum_{s=0}^{S}\E\left[\norm{\E\left[\bar v_{s+1}\middle|\cF_s\right]}^2\right] \notag \\
& + e_z^T\eta_{x}^2 \left(\ell_{z_*}^2 + \frac{\ell_{\nabla z_*}}{2a_2} \right)\sum_{s=0}^{S}\E\left[\norm{\bar v_{s+1}}^2\right]  + \frac{e_z^{T}\eta_z^3\eta_x^2}{(1-\rho^2)^4}\sum_{s=0}^{S-1}\E\left[\norm{\bar v_{s+1}}^2\right] \Biggl\} \notag \\
& \lesssim C_{z_*,0}\Delta_{z_*,0}+  C_{Z,0}\Delta_{Z,0}  + C_{U_z,0}\Delta_{U_z,0}  + C_{z,v} \sum_{s=0}^{S}\E\left[\norm{\E\left[\bar v_{s+1}\middle|\cF_s\right]}^2\right]  \notag \\ &+ C_{z,vs} \sum_{s=0}^{S-1}\E\left[\norm{\bar v_{s+1}}^2\right]
 + C_{z,x} \frac{1}{n}\sum_{s=0}^{S-1}\E\left[\norm{\mX_{s} - \bar x_{s}\bfonet}^2\right] + S \cdot C_{z,\sigma}  \sigma_z^2 \label{ineq: z_recur_final}
\end{align}
where the constants are given by
\begin{align}
& e_z   = 1 - \frac{2\mu_g\eta_z}{3}, r_z \leq  1 - \frac{\mu_g\eta_z}{3}, e_{\rho,1} = \frac{\rho^2+1}{2}, a_1 >0, a_2 >0, \notag \\
& \Delta_{z_*,0}  = \E\left[\norm{\bar z_0 - z_{*, 0}}^2\right],  \Delta_{Z,0} = \frac{1}{n}\E\left[\norm{\mZ_{0} - \bar z_0\bfonet_n}^2\right], \notag \\
& \Delta_{U_z,0}   = \frac{1}{n}\E\left[\norm{\mU_{0,z} - \bar u_{0,z}\bfonet_n}^2\right] = \cO\left(1\right),\notag \\
 & C_{z_*,0} = \max \left(\frac{e_{\rho,1}^T\eta_z^2}{\mu_g(1-\rho^2)^4}, \frac{1}{\mu_g\eta_z}\right)  = \cO\left(\frac{1}{\eta_z}\right), C_{Z,0} = \max\left(\frac{1}{1-\rho^2},\frac{e_{z}^T}{\mu_g(1-\rho^2)} \right) = \cO\left(\frac{1}{1-\rho^2}\right),\notag \\
& C_{U_z,0}  = \max \left(e_{\rho,1}^T,\frac{e_z^T}{\mu_g(1-\rho^2)} \right)\frac{\eta_z^2}{(1-\rho^2)^4} = \cO\left(\frac{\eta_z^2}{(1-\rho^2)^{5}}\right), \notag \\
& C_{z,v}  = \max\left( \frac{1}{\mu_g\eta_z}, 2\right) e_z^T \frac{\eta_{x}\ell_{z^*}}{a_1 } =  \cO\left(\frac{e_z^T\eta_x}{a_1 \eta_z} \right),\notag \\
& C_{z,vs}   = \cO\left(\frac{e_z^T\eta_{x}^2}{\eta_z}\left(1 + \frac{1}{a_2} \right)  + \frac{e_z^{T}\eta_z^2\eta_x^2}{(1-\rho^2)^4} \right), \notag \\
& C_{z,x}   = \cO\left(\frac{e_{z}^T\eta_z^4}{(1-\rho^2)^7} + \frac{e_{\rho,2}^Te_{z}^T\eta_z^2}{(1-\rho^2)^5} + \min \left(T, \frac{1}{\eta_z} 
 \right)\right), \notag \\
& C_{z,\sigma} =   \cO\left(\left(\frac{\eta_z}{n} + \frac{\eta_z^2}{(1-\rho^2)^4} \right)\min\left(T, \frac{1}{\eta_z}\right) + \frac{e_{z}^T\eta_z^2}{(1-\rho^2)^2} \right). \label{inequ:z:const}
\end{align}

Following the same reasoning in \eqref{ineq: z_recur}, \eqref{ineq: alset_term}, \eqref{inequ: Z:recur}, \eqref{ineq: xzt}, \eqref{inequ: U_z:recur},  and \eqref{ineq: z_recur_final} we may obtain a similar conclusion for $y$. The variable $y$ is to optimize the objective, $f + \alpha g$ with respect to $y$ which is $\frac{\alpha\mu_g}{2}$-strongly convex and $\frac{3\alpha \ell_{g,1}}{2}$-smooth. The stochastic gradient $h_{t+1,y}$ of updating $y$ is variance-bounded by $\sigma_{y}^2 = \sigma_f^2 + \alpha^2 \sigma_g^2$. Let $\mu =\frac{\alpha\mu_g}{2}, \ell=\frac{3\alpha \ell_{g,1}}{2}$ in Lemma~\ref{lem: inner_error}. If the step-size $\eta_{y}$ satisfies that 
    \begin{align}
    \eta_{y} \leq \mathcal{O} \left(\frac{1}{\alpha}\min \left\lbrace {\frac{1-\rho^2}{\ell_{g,1}}}, \frac{(1-\rho^2)\sqrt{\mu_g}}{\ell_{g,1}\sqrt{\ell_{g,1}}}, \frac{(1-\rho^2)^2}{\ell_{g,1}}\right\rbrace \right),
    \end{align}
 we have
\begin{align}
& \E\left[\norm{\bar y_{s+1} - y_{*, s+1}^{\alpha}}^2 \right] \notag \\
 & \leq \left(1-\frac{\alpha\mu_g\eta_y}{2} \right)^T \left( 1+ \frac{\alpha^4\eta_y^4\ell_{g,1}^6}{\mu_g^2(1-\rho^2)^4} \right)\E\left[\norm{\bar y_s - y_{*,s+1}^{\alpha}}^2\right] \notag \\
 & + \frac{\left(1-\frac{\alpha\mu_g\eta_y}{2} \right)^T\alpha\eta_y\ell_{g,1}^2}{\mu_g(1-\rho^2)}\frac{1}{n}\E\left[\norm{\mY_s - \bar y_s\bfonet_n}^2\right]  \notag \\  & +  \frac{\left(1-\frac{\alpha\mu_g\eta_y}{2} \right)^T\alpha\eta_y^3\ell_{g,1}^2}{\mu_g(1-\rho^2)^3n}\E\left[\norm{\mU_{s,y} - \bar u_{s,y}\bfonet}^2\right]  +   \min \left(T, \frac{1}{\alpha\mu_g\eta_y} \right)  \frac{C_{x,1}\eta_y}{n}\E\left[\norm{\mX_s - \bar{x}_s\bfonet_n}^2\right]  \notag \\ 
& +  \min \left(T, \frac{1}{\alpha\mu_g\eta_y} \right)\frac{C_{\sigma,1}}{n} \eta_y^2 \sigma_y^2  + \frac{\left(1-\frac{\alpha\mu_g\eta_y}{2} \right)^T\alpha\eta_y^3\ell_{g,1}^2}{\mu_g(1-\rho^2)^3} \cdot 6\sigma_{y}^2 \notag \\
  & + \frac{\left(1-\frac{\alpha\mu_g\eta_y}{2} \right)^T\alpha\eta_y^3\ell_{g,1}^2}{\mu_g(1-\rho^2)^3} \left( \alpha^2\ell_{g,1}^2\frac{1}{n}\E\left[\norm{\mX_s - \bar x_s\bfonet}^2 + \norm{\mX_{s-1} - \bar x_{s-1}\bfonet}^2 + n\norm{\bar x_s - \bar x_{s-1}}^2\right]\right) \notag \\
  & \lesssim e_y^{T} \E\left[\norm{\bar y_s - y_{*,s+1}^{\alpha}}^2\right] + \frac{e_y^{T} \alpha \eta_y}{1-\rho^2}\frac{1}{n}\E\left[\norm{\mY_s - \bar y_s\bfonet_n}^2\right] + \frac{e_y^T\alpha\eta_y^3}{(1-\rho^2)^3}\frac{1}{n}\E\left[\norm{\mU_{s,y} - \bar u_{s,y}\bfonet}^2\right] \notag \\
  &  +  \min \left(T, \frac{1}{\alpha\mu_g\eta_y} \right) \left( \frac{C_{x,1}\eta_y}{n}\E\left[\norm{\mX - \bar{x}\bfonet_n}^2\right] + \frac{C_{\sigma,1}}{n} \eta_y^2 \sigma_y^2   \right)  \notag \\
 & + \frac{e_y^T\alpha^3\eta_y^3}{(1-\rho^2)^3}\left(\frac{1}{n}\E\left[\norm{\mX_{s-1} - \bar{x}_{s-1}\bfonet_n}^2\right] + \eta_x^2\E\left[\norm{\bar{v}_s}^2\right]  \right)   \end{align}
where $e_{y} = 1-\frac{\alpha\mu_g\eta_y}{3}$ and for simplicity we choose sufficient small $\eta_y \leq \cO\left( (1-\rho^2)\mu_g/(\alpha\ell_{g,1}^2)\right)$ and for any $T \geq 1$ such that 
\begin{align}
&\left(1-\frac{\alpha\mu_g\eta_y}{2} \right)^T \left( 1+ \frac{\alpha^4\eta_y^4\ell_{g,1}^6}{\mu_g^2(1-\rho^2)^4} \right)  \leq \left(1-\frac{\alpha\mu_g\eta_y}{3} \right)^T \notag \\
 & \frac{\left(1-\frac{\alpha\mu_g\eta_y}{2} \right)^T\alpha^3\eta_y^3\ell_{g,1}^4}{\mu_g(1-\rho^2)^3}  \leq \min\left(T, \frac{1}{\alpha\mu_g\eta_y} \right) C_{x,1}\eta_y,  \text{where}\,\, C_{x,1} \sim \cO\left(\frac{\alpha\ell_{g,1}^2}{\mu_g} + \frac{\alpha^5\eta_y^4\ell_{g,1}^6}{\mu_g(1-\rho^2)^4} \right) \notag \\
 & \frac{\left(1-\frac{\alpha\mu_g\eta_y}{2} \right)^T\alpha\eta_y\ell_{g,1}^2}{\mu_g(1-\rho^2)^3} \leq \frac{\alpha\eta_y\ell_{g,1}^2}{\mu_g(1-\rho^2)^3} \leq \frac{C_{\sigma,1}}{n}, \text{where}\,\, C_{\sigma,1}\sim \cO\left(\frac{\alpha\eta_y \ell_{g,1}^2n}{\mu_g(1-\rho^2)^4} +1\right).
\end{align}
 Recalling the inequality~\eqref{inequ:alert:z1} for $z$, we also have a similar result for $y$.  When we consider the convergence of the consensus convergence $\mY_s$ and $\mU_{s,y}$, we only need the following inequality
    \begin{align}\label{inequ:alert:y1}
     \norm{\bar y_s - y_{*, s+1}^{\alpha}}^2 \leq  2\norm{\bar y_s - y_{*, s}^{\alpha}}^2 + 2\norm{y_{*, s}^{\alpha} - y_{*, s+1}^{\alpha}}^2 = 2\norm{\bar y_s - y_{*, s}^{\alpha}}^2 + 2\eta_{x}^2\ell_{y_*}^2\norm{\bar v_{s+1}}^2.
    \end{align}
 For the convergence of variable $\bar y_s$, we need a careful estimate about $ \norm{\bar y_s - y_{*, s+1}^{\alpha}}^2$ just as $\bar{z}$,
    \begin{align}
        \E\left[\norm{\bar y_s - y_{*, s+1}^{\alpha}}^2\right]\leq &\left(1 + a_1\eta_{x}\ell_{y_*} + \frac{a_2\ell_{\nabla y_*}c_{\delta}\eta_{x}^2\alpha^2}{2}\right)\E\left[\norm{\bar y_s - y_{*,s}^{\alpha}}^2\right] \notag\\
        &+ \frac{\eta_{x}\ell_{y^*}}{a_1}\E\left[\norm{\E\left[\bar v_{s+1}\middle|\cF_s\right]}^2\right] + \eta_{x}^2 \left(\ell_{y_*}^2 + \frac{\ell_{\nabla y_*}}{2a_2} \right)\E\left[\norm{\bar v_{s+1}}^2\right] \label{ineq: alset_term:y}
    \end{align}
 and we properly choose 
    $\eta_{x}, \eta_{y}$ and $T$ such that 
    \begin{align}\label{ineq: yzt}
       r_y =  \left(1 + a_1\eta_{x}\ell_{y_*} + \frac{a_2\ell_{\nabla y_*}c_{\delta}\eta_{x}^2\alpha^2}{2}\right)\left(1-\frac{\mu_g\eta_{y}\alpha}{3}\right)^T\leq 1 - \frac{\mu_g\eta_{y}\alpha}{6}.
    \end{align}
Combining the above inequalities, we have
\begin{align}
& \E\left[\norm{\bar y_{s+1} - y_{*, s+1}^{\alpha}}^2 \right] \notag \\
\lesssim & r_y^{T} \E\left[\norm{\bar y_s - y_{*,s}^{\alpha}}^2\right] + \frac{e_y^{T} \alpha \eta_y}{1-\rho^2}\frac{1}{n}\E\left[\norm{\mY_s - \bar y_s\bfonet_n}^2\right] + \frac{e_y^T\alpha\eta_y^3}{(1-\rho^2)^3}\frac{1}{n}\E\left[\norm{\mU_{s,y} - \bar u_{s,y}\bfonet}^2\right] \notag \\
 + &    \min \left(T, \frac{1}{\alpha\mu_g\eta_y} \right) \left( \frac{C_{x,1}\eta_y}{n}\E\left[\norm{\mX - \bar{x}\bfonet_n}^2\right] + \frac{C_{\sigma,1}}{n} \eta_y^2 \sigma_y^2   \right)  +e_y^T\frac{\eta_{x}\ell_{y^*}}{a_1}\E\left[\norm{\E\left[\bar v_{s+1}\middle|\cF_s\right]}^2\right]   \notag \\
 + &   e_y^T\eta_{x}^2 \left(\ell_{y_*}^2 + \frac{\ell_{\nabla y_*}}{2a_2} \right)\E\left[\norm{\bar v_{s+1}}^2\right] + \frac{e_y^T\alpha^3\eta_y^3}{(1-\rho^2)^3}\left(\frac{1}{n}\E\left[\norm{\mX_{s-1} - \bar{x}_{s-1}\bfonet_n}^2\right] + \eta_x^2\E\left[\norm{\bar{v}_s}^2\right]  \right) \label{inequ:y:recur}
\end{align}
where $e_y = 1 - \frac{\alpha \mu_g\eta_y}{3}$ and $r_y \leq 1 - \frac{\alpha \mu_g\eta_y}{6}$. Recalling the result of Lemma~\ref{lem: inner_error} for the consensus of $y$ and incorporating Inequalities~\eqref{ineq: U1} and \eqref{inequ:alert:y1}, we have
\begin{align}
& \frac{1}{n}\E\left[\norm{\mY_{s+1} - \bar y_{s+1}\bfonet_n}^2\right] \notag \\
 &  \lesssim \left(1 + \frac{\alpha^2\eta_y^2\ell_{g,1}^2}{(1-\rho^2)^4} \right)\frac{e_{\rho,1}^{T}}{n}\E\left[\norm{\mY_{s} - \bar y_s\bfonet_n}^2\right] +  \frac{e_{\rho,1}^T\alpha^3\eta_y^3}{(1-\rho^2)^3}\E\left[\norm{\bar y_s - y_{*,s}^{\alpha}}^2\right]  +  \frac{e_{\rho,1}^T \eta_y^2}{(1-\rho^2)^2 n}\E\left[\norm{\mU_{s,y} - \bar u_{s,y}\bfonet_n}^2\right]   \notag \\
 &+  \frac{\eta_y^2}{(1-\rho^2)n}\left(C_{x,2}\eta_y^2+ \frac{e_{\rho,1}^T\alpha^2}{(1-\rho^2)^2} \right)\E\left[\norm{\mX_s - \bar{x}_s\bfonet_n}^2\right]  + \frac{e_{\rho,1}^T \alpha^2\eta_y^2}{(1-\rho^2)^3} \frac{1}{n}\E\left[\norm{\mX_{s-1} - \bar x_{s-1}\bfonet}^2\right] \notag \\
 &  +  \left(\frac{C_{\sigma,2}}{n}+\frac{e_{\rho,1}^T}{(1-\rho^2)^2} \right)\eta_y^2\sigma_y^2 + \frac{e_{\rho,1}^T \alpha^2\eta_y^2}{(1-\rho^2)^3} \eta_x^2\E\left[\norm{\bar v_s}^2\right]  +   \frac{e_{\rho,1}^T\alpha^3\eta_y^3\eta_{x}^2 }{(1-\rho^2)^3}\E\left[\norm{\bar v_{s+1}}^2\right]. \label{inequ: Y:recur}
\end{align}
for sufficient small $\eta_y$ such that $\left(1 + \frac{\alpha^2\eta_y^2\ell_{g,1}^2}{(1-\rho^2)^4} \right)e_{\rho,1} \leq e_{\rho,2} = \frac{3+\rho^2}{4} $, then $\left(1 + \frac{\alpha^2\eta_y^2\ell_{g,1}^2}{(1-\rho^2)^4} \right)e_{\rho,1}^T \leq e_{\rho,2}^T$ for any $T \geq 1$. Similarly, we apply the result of Lemma~\ref{lem: inner_error}  to the consensus $U_{s,y}$ for $y$ and incorporate Inequalities~\eqref{ineq: U1} and \eqref{inequ:alert:y1}
\begin{align}
& \frac{1}{n}\E\left[\norm{\mU_{s+1,y} - \bar u_{s+1,y}\bfonet_n}^2\right] \notag \\
 &  \lesssim e_{\rho,2}^T \left(1 + \frac{\alpha^2\ell_{g,1}^2\eta_y^2}{(1-\rho^2)^4}\right)\frac{1}{n}\E\left[\norm{\mU_{s,y} - \bar u_{s,y}\bfonet_n}^2\right] + + \frac{e_{\rho,2}^{T-1}\alpha^3\eta_y}{(1-\rho^2)}\E\left[\norm{\bar y_{s} - y_{*,s}^{\alpha}}^2\right]  \notag \\
& + \frac{\alpha^2 e_{\rho,2}^{T-1}}{(1-\rho^2)^3}\frac{1}{n}\E\left[\norm{\mY_{s} - \bar y_{s}\bfonet_n}^2\right]   +  \left(C_{x,3}\eta_y^2  + \frac{e_{\rho,2}^T\alpha^2}{1-\rho^2}\right)\frac{1}{n}\E\left[\norm{\mX_s - \bar{x}_s\bfonet_n}^2\right] \notag \\
 &+ \frac{e_{\rho,2}^T\alpha^2}{(1-\rho^2)n}\E\left[\norm{\mX_{s-1} - \bar x_{s-1}\bfonet}^2\right] + \left(C_{\sigma, 3}  + \frac{e_{\rho,2}^T}{1-\rho^2}\right)\sigma_y^2 \notag \\
 &+ \frac{e_{\rho,2}^T}{1-\rho^2} \eta_x^2\alpha^2 \E\left[\norm{\bar v_s}^2\right] +\frac{e_{\rho,2}^{T}}{(1-\rho^2)}\alpha^3\eta_y\eta_{x}^2\E\left[\norm{\bar v_{s+1}}^2\right]. \label{inequ: U_y:recur}
\end{align}
for sufficiently small $\eta_y \leq (1-\rho^2)^2/(\alpha \ell_{g,1})$, we have $ r_{U,y}:=\left(1 + \frac{\alpha^2\ell_{g,1}^2\eta_y^2}{(1-\rho^2)^4}\right)e_{\rho,2} :=\frac{4 + \rho^2}{5}$.

Combining the above results for $\E\left[\norm{\bar y_{s+1} - y_{*, s+1}^{\alpha}}^2 \right]$,$\frac{1}{n}\E\left[\norm{\mY_{s+1} - \bar y_{s+1}\bfonet_n}^2\right]$, and $\frac{1}{n}\E\left[\norm{\mU_{s+1,y} - \bar u_{s+1,y}\bfonet_n}^2\right]$, we follow the same procedure for variable $z$ and define the vector function $\Omega_{Y,s}$:
\begin{align}
\Omega_{Y,s} = \left(\E\left[\norm{\bar y_s - y_{*, s}^{\alpha}}^2\right],  \frac{1}{n}\E\left[\norm{\mY_{s} - \bar y_s\bfonet_n}^2\right], \frac{1}{n}\E\left[\norm{\mU_{s,y} - \bar u_{s,y}\bfonet_n}^2\right]\right)
\end{align}
and an $3\times 3$ matrix $M_{Y}$
\begin{equation}
M_{Y} = 
\begin{pmatrix}
M_{11} & M_{12} & M_{13}\\
M_{21} & M_{22} & M_{23}\\
M_{31} & M_{32} & M_{33}\\
\end{pmatrix}
\end{equation}
where
\begin{align}
   &  M_{11} = r_y^T; \quad M_{12}=\frac{e_y^{T}\alpha\eta_y}{1-\rho^2}; \quad M_{13} = \frac{e_y^{T}\alpha\eta_y^3}{(1-\rho^2)^3} \notag \\
    & M_{21} = \frac{e_{\rho,1}^T\alpha^3\eta_y^3}{(1-\rho^2)^3};\quad M_{22} = e_{\rho,2}^T; \quad M_{23} = \frac{e_{\rho,1}^T \eta_y^2}{(1-\rho^2)^2} \notag \\
    & M_{31}= \frac{e_{\rho,2}^{T}\alpha^3\eta_y}{(1-\rho^2)}; \quad M_{32} = \frac{\alpha^2e_{\rho,2}^T}{(1-\rho^2)^3}; \quad M_{33} = r_{U,y}^T.
\end{align}
By the inequalities~\eqref{inequ:y:recur}, \eqref{inequ: Y:recur} and \eqref{inequ: U_y:recur},  we have
\begin{align}
\Omega_{Y,s+1} \leq M_{Y}\Omega_{Y,s} + \tilde{C}_{y,s}
\end{align}
where $\tilde{C}_{y,s} \in \R^3$ are defined as below:
\begin{align}
\tilde{C}_{y,s}[1]  & =   \min \left(T, \frac{1}{\alpha\mu_g\eta_y} \right) \left( \frac{C_{x,1}\eta_y}{n}\E\left[\norm{\mX_s - \bar{x}\bfonet_n}^2\right] + \frac{C_{\sigma,1}}{n} \eta_y^2 \sigma_y^2   \right)   \notag \\
& + \frac{e_y^T\alpha^3\eta_y^3}{(1-\rho^2)^3}\frac{1}{n}\E\left[\norm{\mX_{s-1} - \bar{x}_{s-1}\bfonet_n}^2\right]  + e_y^T\frac{\eta_{x}\ell_{y^*}}{a_1}\E\left[\norm{\E\left[\bar v_{s+1}\middle|\cF_s\right]}^2\right] \notag \\
 & + e_y^T\eta_{x}^2 \left(\ell_{y_*}^2 + \frac{\ell_{\nabla y_*}}{2a_2} \right)\E\left[\norm{\bar v_{s+1}}^2\right] + \frac{e_y^T\alpha^3\eta_y^3}{(1-\rho^2)^3}\eta_x^2\E\left[\norm{\bar{v}_s}^2\right]  \notag \\
 \tilde{C}_{y,s}[2]  & = \frac{\eta_y^2}{(1-\rho^2)n}\left(C_{x,2}\eta_y^2+ \frac{e_{\rho,1}^T\alpha^2}{(1-\rho^2)^2} \right)\E\left[\norm{\mX_s - \bar{x}_s\bfonet_n}^2\right]  \notag \\
 & + \frac{e_{\rho,1}^T \alpha^2\eta_y^2}{(1-\rho^2)^3} \frac{1}{n}\E\left[\norm{\mX_{s-1} - \bar x_{s-1}\bfonet}^2\right] +   \left(\frac{C_{\sigma,2}}{n}+\frac{e_{\rho,1}^T}{(1-\rho^2)^2} \right)\eta_y^2\sigma_y^2 \notag \\
&  + \frac{e_{\rho,1}^T \alpha^2\eta_y^2}{(1-\rho^2)^3} \eta_x^2\E\left[\norm{\bar v_s}^2\right]  +   \frac{e_{\rho,1}^T\alpha^3\eta_y^3\eta_{x}^2 }{(1-\rho^2)^3}\E\left[\norm{\bar v_{s+1}}^2\right] \notag \\
\tilde{C}_{y,s}[3]  & = \left(C_{x,3}\eta_y^2  + \frac{e_{\rho,2}^T\alpha^2}{1-\rho^2}\right)\frac{1}{n}\E\left[\norm{\mX_s - \bar{x}_s\bfonet_n}^2\right] + \frac{e_{\rho,2}^T\alpha^2}{(1-\rho^2)n}\E\left[\norm{\mX_{s-1} - \bar x_{s-1}\bfonet}^2\right] \notag \\ & + \left(C_{\sigma, 3}  + \frac{e_{\rho,2}^T}{1-\rho^2}\right)\sigma_y^2 + \frac{e_{\rho,2}^T}{1-\rho^2} \eta_x^2\alpha^2 \E\left[\norm{\bar v_s}^2\right] +\frac{e_{\rho,2}^{T}}{(1-\rho^2)}\alpha^3\eta_y\eta_{x}^2\E\left[\norm{\bar v_{s+1}}^2\right]. 
\end{align}
For simplicity we overload the same notation and set $\Omega_{Y,s} = (a_s, b_s, c_s)^\top$ and $\tilde{C}_{y,s} = (d_{1,s}, d_{2,s}, d_{3,s})^\top$. We thus obtain a similar conclusion for $y$. 
\begin{align}
Q_1 & \sim \frac{e_y^T\alpha\eta_y^3}{(1-\rho^2)^4}\max\left(\frac{1}{\alpha \mu_g\eta_y},2 \right); \quad 
Q_2 \sim \frac{e_{\rho,1}^T\eta_y^2}{(1-\rho^2)^4}.
\end{align}
For sufficient small $\eta_y \leq (1-\rho^2)^{3.5}/(\alpha)$ such that 
\begin{align}
& 1 - Q_1M_{31}  \geq \frac{2}{3}, \quad \quad 1 - Q_2 M_{32} \geq \frac{2}{3}, \notag \\
& \left(\left(1 - Q_1M_{31}\right)\left( 1- Q_2M_{32}\right)-  \left(Q_2M_{31} + \frac{M_{21}}{1-M_{22}}\right)\left(Q_1M_{32} + \frac{M_{12}}{1-M_{11}}\right)\right)   \geq \frac{1}{3}.
\end{align}
Then
\begin{align}
\sum_{i=0}^sb_i  & \lesssim \frac{e_{\rho,1}^T\alpha^2\eta_y^2}{\mu_g(1-\rho^2)^5}a_0 + \max\left(\frac{1}{1-\rho^2},2 \right)b_0 + \frac{e_{\rho,1}^T\eta_y^2}{(1-\rho^2)^4}c_0 + \frac{e_{\rho,1}^T\eta_y^2}{(1-\rho^2)^4}\sum_{i=0}^s d_{3,i}  \notag \\
& + \max\left(\frac{1}{1-\rho^2},2 \right)\sum_{i=0}^s d_{2,i}  + \frac{e_{\rho,1}^T\alpha^2\eta_y^2}{(1-\rho^2)^5}\sum_{i=0}^s d_{1,i},
\end{align}
\begin{align}
\sum_{i=0}^s a_i & \lesssim \max\left(\frac{1}{\alpha\mu_g\eta_y} ,2\right)a_0 + \frac{e_{y}^T}{\mu_g(1-\rho^2)^2}b_0 + \frac{e_{y}^T\eta_y^2}{\mu_g(1-\rho^2)^4}c_0 + \frac{e_{y}^T\eta_y^2}{(1-\rho^2)^5}\sum_{i=0}^s d_{3,i} \notag \\ & + \frac{e_y^T}{\mu_g(1-\rho^2)^2}\sum_{i=0}^s d_{2,i} + \max\left( \frac{1}{\alpha\mu_g\eta_y}, 2\right)\sum_{i=0}^s d_{1,i}.
\end{align}
Combining the above inequalities we have
\begin{align}
&  \frac{1}{2n}\sum_{s=0}^{S}\E\left[\norm{\mY_{s} - y_{*,s}^{\alpha}\bfonet}^2\right] \leq \sum_{s=0}^{S} \E\left[\norm{\bar y_{s} - y_{*, s}^{\alpha}}^2 \right] + \sum_{s=0}^{S} \left[\frac{1}{n}\norm{\mY_{s} - \bar y_{s+1}\bfonet_n}^2 \right] := \sum_{s=0}^{S} \left( a_{s} + b_{s} \right) \notag \\
& \lesssim \frac{e_{\rho,1}^T\alpha^2\eta_y^2}{\mu_g(1-\rho^2)^5}a_0 + \max\left(\frac{1}{1-\rho^2},2 \right)b_0 + \frac{e_{\rho,1}^T\eta_y^2}{(1-\rho^2)^4}c_0 + \frac{e_{\rho,1}^T\eta_y^2}{(1-\rho^2)^4}\sum_{i=0}^s d_{3,i}  \notag \\
& + \max\left(\frac{1}{1-\rho^2},2 \right)\sum_{i=0}^s d_{2,i}  + \frac{e_{\rho,1}^T\alpha^2\eta_y^2}{(1-\rho^2)^5}\sum_{i=0}^s d_{1,i} \notag \\
& + \max\left(\frac{1}{\alpha\mu_g\eta_y} ,2\right)a_0 + \frac{e_{y}^T}{\mu_g(1-\rho^2)^2}b_0 + \frac{e_{y}^T\eta_y^2}{\mu_g(1-\rho^2)^4}c_0 + \frac{e_{y}^T\eta_y^2}{(1-\rho^2)^5}\sum_{i=0}^s d_{3,i} \notag \\ & + \frac{e_y^T}{\mu_g(1-\rho^2)^2}\sum_{i=0}^s d_{2,i} + \max\left( \frac{1}{\alpha\mu_g\eta_y}, 2\right)\sum_{i=0}^s d_{1,i} \notag \\
& \lesssim \max\left(\frac{1}{\alpha\mu_g\eta_y} ,2\right)a_0 + \max\left(1, \frac{e_{y}^T}{\mu_g(1-\rho^2)}\right)\frac{b_0}{1-\rho^2} + \max\left(e_{\rho,1}^T, \frac{e_{y}^T}{\mu_g} \right)\frac{\eta_y^2}{(1-\rho^2)^4}c_0 \notag \\
& + \max\left(e_{\rho,1}^T, \frac{e_{y}^T}{(1-\rho^2)} \right) \frac{\eta_y^2}{(1-\rho^2)^4}\sum_{i=0}^s d_{3,i} + \max\left(\frac{1}{1-\rho^2}, \frac{e_y^T}{\mu_g(1-\rho^2)^2}\right)\sum_{i=0}^s d_{2,i}  + \max\left( \frac{1}{\alpha\mu_g\eta_y}, 2\right)\sum_{i=0}^s d_{1,i} \notag \\
& \lesssim \max\left(\frac{1}{\alpha\mu_g\eta_y} ,2\right)\Delta_{y_*,0} + \max\left(1, \frac{e_{y}^T}{\mu_g(1-\rho^2)}\right)\frac{\Delta_{Y,0}}{1-\rho^2} + + \max\left(e_{\rho,1}^T, \frac{e_{y}^T}{\mu_g} \right)\frac{\eta_y^2}{(1-\rho^2)^4}\Delta_{U_y,0} \notag \\
& + \max\left(e_{\rho,1}^T, \frac{e_{y}^T}{(1-\rho^2)} \right) \frac{\eta_y^2}{(1-\rho^2)^4}\Bigg\{ \left(C_{x,3}\eta_y^2  + \frac{e_{\rho,2}^T\alpha^2}{1-\rho^2}\right)\sum_{s=0}^S\frac{1}{n}\E\left[\norm{\mX_s - \bar{x}_s\bfonet_n}^2\right]  \notag \\
& + \frac{e_{\rho,2}^T\alpha^2}{(1-\rho^2)}\sum_{s=0}^{S-1}\frac{1}{n}\E\left[\norm{\mX_{s-1} - \bar x_{s-1}\bfonet_n}^2\right] + S\left(C_{\sigma, 3}  + \frac{e_{\rho,2}^T}{1-\rho^2}\right)\sigma_y^2 \notag \\
&  + \frac{e_{\rho,2}^T}{1-\rho^2} \eta_x^2\alpha^2 \sum_{s=0}^{S-1}\E\left[\norm{\bar v_{s+1}}^2\right] +\frac{e_{\rho,2}^{T}}{(1-\rho^2)}\alpha^3\eta_y\eta_{x}^2\sum_{s=0}^S \E\left[\norm{\bar v_{s+1}}^2\right] \Bigg\}  \notag \\
&+ \max\left(\frac{1}{1-\rho^2}, \frac{e_y^T}{\mu_g(1-\rho^2)^2}\right) \Bigg\{ \frac{\eta_y^2}{(1-\rho^2)}\left(C_{x,2}\eta_y^2+ \frac{e_{\rho,1}^T\alpha^2}{(1-\rho^2)^2} \right)\sum_{s=0}^S\frac{1}{n}\E\left[\norm{\mX_s - \bar{x}_s\bfonet_n}^2\right] \notag \\
& + \frac{e_{\rho,1}^T \alpha^2\eta_y^2}{(1-\rho^2)^3} \sum_{s=0}^{S-1}\frac{1}{n}\E\left[\norm{\mX_{s} - \bar x_{s}\bfonet_n}^2\right]+   S \left(\frac{C_{\sigma,2}}{n}+\frac{e_{\rho,1}^T}{(1-\rho^2)^2} \right)\eta_y^2\sigma_y^2 \notag \\
& + \frac{e_{\rho,1}^T \alpha^2\eta_y^2}{(1-\rho^2)^3} \eta_x^2\sum_{s=0}^{S-1}\E\left[\norm{\bar v_{s+1}}^2\right]  +   \frac{e_{\rho,1}^T\alpha^3\eta_y^3\eta_{x}^2 }{(1-\rho^2)^3}\sum_{s=0}^S\E\left[\norm{\bar v_{s+1}}^2\right]  \Bigg\} \notag \\
& + \max\left( \frac{1}{\alpha\mu_g\eta_y}, 2\right)\Bigg\{ \min \left(T, \frac{1}{\alpha\mu_g\eta_y} \right) \left( \frac{C_{x,1}\eta_y}{n}\sum_{s=0}^S\E\left[\norm{\mX_s - \bar{x}\bfonet_n}^2\right] + S\frac{C_{\sigma,1}}{n} \eta_y^2 \sigma_y^2   \right)   \notag \\
& + \frac{e_y^T\alpha^3\eta_y^3}{(1-\rho^2)^3}\sum_{s=0}^{S-1}\frac{1}{n}\E\left[\norm{\mX_{s} - \bar{x}_{s}\bfonet_n}^2\right]  + e_y^T\frac{\eta_{x}\ell_{y^*}}{a_1}\sum_{s=0}^{S}\E\left[\norm{\E\left[\bar v_{s+1}\middle|\cF_s\right]}^2\right] \notag \\
& + e_y^T\eta_{x}^2 \left(\ell_{y_*}^2 + \frac{\ell_{\nabla y_*}}{2a_2} \right)\sum_{s=0}^{S}\E\left[\norm{\bar v_{s+1}}^2\right] + \frac{e_y^T\alpha^3\eta_y^3}{(1-\rho^2)^3}\eta_x^2\sum_{s=0}^{S-1}\E\left[\norm{\bar{v}_{s+1}}^2\right]  \Bigg\}
\end{align}
where we use these notations to simplify the inequality
\begin{align*}
\Delta_{y_*,0} & = \E\left[\norm{\bar y_0 - y_{*, 0}^{\alpha}}^2\right], \Delta_{Y,0} = \frac{1}{n}\E\left[\norm{\mY_{0} - \bar y_{0}\bfonet_n}^2 \right], \Delta_{U_y,0} = \frac{1}{n}\E\left[\norm{\mU_{0,y} - \bar u_{0,y}\bfonet_n}^2\right].
\end{align*}
We further re-arrange the above inequality and get that
\begin{align}
& \sum_{s=0}^{S} \frac{1}{2n}\E\left[\norm{\mY_{s} - y_{*,s}^{\alpha}\bfonet}^2\right] \notag \\
& \lesssim C_{y_*,0}\Delta_{y_*,0} +  C_{Y,0}\Delta_{Y,0} + C_{U_y,0}\Delta_{U_y,0}  + C_{y,v}\sum_{s=0}^{S}\E\left[\norm{\E\left[\bar v_{s+1}\middle|\cF_s\right]}^2\right] + C_{y,vs} \sum_{s=0}^{S}\E\left[\norm{\bar v_{s+1}}^2\right]  \notag \\
&  + C_{y,x} \sum_{s=0}^S\frac{1}{n}\E\left[\norm{\mX_s - \bar{x}_s\bfonet_n}^2\right] + S C_{y,\sigma} \sigma_y^2. 
\label{inequ: y_recur_final}
\end{align}
where the constants are given by 
\begin{align}
C_{y_*,0} & = \max \left(\frac{1}{\alpha\mu_g\eta_y} ,2\right), C_{Y,0}=\max\left(1, \frac{e_{y}^T}{\mu_g(1-\rho^2)}\right)\frac{1}{1-\rho^2}, \notag \\
C_{U_y,0} & =\max\left(e_{\rho,1}^T, \frac{e_{y}^T}{\mu_g} \right)\frac{\eta_y^2}{(1-\rho^2)^4} \notag \\
C_{y,v}  & =  \max\left( \frac{1}{\alpha\mu_g\eta_y}, 2\right) e_y^T\frac{\eta_{x}\ell_{y^*}}{a_1} \sim \cO\left( \frac{e_y^T\eta_x}{a_1\alpha \eta_y}\right)\notag \\
C_{y,vs}  & = \frac{e_y^T\eta_y^2}{(1-\rho^2)^5}\left(\frac{e_{\rho,2}^2\eta_x^2\alpha^2}{1-\rho^2} + \frac{e_{\rho,2}^2\eta_x^2\alpha^3\eta_y}{1-\rho^2} \right)  + \frac{e_{y}^T}{(1-\rho^2)^2} \left(\frac{e_{\rho,1}^T\alpha^2\eta_y^2\eta_x^2}{(1-\rho^2)^3} + \frac{e_{\rho,1}^T\alpha^3\eta_y^3\eta_x^2}{(1-\rho^2)^3}\right) \notag \\
& + \max\left( \frac{1}{\alpha\mu_g\eta_y}, 2\right)\left(e_y^T\eta_{x}^2 \left(\ell_{y_*}^2 + \frac{\ell_{\nabla y_*}}{2a_2} \right) + \frac{e_y^T\alpha^3\eta_y^3}{(1-\rho^2)^3}\eta_x^2\right)  \notag \\
& \sim \cO\left( \frac{e_y^Te_{\rho,2}^T\eta_y^2}{(1-\rho^2)^6}\left(\eta_x^2\alpha^2 + \eta_x^2\alpha^3\eta_y\right)  + \frac{e_y^T\eta_x^2}{\alpha\eta_y}\left( 1 + \frac{1}{a_2} + \frac{\alpha^3\eta_y^3}{(1-\rho^2)^3} \right) \right) \notag \\
C_{y,x}   & = \frac{e_y^T\eta_y^2}{(1-\rho^2)^5} \left(\left(C_{x,3}\eta_y^2  + \frac{e_{\rho,2}^T\alpha^2}{1-\rho^2}\right)  + \frac{e_{\rho,2}^T\alpha^2}{(1-\rho^2)}\right) \notag \\
& + \frac{e_{y}^T}{(1-\rho^2)^2} \left(\frac{\eta_y^2}{(1-\rho^2)}\left(C_{x,2}\eta_y^2+ \frac{e_{\rho,1}^T\alpha^2}{(1-\rho^2)^2} \right) + \frac{e_{\rho,1}^T \alpha^2\eta_y^2}{(1-\rho^2)^3}\right) +  \notag \\
& + \max\left( \frac{1}{\alpha\mu_g\eta_y}, 2\right) \left(\min \left(T, \frac{1}{\alpha\mu_g\eta_y} \right) C_{x,1}\eta_y +  \frac{e_y^T\alpha^3\eta_y^3}{(1-\rho^2)^3}\right)\notag \\
& \sim \cO\left(\min\left(T, \frac{1}{\alpha\eta_y}\right) + \frac{e_y^T\alpha^4\eta_y^4}{(1-\rho^2)^7} + \frac{e_{\rho,1}^T\alpha^2\eta_y^2}{(1-\rho^2)^6} + \min\left(T, \frac{1}{\alpha\eta_y}\right)\frac{\alpha^4\eta_y^4}{(1-\rho^2)^4}\right) \notag \\
C_{y,\sigma}  &  = \frac{e_y^T\eta_y^2}{(1-\rho^2)^5}\left(C_{\sigma, 3}  + \frac{e_{\rho,2}^T}{1-\rho^2}\right)  + \frac{e_{y}^T}{(1-\rho^2)^2} \left(\frac{C_{\sigma,2}}{n}+\frac{e_{\rho,1}^T}{(1-\rho^2)^2} \right)\eta_y^2 \notag \\
& + \frac{1}{\alpha\eta_y}  \min \left(T, \frac{1}{\alpha\eta_y} \right) \frac{C_{\sigma,1}}{n} \eta_y^2 \notag \\
& \sim \cO\left(\min\left(T, \frac{1}{\alpha\eta_y}\right)\frac{\eta_y}{\alpha} \left(\frac{\alpha\eta_y}{(1-\rho^2)^4} + \frac{1}{n} \right) + \frac{e_y^T\eta_y^2}{(1-\rho^2)^7}\right).
\label{inequ:y:const}
\end{align}
The proof is complete.  
\end{proof}

Next, we derive the consensus analysis for the upper-level variable $x$.
\begin{restatable}[]{lemma}{lemconsensusconverge} 
\label{lem: consensus}
    Suppose Assumptions \ref{aspt: smoothness}, \ref{aspt: so}, \ref{aspt: sgap}, and \ref{aspt: bdd_grad} hold, consider Algorithm \ref{alg:stochastic_minmax:onestage},
by properly choosing $\eta_x $ such that  $$\eta_x \leq \mathcal{O}\left(\min \left\lbrace \frac{(1-\rho^2)}{\alpha \ell_{f,1}}, \frac{(1-\rho^2)^2}{\alpha \ell_{g,1}}\right\rbrace\right), $$
we have
    \begin{align}
    & \frac{1}{2n}\sum_{s=0}^{S}\E\left[\norm{\mX_{s+1} - \bar x_{s+1}\bfonet_n}^2 \right]\notag \\ 
      & \lesssim\frac{\eta_{x}^2}{(1-\rho^2)^3}\left( C_{vz}C_{z,v} + C_{vy}C_{y,v} \right)\sum_{s=0}^{S}\E\left[\norm{\E\left[\bar v_{i+1}\middle|\cF_i\right]}^2\right] \notag\\
   & + \frac{\eta_{x}^2}{(1-\rho^2)^3}\left(4C_{vz}C_{z,vs} + 4C_{vy}C_{y,vs}  + \frac{C_{vv}}{n}\right)\sum_{s=0}^{S-1}\E\left[\norm{\bar v_{i+1}}^2\right] \notag \\
    & +\frac{\eta_{x}^2 C_{vz}}{(1-\rho^2)^3} \left(C_{z_*,0}\Delta_{z_*,0} +  C_{Z,0}\Delta_{Z,0}  + C_{U_z,0}\Delta_{U_z,0}  \right) \notag \\
   & + \frac{C_{vy}\eta_x^2}{(1-\rho^2)^3} \left( C_{y_*,0}\Delta_{y_*,0} +  C_{Y,0}\Delta_{Y,0} + C_{U_y,0}\Delta_{U_y,0} \right) \notag \\
    & + \frac{\eta_{x}^2S}{(1-\rho^2)^3}\left(C_{vz}C_{z,\sigma} \sigma_z^2 +  C_{vy}C_{y,\sigma} \sigma_y^2 \right)  + \frac{S\eta_{x}^2\sigma_{x}^2}{(1-\rho^2)^3} + \frac{\eta_{x}^2\ell_{f,0}^2(1+\rho^2)}{(1-\rho^2)^3} \notag
 \end{align}
where $\sigma_x^2 = \sigma_f^2 + 2\alpha^2\sigma_g^2$, $\sigma_y^2 = \sigma_f^2 + \alpha^2\sigma_g^2$,  $\sigma_z^2 = \sigma_g^2$,  and other constants are defined in~\eqref{inequ:z:const} and \eqref{inequ:y:const} of Lemma~\ref{lem: yz_convergence}.
\end{restatable}
\begin{proof}(of Lemma~\ref{lem: consensus})
    Note that in Algorithm \ref{alg:stochastic_minmax:onestage} we have
    \begin{align*}
        \mX_{s+1} = \mX_s \mW - \eta_{x}\mV_{s+1},\ \mV_{s+1} = \mV_s\mW + \Delta_{s+1} - \Delta_s,\ 
        \bar x_{s+1} = \bar x_s - \eta_{x}\bar v_{s+1},\ \bar v_s = \bar \delta_s.
    \end{align*}
    Thus we know by Lemmas \ref{lem: cs} (with $c = \frac{1-\rho^2}{2\rho^2}$) and \ref{lem: W_m},
    \begin{align*}
        \norm{\mX_{s+1} - \bar x_{s+1}\bfonet_n}^2\leq \frac{1+\rho^2}{2}\norm{\mX_s - \bar x_s\bfonet_n}^2 + \frac{(1+\rho^2)\eta_{x}^2}{1-\rho^2}\norm{\mV_{s+1} - \bar v_{s+1}\bfonet_n}^2.
    \end{align*}
    We also have
    \begin{align*}
        \mV_{s+1} - \bar v_{s+1}\bfonet_n 
        &= \mV_s\mW + \Delta_{s+1} - \Delta_s - (\bar v_s + \bar \delta_{s+1} - \bar \delta_s)\bfonet_n \\
        &= \left(\mV_s - \bar v_s\bfonet_n\right)\left(\mW - \frac{\bfone_n\bfonet_n}{n}\right) + \left(\Delta_{s+1} - \Delta_s\right)\left(\mI_n - \frac{\bfone_n\bfonet_n}{n}\right)
    \end{align*}
which, together with Lemmas \ref{lem: cs}, \ref{lem: F_norm_ineq} and \ref{lem: W_m}, and $\norm{\mI_n - \frac{\bfone_n\bfonet_n}{n}}_2\leq 1$, implies
\begin{align}
    \norm{\mV_{s+1} - \bar v_{s+1}\bfonet_n}^2\leq \frac{1+\rho^2}{2}\norm{\mV_s - \bar v_s\bfonet_n}^2 + \frac{1+\rho^2}{1-\rho^2}\norm{\Delta_{s+1} - \Delta_s}^2\label{ineq: v_recursion}
\end{align}
To bound $\norm{\Delta_{s+1} - \Delta_s}$, we have
\begin{align*}
    \Delta_{s+1} - \Delta_s = \Delta_{s+1} - \E\left[\Delta_{s+1}\mid \cF_s\right] - (\Delta_s - \E\left[\Delta_s\mid \cF_{s-1}\right])+\E\left[\Delta_{s+1}\mid \cF_s\right]- \E\left[\Delta_s\mid \cF_{s-1}\right],
\end{align*}
and thus
\begin{align}
    &\E\left[\norm{\Delta_{s+1} - \Delta_s}^2\right]\notag\\
    \leq &3\E\left[\norm{\Delta_{s+1} - \E\left[\Delta_{s+1}\mid \cF_s\right]}^2 + \norm{\Delta_s - \E\left[\Delta_s\mid \cF_{s-1}\right]}^2 + \norm{\E\left[\Delta_{s+1}\mid \cF_s\right]- \E\left[\Delta_s\mid \cF_{s-1}\right]}^2\right] \notag\\
    \leq & 6n\sigma_x^2 + 3\E\left[\norm{\E\left[\Delta_{s+1}\mid \cF_s\right]- \E\left[\Delta_s\mid \cF_{s-1}\right]}^2\right] \label{ineq: Delta_diff},
\end{align}
where the stochastic gradient $\Delta_{s+1}$ of updating variable $x$ is variance-bounded by $\sigma_x^2 = \sigma_f^2 + 2\alpha^2\sigma_g^2$.
We then bound $\norm{\E\left[\Delta_{s+1}\mid \cF_s\right]- \E\left[\Delta_s\mid \cF_{s-1}\right]}$ via the following inequalities:
\begin{align}
    &\norm{\E\left[\Delta_{s+1}\mid \cF_s\right]- \E\left[\Delta_s\mid \cF_{s-1}\right]}^2 \notag\\
    \leq &\sum_{i=1}^{n}3\norm{\nabla_{x} f_i(x_{s}^{(i)}, y_{s}^{(i)}) - \nabla_{x} f_i(x_{s-1}^{(i)}, y_{s-1}^{(i)})}^2 + \sum_{i=1}^{n}3\alpha^2\norm{\nabla_{x} g_i(x_{s}^{(i)}, y_{s}^{(i)})  -  \nabla_{x} g_i(x_{s-1}^{(i)}, y_{s-1}^{(i)})}^2 \notag\\
    &+ \sum_{i=1}^{n}3\alpha^2\norm{\nabla_{x} g_i(x_{s}^{(i)}, z_{s}^{(i)}) - \nabla_{x} g_i(x_{s-1}^{(i)}, z_{s-1}^{(i)})}^2 \notag\\
    \leq &(3 + 6\alpha^2)\ell_{f,1}^2\norm{\mX_s - \mX_{s-1}}^2 + (3 + 3\alpha^2)\ell_{g,1}^2\norm{\mY_{s} - \mY_{s-1}}^2 + 3\alpha^2\ell_{g,1}^2\norm{\mZ_{s} - \mZ_{s-1}}^2. \label{ineq: exp_Delta_diff}
\end{align}
Note that we also have for $\norm{\mX_s - \mX_{s-1}}$,
\begin{align}
   \norm{\mX_{s+1} - \mX_s}^2
    = &\norm{\left(\mX_s - \bar x_s\bfonet_n\right)\left(\mW - \mI\right) - \eta_{x}\mV_{s+1}}^2 \notag\\
    \leq &2\norm{\left(\mX_s - \bar x_s\bfonet_n\right)\left(\mW - \mI\right)}^2 + 2\eta_{x}^2\norm{\mV_{s+1}}^2 \notag\\
    \leq &8\norm{\mX_s - \bar x_s\bfonet}^2 + 2\eta_{x}^2\norm{\mV_{s+1} - \bar v_{s+1}\bfonet_n}^2 + 2n\eta_{x}^2\norm{\bar v_{s+1}}^2 \label{ineq: x_diff}
\end{align}
for $\norm{\mY_s - \mY_{s-1}}$,
\begin{align}
   \norm{\mY_s - \mY_{s-1}}^2 
    = &\norm{\mY_s - y_{\ast}^{\alpha}(\bar x_s)\bfonet_n - \mY_{s-1} + y_{\ast}^{\alpha}(\bar x_{s-1})\bfonet_n + y_{\ast}^{\alpha}(\bar x_s)\bfonet_n - y_{\ast}^{\alpha}(\bar x_{s-1})\bfonet_n}^2 \notag\\
    \leq &3\norm{\mY_s - y_{\ast}^{\alpha}(\bar x_s)\bfonet_n}^2 + 3\norm{\mY_{s-1} - y_{\ast}^{\alpha}(\bar x_{s-1})\bfonet_n}^2 + 3n\ell_{y_{\ast}}^2\eta_{x}^2\norm{\bar v_s}^2 \label{ineq: Y_diff}
\end{align}
and for $\norm{\mZ_s - \mZ_{s-1}}$, 
\begin{align}
    \norm{\mZ_s - \mZ_{s-1}}^2 
    = &\norm{\mZ_s - z_{\ast}(\bar x_s)\bfonet_n - \mZ_{s-1} + z_{\ast}(\bar x_{s-1})\bfonet_n + z_{\ast}(\bar x_s)\bfonet_n - z_{\ast}(\bar x_{s-1})\bfonet_n}^2 \notag\\
    \leq &3\norm{\mZ_s - z_{\ast}(\bar x_s)\bfonet_n}^2 + 3\norm{\mZ_{s-1} - z_{\ast}(\bar x_{s-1})\bfonet_n}^2 + 3n\ell_{z_{\ast}}^2\eta_{x}^2\norm{\bar v_s}^2.\label{ineq: Z_diff}
\end{align}
Combining all the inequalities above and setting $\alpha\geq 1$, we obtain
\begin{align*}
    &\E\left[\norm{\E\left[\Delta_{s+1}\mid \cF_s\right]- \E\left[\Delta_s\mid \cF_{s-1}\right]}^2\right]\\
    \leq &9\alpha^2\ell_{f,1}^2\left(8\norm{\mX_{s-1} - \bar x_{s-1}\bfonet}^2 + 2\eta_{x}^2\norm{\mV_s - \bar v_s\bfonet_n}^2 + 2n\eta_{x}^2\norm{\bar v_{s}}^2\right) \\
    +& 6\alpha^2\ell_{g,1}^2\left(3\norm{\mY_s - y_{\ast}^{\alpha}(\bar x_s)\bfonet_n}^2 + 3\norm{\mY_{s-1} - y_{\ast}^{\alpha}(\bar x_{s-1})\bfonet_n}^2 + 3n\ell_{y_{\ast}}^2\eta_{x}^2\norm{\bar v_s}^2\right) \\
   + & 3\alpha^2\ell_{g,1}^2\left(3\norm{\mZ_s - z_{\ast}(\bar x_s)\bfonet_n}^2 + 3\norm{\mZ_{s-1} - z_{\ast}(\bar x_{s-1})\bfonet_n}^2 + 3n\ell_{z_{\ast}}^2\eta_{x}^2\norm{\bar v_s}^2\right)
\end{align*}
and thus for $s \geq 1$, we have
\begin{align}
    &\E\left[\norm{\mV_{s+1} - \bar v_{s+1}\bfonet_n}^2\right] \notag\\
     & \leq\frac{1+\rho^2}{2}\E\left[\norm{\mV_s - \bar v_s\bfonet_n}^2\right] + \frac{1+\rho^2}{1-\rho^2}\E\left[\norm{\Delta_{s+1} - \Delta_s}^2\right] \notag\\
     & \leq\frac{1+\rho^2}{2}\E\left[\norm{\mV_s - \bar v_s\bfonet_n}^2\right] + \frac{6n(1+\rho^2)\sigma_x^2}{1-\rho^2} \notag\\
   & + \frac{3(1+\rho^2)}{1-\rho^2}\bigg\{3\alpha^2\ell_{f,1}^2\left(8\norm{\mX_{s-1} - \bar x_{s-1}\bfonet}^2 + 2\eta_{x}^2\norm{\mV_s - \bar v_s\bfonet_n}^2 + 2n\eta_{x}^2\norm{\bar v_{s}}^2\right)
    \notag\\
    & +6\alpha^2\ell_{g,1}^2\left(\norm{\mY_s - y_{\ast}^{\alpha}(\bar x_s)\bfonet_n}^2 + \norm{\mY_{s-1} - y_{\ast}^{\alpha}(\bar x_{s-1})\bfonet_n}^2 + n\ell_{y_{\ast}}^2\eta_{x}^2\norm{\bar v_s}^2\right) 
    \notag\\
    & +3\alpha^2\ell_{g,1}^2\left(\norm{\mZ_s - z_{\ast}(\bar x_s)\bfonet_n}^2 + \norm{\mZ_{s-1} - z_{\ast}(\bar x_{s-1})\bfonet_n}^2 + n\ell_{z_{\ast}}^2\eta_{x}^2\norm{\bar v_s}^2\right)\bigg\} \notag\\
    & = C_{v}\E\left[\norm{\mV_s - \bar v_s\bfonet_n}^2\right] + C_{vx}\E\left[\norm{\mX_{s-1} - \bar x_{s-1}\bfonet}^2\right]  + C_{vy}\E\left[\norm{\mY_s - y_{\ast}^{\alpha}(\bar x_s)\bfonet_n}^2 + \norm{\mY_{s-1} - y_{\ast}^{\alpha}(\bar x_{s-1})\bfonet_n}^2\right] \notag\\
    &+ C_{vz}\E\left[\norm{\mZ_s - z_{\ast}(\bar x_s)\bfonet_n}^2 + \norm{\mZ_{s-1} - z_{\ast}(\bar x_{s-1})\bfonet_n}^2\right] + C_{vv}\E\left[\norm{\bar v_s}^2\right]+\frac{6n(1+\rho^2)\sigma_x^2}{1-\rho^2}. \label{inequ:v:recursive}
\end{align}
where the constants are given by
\begin{align*}
    &C_{v} = \left(\frac{1+\rho^2}{2} + \frac{18\alpha^2\ell_{f,1}^2\eta_{x}^2(1+\rho^2)}{1-\rho^2}\right) ,\\
    &C_{vx} = \frac{72\alpha^2\ell_{f,1}^2(1+\rho^2)}{1-\rho^2} = \cO\left(\frac{\alpha^2}{1-\rho^2}\right),\ C_{vy} = \frac{18\alpha^2\ell_{g,1}^2(1+\rho^2)}{1-\rho^2} = \cO\left(\frac{\alpha^2}{1-\rho^2}\right),\\
    &C_{vz} = \frac{9\alpha^2\ell_{g,1}^2(1+\rho^2)}{1-\rho^2} = \cO\left(\frac{\alpha^2}{1-\rho^2}\right),\\
    &C_{vv} = \frac{3n\alpha^2\eta_{x}^2(1+\rho^2)(6\ell_{f,1}^2 + 6\ell_{g,1}^2\ell_{y_*}^2 + 3\ell_{g,1}^2\ell_{z_*}^2)}{1-\rho^2} = \cO\left(\frac{n\alpha^2\eta_x^2}{1-\rho^2}\right).
\end{align*}
Especially, for $s=0$, by \eqref{ineq: v_recursion} and $\Delta_0=0$ and $V_0=0$, if we initialize $y_0^{(i)} = z_0^{(i)}$ at each agent, we have
\begin{align}
\E\left[\norm{\mV_{1} - \bar v_{1}\bfonet_n}^2\right] \leq & \frac{1+\rho^2}{2}\norm{\mV_0 - \bar v_0\bfonet_n}^2 + \frac{1+\rho^2}{1-\rho^2}\norm{\Delta_{1} - \Delta_0}^2 \notag \\
\leq &  \frac{1+\rho^2}{2}\norm{\mV_0 - \bar v_0\bfonet_n}^2 + \frac{1+\rho^2}{1-\rho^2}\sum_{i=1}^n \left(\delta_1^{(i)} \right)^2 \notag \\
\leq  &  \frac{1+\rho^2}{2}\norm{\mV_0 - \bar v_0\bfonet_n}^2 + \frac{n \ell_{f,0}^2(1+\rho^2)}{1-\rho^2} \leq  \frac{n \ell_{f,0}^2(1+\rho^2)}{1-\rho^2}. 
\end{align}
Combining the result at the initial iteration $s=0$ and telescoping the inequality~\eqref{inequ:v:recursive} for $s=1,2,\cdots, S$, we have
\begin{align}
& \sum_{i=0}^{S}\E\left[\norm{\mV_{s+1} - \bar v_{s+1}\bfonet_n}^2\right] \notag \\
 \leq & C_{v}\sum_{i=0}^{S}\E\left[\norm{\mV_s - \bar v_s\bfonet_n}^2\right] + C_{vx}\sum_{i=0}^{S-1}\E\left[\norm{\mX_{s} - \bar x_{s}\bfonet}^2\right] + C_{vy}\sum_{i=0}^{S}\E\left[\norm{\mY_s - y_{\ast}^{\alpha}(\bar x_s)\bfonet_n}^2 \right] \notag \\
 + &  C_{vy}\sum_{i=0}^{S-1}\E\left[\norm{\mY_{s} - y_{\ast}^{\alpha}(\bar x_{s})\bfonet_n}^2\right] + \frac{n \ell_{f,0}^2(1+\rho^2)}{1-\rho^2} +\frac{6nS(1+\rho^2)\sigma_{x}^2}{1-\rho^2}\notag \\
   + & C_{vz}\sum_{i=0}^{S}\E\left[\norm{\mZ_s - z_{\ast}(\bar x_s)\bfonet_n}^2 \right] + C_{vz}\sum_{i=0}^{S-1}\E\left[\norm{\mZ_{s} - z_{\ast}(\bar x_{s})\bfonet_n}^2\right] + C_{vv}\sum_{i=0}^{S-1}\E\left[\norm{\bar v_{s+1}}^2\right] \notag \\ 
  \leq & C_{v}\sum_{i=0}^{S}\E\left[\norm{\mV_s - \bar v_s\bfonet_n}^2\right] + C_{vx}\sum_{i=0}^{S-1}\E\left[\norm{\mX_{s} - \bar x_{s}\bfonet}^2\right] + 2C_{vy}\sum_{i=0}^{S}\E\left[\norm{\mY_s - y_{\ast}^{\alpha}(\bar x_s)\bfonet_n}^2 \right] \notag \\
  + & 2C_{vz}\sum_{i=0}^{S}\E\left[\norm{\mZ_s - z_{\ast}(\bar x_s)\bfonet_n}^2 \right] + C_{vv}\sum_{i=0}^{S-1}\E\left[\norm{\bar v_{s+1}}^2\right] + \frac{n \ell_{f,0}^2(1+\rho^2)}{1-\rho^2} +\frac{6nS(1+\rho^2)\sigma_{x}^2}{1-\rho^2}. \label{inequ:v:sum}
\end{align}
Applying the consensus convergence estimations of $y, z$ in Lemma~\ref{lem: yz_convergence} (see~\eqref{ineq: z_recur_final} and \eqref{inequ: y_recur_final}) into \eqref{inequ:v:sum}, we have
\begin{align}\label{inequ:main:V}
& \sum_{i=0}^{S}\frac{1}{n}\E\left[\norm{\mV_{s+1} - \bar v_{s+1}\bfonet_n}^2\right] \notag \\
\lesssim & C_{v}\sum_{i=0}^{S}\frac{1}{n}\E\left[\norm{\mV_s - \bar v_s\bfonet_n}^2\right]  + C_{vx}\sum_{i=0}^{S-1}\frac{1}{n}\E\left[\norm{\mX_{s} - \bar x_{s}\bfonet_n}^2\right]  + C_{vv}\sum_{i=0}^{S-1}\E\left[\norm{\bar v_{s+1}}^2\right]   \notag \\
+ &  4C_{vz} \bigg\{ C_{z_*,0}\Delta_{z_*,0}+  C_{Z,0}\Delta_{Z,0}  + C_{U_z,0}\Delta_{U_z,0}  + C_{z,v} \sum_{s=0}^{S}\E\left[\norm{\E\left[\bar v_{s+1}\middle|\cF_s\right]}^2\right] \notag \\
& + C_{z,vs} \sum_{s=0}^{S-1}\E\left[\norm{\bar v_{s+1}}^2\right] + C_{z,x} \frac{1}{n}\sum_{s=0}^{S-1}\E\left[\norm{\mX_{s} - \bar x_{s}\bfonet_n}^2\right] + S \cdot C_{z,\sigma}  \sigma_z^2 \bigg\} +\frac{6S(1+\rho^2)\sigma_{x}^2}{1-\rho^2} \notag \\
+ & 4C_{vy} \bigg\{ C_{y_*,0}\Delta_{y_*,0} +  C_{Y,0}\Delta_{Y,0} + C_{U_y,0}\Delta_{U_y,0}  + C_{y,vs} \sum_{s=0}^{S}\E\left[\norm{\bar v_{s+1}}^2\right]  \notag \\
&  + C_{y,v}\sum_{s=0}^{S}\E\left[\norm{\E\left[\bar v_{s+1}\middle|\cF_s\right]}^2\right]+ C_{y,x} \sum_{s=0}^S\frac{1}{n}\E\left[\norm{\mX_s - \bar{x}_s\bfonet_n}^2\right] + S C_{y,\sigma} \sigma_y^2 \bigg\}  +  \frac{ \ell_{f,0}^2(1+\rho^2)}{1-\rho^2}\notag \\
\lesssim & C_{v}\sum_{i=0}^{S}\frac{1}{n}\E\left[\norm{\mV_s - \bar v_s\bfonet_n}^2\right] + \left(4C_{vz}C_{z,vs} + 4C_{vy}C_{y,vs}  + \frac{C_{vv}}{n}\right)\sum_{s=0}^{S}\E\left[\norm{\bar v_{s+1}}^2\right] \notag \\
+ &  \left( 2C_{vz}C_{z,v} + 2C_{vy}C_{y,v} \right)\sum_{s=0}^{S}\E\left[\norm{\E\left[\bar v_{s+1}\middle|\cF_s\right]}^2\right] 
+   \left( C_{vx} + 4C_{vz}C_{z,x} + 4C_{vy}C_{y,x} \right)\sum_{i=0}^{S}\frac{1}{n}\E\left[\norm{\mX_{s} - \bar x_{s}\bfonet_n}^2\right]
\notag \\
+ &  4C_{vz} \left(C_{z_*,0}\Delta_{z_*,0} +  C_{Z,0}\Delta_{Z,0}  + C_{U_z,0}\Delta_{U_z,0}  \right) + 4C_{vy} \left( C_{y_*,0}\Delta_{y_*,0} +  C_{Y,0}\Delta_{Y,0} + C_{U_y,0}\Delta_{U_y,0} \right) \notag \\
+ &  4S\left(C_{vz}C_{z,\sigma} \sigma_z^2 +  C_{vy}C_{y,\sigma} \sigma_y^2 \right) + \frac{ \ell_{f,0}^2(1+\rho^2)}{1-\rho^2} +\frac{6S(1+\rho^2)\sigma_{x}^2}{1-\rho^2} 
\end{align}
For sufficient small $\eta_x$  such that
\begin{align}\label{inequ:etax:1}
C_{v}  =\left(\frac{1+\rho^2}{2} + \frac{18\alpha^2\ell_{f,1}^2\eta_{x}^2(1+\rho^2)}{1-\rho^2}\right) & \leq \frac{3+\rho^2}{4},
\end{align}
since $v_0^{(i)}=0$ for each agent $i$,  dividing the both side of \eqref{inequ:V:mid} by $1-C_{v}$ then
\begin{align}\label{inequ:V:mid}
& \frac{1}{n}\sum_{s=0}^{S} \E\left[\norm{\mV_{s+1} - \bar v_{s+1}\bfonet_n}^2\right] \notag \\
&  \lesssim \frac{4}{1-\rho^2}\left(4C_{vz}C_{z,vs} + 4C_{vy}C_{y,vs}  + \frac{C_{vv}}{n}\right)\sum_{s=0}^{S}\E\left[\norm{\bar v_{s+1}}^2\right] \notag \\
& +  \frac{4}{1-\rho^2}\left( 4C_{vz}C_{z,v} + 4C_{vy}C_{y,v} \right)\sum_{s=0}^{S}\E\left[\norm{\E\left[\bar v_{s+1}\middle|\cF_s\right]}^2\right] \notag \\
& +  \frac{16}{1-\rho^2} \left( C_{vx} + 4C_{vz}C_{z,x} + 4C_{vy}C_{y,x} \right)\sum_{i=0}^{S}\frac{1}{n}\E\left[\norm{\mX_{s} - \bar x_{s}\bfonet_n}^2\right]
  \notag \\ 
& + \frac{4C_{vz}}{1-\rho^2} \left(C_{z_*,0}\Delta_{z_*,0} +  C_{Z,0}\Delta_{Z,0}  + C_{U_z,0}\Delta_{U_z,0}  \right) \notag \\
  & + \frac{16C_{vy}}{1-\rho^2} \left( C_{y_*,0}\Delta_{y_*,0} +  C_{Y,0}\Delta_{Y,0} + C_{U_y,0}\Delta_{U_y,0} \right) + \frac{16S}{1-\rho^2}\left(C_{vz}C_{z,\sigma} \sigma_z^2 +  C_{vy}C_{y,\sigma} \sigma_y^2 \right) \notag \\
  & + \frac{ \ell_{f,0}^2(1+\rho^2)}{(1-\rho^2)^2} +\frac{S\sigma_{x}^2}{(1-\rho^2)^2} . 
\end{align}
Recalling that 
    \begin{align}\label{inequ:X:V}
        \norm{\mX_{s+1} - \bar x_{s+1}\bfonet_n}^2\leq \frac{1+\rho^2}{2}\norm{\mX_s - \bar x_s\bfonet_n}^2 + \frac{(1+\rho^2)\eta_{x}^2}{1-\rho^2}\norm{\mV_{s+1} - \bar v_{s+1}\bfonet_n}^2,
    \end{align}
    we have
     \begin{align}
   \frac{1}{n}\sum_{s=0}^{S}\E\left[\norm{\mX_{s+1} - \bar x_{s+1}\bfonet_n}^2 \right] \leq &  \frac{1+\rho^2}{2}\sum_{s=0}^{S}\frac{1}{n}\E\left[\norm{\mX_s - \bar x_s\bfonet_n}^2  \right]+ \frac{(1+\rho^2)\eta_{x}^2}{1-\rho^2}\sum_{s=0}^{S} \frac{1}{n}\E\left[\norm{\mV_{s+1} - \bar v_{s+1}\bfonet_n}^2\right] \notag \\
   \lesssim &  \left(\frac{1+\rho^2}{2} +  \frac{16\eta_x^2}{(1-\rho^2)^2} \left( C_{vx} + 4C_{vz}C_{z,x} + 4C_{vy}C_{y,x} \right)\right)\sum_{s=0}^{S}\frac{1}{n}\E\left[\norm{\mX_s - \bar x_s\bfonet_n}^2 \right]\notag \\
   + &\frac{(1+\rho^2)\eta_{x}^2}{(1-\rho^2)^2}\left( C_{vz}C_{z,v} + C_{vy}C_{y,v} \right)\sum_{s=0}^{S}\E\left[\norm{\E\left[\bar v_{i+1}\middle|\cF_i\right]}^2\right] \notag\\
   + & \frac{\eta_{x}^2}{(1-\rho^2)^2}\left(4C_{vz}C_{z,vs} + 4C_{vy}C_{y,vs}  + \frac{C_{vv}}{n}\right)\sum_{s=0}^{S-1}\E\left[\norm{\bar v_{i+1}}^2\right] \notag \\
   + & \frac{\eta_{x}^2 C_{vz}}{(1-\rho^2)^2} \left(C_{z_*,0}\Delta_{z_*,0} +  C_{Z,0}\Delta_{Z,0}  + C_{U_z,0}\Delta_{U_z,0}  \right) \notag \\
   + &  \frac{C_{vy}\eta_x^2}{(1-\rho^2)^2} \left( C_{y_*,0}\Delta_{y_*,0} +  C_{Y,0}\Delta_{Y,0} + C_{U_y,0}\Delta_{U_y,0} \right) \notag \\
   + & \frac{\eta_{x}^2S}{1-\rho^2}\left(C_{vz}C_{z,\sigma} \sigma_z^2 +  C_{vy}C_{y,\sigma} \sigma_y^2 \right) + \frac{ \eta_{x}^2\ell_{f,0}^2(1+\rho^2)}{(1-\rho^2)^2} + \frac{S\eta_{x}^2\sigma_{x}^2}{(1-\rho^2)^2}.
     \end{align}
Properly choosing $\eta_{x} \leq \mathcal{O}\left((1-\rho^2)^2/(\alpha\ell_{g,1}) \right)$ such that
\begin{align} 
\frac{1+\rho^2}{2} +  \frac{16\eta_x^2}{(1-\rho^2)^2} \left( C_{vx} + 4C_{vz}C_{z,x} + 4C_{vy}C_{y,x} \right) & \leq  \frac{3+\rho^2}{4}
\end{align}
and due to that $x_0^{(i)}=x_0$ for each agent, 
we have
     \begin{align}
   \frac{1}{n}\sum_{s=0}^{S}\E\left[\norm{\mX_{s+1} - \bar x_{s+1}\bfonet_n}^2 \right]
 \lesssim   &   \frac{\eta_{x}^2}{(1-\rho^2)^3}\left( C_{vz}C_{z,v} + C_{vy}C_{y,v} \right)\sum_{s=0}^{S}\E\left[\norm{\E\left[\bar v_{i+1}\middle|\cF_i\right]}^2\right] \notag\\
   + & \frac{\eta_{x}^2}{(1-\rho^2)^3}\left(4C_{vz}C_{z,vs} + 4C_{vy}C_{y,vs}  + \frac{C_{vv}}{n}\right)\sum_{s=0}^{S-1}\E\left[\norm{\bar v_{i+1}}^2\right] \notag \\
   + & \frac{\eta_{x}^2 C_{vz}}{(1-\rho^2)^3} \left(C_{z_*,0}\Delta_{z_*,0} +  C_{Z,0}\Delta_{Z,0}  + C_{U_z,0}\Delta_{U_z,0}  \right) \notag \\
   + & \frac{C_{vy}\eta_x^2}{(1-\rho^2)^3} \left( C_{y_*,0}\Delta_{y_*,0} +  C_{Y,0}\Delta_{Y,0} + C_{U_y,0}\Delta_{U_y,0} \right) \notag \\
   + & \frac{\eta_{x}^2S}{(1-\rho^2)^3}\left(C_{vz}C_{z,\sigma} \sigma_z^2 +  C_{vy}C_{y,\sigma} \sigma_y^2 \right)  + \frac{S\eta_{x}^2\sigma_{x}^2}{(1-\rho^2)^3} + \frac{\eta_{x}^2\ell_{f,0}^2(1+\rho^2)}{(1-\rho^2)^3}.
 \label{inequ:x:consensus:sum}
 \end{align}
where $\sigma_x^2 = \sigma_f^2 + 2\alpha^2\sigma_g^2$, $\sigma_y^2 = \sigma_f^2 + \alpha^2\sigma_g^2 $,  $\sigma_z^2 = \sigma_g^2$. Note that the symbol $\lesssim$ indicates that there exists an absolute constant $C$ such that LHS $\leq$ $C$ RHS. Now we have completed the proof.
\end{proof}

\section{Appendix / Convergence complexity in Theorem~\ref{thm: warm_start_rate}}\label{sec: complexity_app}

To derive the convergence rate in Theorem~\ref{thm: warm_start_rate}, we first have the following sufficient decrease lemma.
\begin{restatable}[]{lemma}{lemsuffcientdecrease} 
\label{lem: suffcient:decrease}
    Suppose Assumptions \ref{aspt: smoothness} and \ref{aspt: so} hold. The stepsize $\eta_{x}$ satisfies $\eta_{x}\leq 1/ \ell_{\Gamma}$ for all $s$. We have
    \begin{align}
        &\frac{\eta_{x}}{2}\norm{\nabla \Gamma^{\alpha}(\bar x_s)}^2 + \left(\frac{\eta_{x}}{2} - \frac{\eta_{x}^2\ell_{\Gamma}}{2}\right)\norm{\E\left[\bar v_{s+1}\middle|\cF_s\right]}^2 \notag\\
        \leq & \Gamma^{\alpha}(\bar x_s) - \E\left[\Gamma^{\alpha}(\bar x_{s+1})\middle|\cF_s\right] + \frac{\eta_{x}}{2}\norm{\E\left[\bar v_{s+1}\middle|\cF_s\right] - \nabla \Gamma^{\alpha}(\bar x_s)}^2  + \frac{\ell_{\Gamma}\eta_{x}^2\sigma^2}{2n} \notag.
    \end{align}
\end{restatable}
\begin{proof}(of Lemma~\ref{lem: suffcient:decrease})
    Note that in Algorithm \ref{alg:stochastic_minmax:onestage} we have 
    \begin{align}\label{eq: x_update}
        \bar x_{s+1} = \bar x_s - \eta_{x}\bar v_{s+1}.
    \end{align}
    By smoothness of $\Gamma^{\alpha}(x)$, we have
    \begin{align}\label{ineq: cL_decrease}
        \Gamma^{\alpha}(\bar x_{s+1}) - \Gamma^{\alpha}(\bar x_s) - \<\nabla \Gamma^{\alpha}(\bar x_s), \bar x_{s+1} - \bar x_s> \leq \frac{\ell_{\Gamma}}{2}\norm{\bar x_{s+1} - \bar x_s}^2
    \end{align}
    and thus
    \begin{align}\label{ineq: cL_decrease_exp}
        \E\left[\Gamma^{\alpha}(\bar x_{s+1})\middle|\cF_s\right] - \Gamma^{\alpha}(\bar x_s) + \eta_{x}\<\nabla \Gamma^{\alpha}(\bar x_s), \E\left[\bar v_{s+1}\middle|\cF_s\right]> \leq \frac{\ell_{\Gamma}\eta_{x}^2}{2}\E\left[\norm{\bar v_{s+1}}^2\middle|\cF_s\right]
    \end{align}
    which implies
    \begin{align}
        &\frac{\eta_{x}}{2}\left(\norm{\nabla \Gamma^{\alpha}(\bar x_s)}^2 + \norm{\E\left[\bar v_{s+1}\middle|\cF_s\right]}^2 - \norm{\E\left[\bar v_{s+1}\middle|\cF_s\right] - \nabla \Gamma^{\alpha}(\bar x_s)}^2\right)\notag\\
        \leq &  \Gamma^{\alpha}(\bar x_s) - \E\left[\Gamma^{\alpha}(\bar x_{s+1})\middle|\cF_s\right] + \frac{\ell_{\Gamma}\eta_{x}^2\sigma^2}{2n} + \frac{\ell_{\Gamma}\eta_{x}^2}{2}\norm{\E\left[\bar v_{s+1}\middle|\cF_s\right]}^2.\label{ineq: noncvx_basic}
    \end{align}
    Rearranging terms on both sides completes the proof. 
\end{proof}
Note that following the analysis in \cite{chen2021closing}, $\norm{\E\left[\bar v_{s+1}\middle|\cF_s\right]}^2$ should not be thrown away since it will be used later in the analysis. Another commonly used approach is to incorporate the moving-average updates in the algorithm, which has been used in distributed optimization \citep{xiao2023one, huang2023stochastic}.

Now we analyze $\norm{\E\left[\bar v_{s+1}\middle|\cF_s\right] - \nabla \Gamma^{\alpha}(\bar x_s)}$.
\begin{restatable}[]{lemma}{lemhypergraderror} 
\label{lem: hypergrad_error}
    Suppose Assumptions \ref{aspt: smoothness}, \ref{aspt: so}, and \ref{aspt: sgap} hold. We have
    \begin{align*}
       \norm{\E\left[\bar v_{s+1}\middle|\cF_s\right] - \nabla \Gamma^{\alpha}(\bar x_s)}^2 
       & \leq \frac{3\ell_{x,1}^2}{n}\norm{\mX_s - \bar x_s\bfonet_n}^2 + \frac{3\ell_{y,1}^2}{n}\norm{\mY_s - y_{\ast}^{\alpha}(\bar x_s)\bfonet_n}^2  +  \frac{3\alpha^2\ell_{z,1}^2}{n}\norm{\mZ_s - z_{\ast}(\bar x_s)\bfonet_n}^2
    \end{align*}
    where $\ell_{x,1}^2 = \ell_{f,1}^2+2\alpha^2\ell_{g,1}^2$, $\ell_{y,1}^2=\ell_{f,1}^2+\alpha^2\ell_{g,1}^2$, and $\ell_{z,1}^2 = \ell_{g,1}^2$.
\end{restatable}
\begin{proof}(of Lemma~\ref{lem: hypergrad_error})
    By Lemma \ref{lem: gt} we know it suffices to analyze $\norm{\E\left[\bar \delta_{s+1}\middle|\cF_s\right] - \nabla \Gamma^{\alpha}(\bar x_s)}$. 
\begin{align*}
  &\norm{\E\left[\bar \delta_{s+1}\middle|\cF_s\right] - \nabla_{x} \cL^{\alpha}(\bar x_s, {y_{\ast}^{\alpha}}(\bar x_s), z_{\ast}(\bar x_s))}^2 \notag\\
  & = \norm{\frac{1}{n} \sum_{i=1}^n\E\left[\delta_{s+1}^{(i)}\middle|\cF_s\right] - \nabla_{x} \cL^{\alpha}(\bar x_s, {y_{\ast}^{\alpha}}(\bar x_s), z_{\ast}(\bar x_s))}^2 \notag \\
  & = \frac{3}{n^2} \norm{\sum_{i=1}^n \nabla_{x}f_i(x_{s}^{(i)}, y_s^{(i)}) - \nabla_{x} f_i(\bar x_s,{y_{\ast}^{\alpha}}(\bar x_s))}^2 + \frac{3\alpha^2}{n^2} \norm{\sum_{i=1}^n \nabla_{x} g_i(x_{s}^{(i)}, y_s^{(i)}) - \nabla_x g_i(\bar x_s, {y_{\ast}^{\alpha}}(\bar x_s))}^2 \notag \\
  & \quad + \frac{3\alpha^2}{n^2} \norm{\sum_{i=1}^n \nabla_{x} g_i(x_{s}^{(i)}, z_s^{(i)}) - \nabla_x g_i(\bar x_s, z_{\ast}(\bar x_s))}^2 \notag \\
  &\leq \frac{3\ell_{f,1}^2}{n}\sum_{i=1}^n\left(\norm{x_{s}^{(i)} - \bar x_s}^2 + \norm{y_s^{(i)} - {y_{\ast}^{\alpha}}(\bar x_s)}^2 \right) + \frac{3\alpha^2\ell_{g,1}^2}{n}\sum_{i=1}^n\left(\norm{x_{s}^{(i)} - \bar x_s}^2 + \norm{y_s^{(i)} - {y_{\ast}^{\alpha}}(\bar x_s)}^2 \right) \notag \\
  &\quad + \frac{3\alpha^2\ell_{g,1}^2}{n}\sum_{i=1}^n\left(\norm{x_{s}^{(i)} - \bar x_s}^2 + \norm{z_s^{(i)} - z_{\ast}(\bar x_s)}^2 \right) \notag\\
  & = \frac{3\left(\ell_{f,1}^2+2\alpha^2\ell_{g,1}^2 \right)}{n}\norm{\mX_s - \bar x_s\bfonet_n}^2 + \frac{3\left(\ell_{f,1}^2+\alpha^2\ell_{g,1}^2 \right)}{n}\norm{\mY_s - {y_{\ast}^{\alpha}}(\bar x_s)\bfonet_n}^2  +  \frac{3\alpha^2\ell_{g,1}^2}{n}\norm{\mZ_s - z_{\ast}(\bar x_s)\bfonet_n}^2. 
\end{align*}
\end{proof}

\thmconverge* 
\begin{proof}(of Theorem~\ref{thm: warm_start_rate})
Incorporating the result of  Lemma~\ref{lem: hypergrad_error} into the sufficient decrease condition in Lemma~\ref{lem: suffcient:decrease} and telescoping the inequality from $s=0$ to $S-1$, we have
\begin{align}
& \sum_{s=0}^{S-1}\frac{\eta_{x}}{2}\norm{\nabla \Gamma^{\alpha}(\bar x_s)}^2 + \left(\frac{\eta_{x}}{2} - \frac{\eta_{x}^2\ell_{\Gamma}}{2}\right)\sum_{s=0}^{S-1}\norm{\E\left[\bar v_{s+1}\middle|\cF_s\right]}^2 \notag \\
        \leq & \sum_{s=0}^{S-1}\Gamma^{\alpha}(\bar x_s) - \E\left[\Gamma^{\alpha}(\bar x_{s+1})\middle|\cF_s\right] + \frac{\eta_{x}}{2} \sum_{s=0}^{S-1}\norm{\E\left[\bar v_{s+1}\middle|\cF_s\right] - \nabla \Gamma^{\alpha}(\bar x_s)}^2  + \sum_{s=0}^{S-1}\frac{\ell_{\Gamma}\eta_{x}^2\sigma_x^2}{2n} \notag\\
      \leq  & \Gamma^{\alpha}(\bar x_0) - \E\left[\Gamma^{\alpha}(\bar x_{S})\middle|\cF_S\right] + \frac{\eta_{x}}{2} \sum_{s=0}^{S-1}\norm{\E\left[\bar v_{s+1}\middle|\cF_s\right] - \nabla \Gamma^{\alpha}(\bar x_s)}^2  + S\frac{\ell_{\Gamma}\eta_{x}^2\sigma_x^2}{2n} \notag \\
      \leq & \Delta_{x,0} 
      +  \frac{\eta_x}{2}\sum_{s=0}^{S-1}\left(\frac{3\ell_{x,1}^2}{n}\norm{\mX_s - \bar x_s\bfonet_n}^2 + \frac{3\ell_{y,1}^2}{n}\norm{\mY_s - {y_{\ast}^{\alpha}}(\bar x_s)\bfonet_n}^2 +  \frac{3\alpha^2\ell_{z,1}^2}{n}\norm{\mZ_s - z_{\ast}(\bar x_s)\bfonet_n}^2 \right)
      \notag \\
      &  + S\frac{\ell_{\Gamma}\eta_{x}^2\sigma_x^2}{2n} \label{inequ:sufficient:decrease:1}
    \end{align}
    where $\Delta_{x,0}= \Gamma^{\alpha}(\bar x_0) - \Gamma^{\alpha}(x^{\ast})$. Incorporating the consensus results of  $Y$ and $Z$ in Lemma~\ref{lem: yz_convergence}, we have
    \begin{align}
  & \sum_{s=0}^{S-1}\frac{\eta_{x}}{2}\norm{\nabla \Gamma^{\alpha}(\bar x_s)}^2 + \left(\frac{\eta_{x}}{2} - \frac{\eta_{x}^2\ell_{\Gamma}}{2}\right)\sum_{s=0}^{S-1}\norm{\E\left[\bar v_{s+1}\middle|\cF_s\right]}^2 \notag \\  
  \lesssim & \Delta_{x,0} + S\frac{\ell_{\Gamma}\eta_{x}^2\sigma_x^2}{2n} + \frac{3\eta_x\ell_{y,1}^2}{2}\bigg\{
  C_{y_*,0}\Delta_{y_*,0} +  C_{Y,0}\Delta_{Y,0} + C_{U_y,0}\Delta_{U_y,0}  + C_{y,vs} \sum_{s=0}^{S-1}\E\left[\norm{\bar v_{s+1}}^2\right]\notag \\
 &   + C_{y,v}\sum_{s=0}^{S-1}\E\left[\norm{\E\left[\bar v_{s+1}\middle|\cF_s\right]}^2\right]  + C_{y,x} \sum_{s=0}^{S-1}\frac{1}{n}\E\left[\norm{\mX_s - \bar{x}_s\bfonet_n}^2\right] + S C_{y,\sigma} \sigma_y^2 \bigg\} \notag \\
 + & \frac{3\eta_x\alpha^2\ell_{z,1}^2}{2}\bigg\{C_{z_*,0}\Delta_{z_*,0}+  C_{Z,0}\Delta_{Z,0}  + C_{U_z,0}\Delta_{U_z,0}  + C_{z,v} \sum_{s=0}^{S-1}\E\left[\norm{\E\left[\bar v_{s+1}\middle|\cF_s\right]}^2\right] \notag \\
 &+ C_{z,vs} \sum_{s=0}^{S-1}\E\left[\norm{\bar v_{s+1}}^2\right]  + C_{z,x} \frac{1}{n}\sum_{s=0}^{S-1}\E\left[\norm{\mX_{s} - \bar x_{s}\bfonet}^2\right] + S \cdot C_{z,\sigma}  \sigma_z^2 \bigg\} \notag \\
 & + \frac{3\eta_x\ell_{x,1}^2}{2}\sum_{s=0}^{S-1}\frac{1}{n}\E\left[\norm{\mX_s - \bar x_s\bfonet_n}^2 \right] \notag \\
\lesssim &  \Delta_{x,0} + \frac{3\eta_x\alpha^2}{2} \left( C_{y_*,0}\Delta_{y_*,0} +  C_{Y,0}\Delta_{Y,0} + C_{U_y,0}\Delta_{U_y,0} + C_{z_*,0}\Delta_{z_*,0}+  C_{Z,0}\Delta_{Z,0}  + C_{U_z,0}\Delta_{U_z,0}\right) \notag \\
+ & \frac{3\eta_x\alpha^2}{2}\left(C_{y,vs} + C_{z,vs}\right)\sum_{s=0}^{S-1}\E\left[\norm{\bar v_{s+1}}^2\right]  + \frac{3\eta_x\alpha^2}{2}\left(C_{y,v} + C_{z,v}\right)\sum_{s=0}^{S-1}\E\left[\norm{\E\left[\bar v_{s+1}\middle|\cF_s\right]}^2\right] \notag \\
+ &   \frac{3\eta_x\alpha^2}{2}\left(C_{y,x} + C_{z,x} + 1\right)\sum_{s=0}^{S-1}\frac{1}{n}\E\left[\norm{\mX_s - \bar{x}_s\bfonet_n}^2\right] + \frac{S\eta_{x}^2\sigma_x^2}{2n} +\frac{S\eta_x\alpha^2}{2}\left(C_{y,\sigma}\sigma_y^2 +  C_{z,\sigma}\sigma_z^2\right) 
    \end{align}
where $\ell_{x,1}^2= \mathcal{O}(\alpha^2), \ell_{y,1}^2= \mathcal{O}(\alpha^2)$ and $\ell_{z,1}^2= \mathcal{O}(1)$.  For simplicity, we let $C_X =  C_{y,x} + C_{z,x} + 1$. Then incorporating the sum w.r.t $\mX_s$  in Lemma~\ref{lem: consensus} and dividing $\eta_x S$ on both side,  we achieve that
\begin{align}
 & \frac{1}{2S}\sum_{s=0}^{S-1}\norm{\nabla \Gamma^{\alpha}(\bar x_s)}^2 + \left(\frac{1}{2S} - \frac{\eta_{x}\ell_{\Gamma}}{2S}\right)\sum_{s=0}^{S-1}\norm{\E\left[\bar v_{s+1}\middle|\cF_s\right]}^2 \notag \\ 
\lesssim & \frac{\Delta_{x,0}}{\eta_x S} + \frac{\alpha^2}{S} \left( C_{y_*,0}\Delta_{y_*,0} +  C_{Y,0}\Delta_{Y,0} + C_{U_y,0}\Delta_{U_y,0} + C_{z_*,0}\Delta_{z_*,0}+  C_{Z,0}\Delta_{Z,0}  + C_{U_z,0}\Delta_{U_z,0}\right) \notag \\
+ & \frac{\alpha^2}{S}\left(C_{y,vs} + C_{z,vs}\right)\sum_{s=0}^{S-1}\E\left[\norm{\bar v_{s+1}}^2\right]  + \frac{\alpha^2}{S}\left(C_{y,v} + C_{z,v}\right)\sum_{s=0}^{S-1}\E\left[\norm{\E\left[\bar v_{s+1}\middle|\cF_s\right]}^2\right] \notag \\
+ &  \frac{\eta_{x}\sigma_x^2}{2n} +\frac{\alpha^2}{2}\left(C_{y,\sigma}\sigma_y^2 +  C_{z,\sigma}\sigma_z^2\right)  + \frac{\alpha^2C_X }{S} \Bigg\{\frac{\eta_{x}^2 \left( C_{vz}C_{z,v} + C_{vy}C_{y,v} \right)}{(1-\rho^2)^3}\sum_{s=0}^{S-1}\E\left[\norm{\E\left[\bar v_{i+1}\middle|\cF_i\right]}^2\right] \notag\\
    & + \frac{\eta_{x}^2}{(1-\rho^2)^3}\left(4C_{vz}C_{z,vs} + 4C_{vy}C_{y,vs}  + \frac{C_{vv}}{n}\right)\sum_{s=0}^{S-1}\E\left[\norm{\bar v_{i+1}}^2\right] \notag \\
   & + \frac{\eta_{x}^2 C_{vz}}{(1-\rho^2)^3} \left(C_{z_*,0}\Delta_{z_*,0} +  C_{Z,0}\Delta_{Z,0}  + C_{U_z,0}\Delta_{U_z,0}  \right) \notag \\
   & + \frac{C_{vy}\eta_x^2}{(1-\rho^2)^3} \left( C_{y_*,0}\Delta_{y_*,0} +  C_{Y,0}\Delta_{Y,0} + C_{U_y,0}\Delta_{U_y,0} \right) \notag \\
    & + \frac{\eta_{x}^2S}{(1-\rho^2)^3}\left(C_{vz}C_{z,\sigma} \sigma_z^2 +  C_{vy}C_{y,\sigma} \sigma_y^2 \right)  + \frac{S\eta_{x}^2\sigma_{x}^2}{(1-\rho^2)^3} + \frac{ \eta_{x}^2\ell_{f,0}^2(1+\rho^2)}{(1-\rho^2)^3} \Bigg\} \notag \\
    \lesssim & \frac{\Delta_{x,0}}{\eta_x S} + \frac{\alpha^2}{S}\left(\left(C_{y,v} + C_{z,v}\right)  + C_X \frac{\eta_{x}^2}{(1-\rho^2)^3}\left( C_{vz}C_{z,v} + C_{vy}C_{y,v} \right)\right)\sum_{s=0}^{S-1}\E\left[\norm{\E\left[\bar v_{i+1}\middle|\cF_i\right]}^2\right] \notag \\
    + &   \frac{\alpha^2}{S}\left( \left(C_{y,vs} + C_{z,vs}\right) + C_X \frac{\eta_{x}^2}{(1-\rho^2)^3}\left(4C_{vz}C_{z,vs} + 4C_{vy}C_{y,vs}  + \frac{C_{vv}}{n}\right)\right)\sum_{s=0}^{S-1}\E\left[\norm{\bar v_{i+1}}^2\right]  \notag \\ 
    +&  \frac{\eta_{x}\sigma_x^2}{2n} +\frac{\alpha^2}{2}\left(C_{y,\sigma}\sigma_y^2 +  C_{z,\sigma}\sigma_z^2\right)   + \alpha^2 C_X \left(\frac{\eta_{x}^2}{(1-\rho^2)^3}\left(C_{vz}C_{z,\sigma} \sigma_z^2 +  C_{vy}C_{y,\sigma}\sigma_y^2 \right) + \frac{\eta_{x}^2\sigma_x^2}{(1-\rho^2)^3}\right)  \notag \\
  +  &  \frac{\alpha^2}{S}C_X \frac{ \eta_{x}^2\ell_{f,0}^2(1+\rho^2)}{(1-\rho^2)^3} + 
    \frac{\alpha^2}{S}\left(1 +  C_X \frac{C_{vy}\eta_x^3}{(1-\rho^2)^3}\right) \left( C_{y_*,0}\Delta_{y_*,0} +  C_{Y,0}\Delta_{Y,0} + C_{U_y,0}\Delta_{U_y,0} \right) \notag \\
   +  & \frac{\alpha^2}{S}\left(1 + C_X\frac{\eta_{x}^2 C_{vz}}{(1-\rho^2)^3}\right)\left(C_{z_*,0}\Delta_{z_*,0}+  C_{Z,0}\Delta_{Z,0}  + C_{U_z,0}\Delta_{U_z,0}\right).
\end{align}
Incorporating the inequality that 
\begin{equation}\label{ineq: bdd_var}
    \E\left[\norm{\bar v_{i}}^2\right] \leq \E\left[\norm{\E\left[\bar v_{i}\middle|\cF_i\right]}^2\right] + \mathcal{O}\left(\frac{\sigma_x^2}{n}\right),
\end{equation}
we have
\begin{align}
 & \frac{1}{2S}\sum_{s=0}^{S-1}\norm{\nabla \Gamma^{\alpha}(\bar x_s)}^2 + \left(\frac{1}{2S} - \frac{\eta_{x}\ell_{\Gamma}}{2S}\right)\sum_{s=0}^{S-1}\norm{\E\left[\bar v_{s+1}\middle|\cF_s\right]}^2 \notag \\ 
 \lesssim & \frac{\Delta_{x,0}}{\eta_x S} + \frac{\alpha^2}{S}\left(\left(C_{y,v} + C_{z,v}\right)  + C_X \frac{\eta_{x}^2}{(1-\rho^2)^3}\left( C_{vz}C_{z,v} + C_{vy}C_{y,v} \right)\right)\sum_{s=0}^{S-1}\E\left[\norm{\E\left[\bar v_{i+1}\middle|\cF_i\right]}^2\right] \notag \\
    + &   \frac{\alpha^2}{S}\left( \left(C_{y,vs} + C_{z,vs}\right) +  \frac{C_X\eta_{x}^2}{(1-\rho^2)^3}\left(4C_{vz}C_{z,vs} + 4C_{vy}C_{y,vs}  + \frac{C_{vv}}{n}\right)\right)\sum_{s=0}^{S-1}\E\left[\norm{\E\left[\bar v_{i+1}\middle|\cF_i\right]}^2\right]  \notag \\ 
   + &  \frac{\alpha^2}{S}\left( \left(C_{y,vs} + C_{z,vs}\right) + C_X \frac{\eta_{x}^2}{(1-\rho^2)^3}\left(4C_{vz}C_{z,vs} + 4C_{vy}C_{y,vs}  + \frac{C_{vv}}{n}\right)\right)\frac{S\sigma_x^2}{n}\notag \\
    + & \frac{\eta_{x}\sigma_x^2}{2n} +\frac{\alpha^2}{2}\left(C_{y,\sigma}\sigma_y^2 +  C_{z,\sigma}\sigma_z^2\right)  + \frac{\alpha^2 C_X\eta_{x}^2}{(1-\rho^2)^3} \left(C_{vz}C_{z,\sigma} \sigma_z^2 +  C_{vy}C_{y,\sigma}\sigma_y^2 + \sigma_x^2\right)  \notag \\
   + &  \frac{\alpha^2}{S}C_X \frac{ \eta_{x}^2\ell_{f,0}^2(1+\rho^2)}{(1-\rho^2)^3} + 
    \frac{\alpha^2}{S}\left(1 +  C_X \frac{C_{vy}\eta_x^3}{(1-\rho^2)^3}\right) \left( C_{y_*,0}\Delta_{y_*,0} +  C_{Y,0}\Delta_{Y,0} + C_{U_y,0}\Delta_{U_y,0} \right) \notag \\
   + &  \frac{\alpha^2}{S}\left(1 + C_X\frac{\eta_{x}^2 C_{vz}}{(1-\rho^2)^3}\right)\left(C_{z_*,0}\Delta_{z_*,0}+  C_{Z,0}\Delta_{Z,0}  + C_{U_z,0}\Delta_{U_z,0}\right). \label{ineq: original_long}
\end{align}
We consider Algorithm~\ref{alg:stochastic_minmax:onestage} with $T=1$ and variance $\sigma_x^2 \sim \Theta\left(\alpha^2\sigma^2 \right), \sigma_y^2 \sim \Theta\left(\alpha^2\sigma^2 \right), \sigma_z^2= \Theta\left(\sigma^2\right)$ (defined in Lemma \ref{lem: consensus}), to ensure that 
     \begin{align}\label{ineq: eta_xzt}
       r_z = \left(1 + a_1\eta_{x}\ell_{z_*} + \frac{a_2\ell_{\nabla z_*}c_{\delta}\eta_{x}^2\alpha^2}{2}\right)\left(1-\frac{2\mu_g\eta_{z}}{3}\right)^T\leq 1 - \frac{\mu_g\eta_{z}}{3},
    \end{align}
    we set
    \begin{align}
    a_1 \lesssim \frac{\eta_z}{\eta_x}; \quad a_2 \lesssim \frac{\eta_z}{\eta_x^2\alpha^2}.
    \end{align}
    Similarly, to guarantee that 
      \begin{align}\label{ineq: eta_xyt}
        r_y = \left(1 + a_1\eta_{x}\ell_{y_*} + \frac{a_2\ell_{\nabla y_*}c_{\delta}\eta_{x}^2\alpha^2}{2}\right)\left(1-\frac{\mu_g\eta_{y}\alpha}{3}\right)^T\leq 1 - \frac{\mu_g\eta_{y}\alpha}{6}.
    \end{align}
    we might choose
    \begin{align}
    a_1 & \lesssim \frac{\alpha\eta_y}{\eta_x};\quad a_2 \lesssim \frac{\alpha\eta_y}{\eta_x^2\alpha^2}. 
    \end{align} 
    Thus to make sure the two conditions~\eqref{ineq: eta_xzt} and \eqref{ineq: eta_xyt} both hold,  we choose $\alpha\eta_y$ and $\eta_z$ are in the same scale, then $a_1$ and $a_2$ are well-defined.
    Setting  
    \begin{align}\label{ineq: condition1}
        \alpha\eta_y & \sim \eta_z \gg \frac{1}{S}, 
       \end{align}
       and $C_X$ can be simplified as
       \begin{align}
       C_{X} & = C_{y,x} + C_{z,x} + 1 \notag \\
       & \sim \cO\left(\min\left(T, \frac{1}{\alpha\eta_y}\right) + \frac{e_y^T\alpha^4\eta_y^4}{(1-\rho^2)^7} + \frac{e_{\rho,1}^T\alpha^2\eta_y^2}{(1-\rho^2)^6} + \min\left(T, \frac{1}{\alpha\eta_y}\right)\frac{\alpha^4\eta_y^4}{(1-\rho^2)^4}\right) \notag \\
       & + \cO\left(\frac{e_{z}^T\eta_z^4}{(1-\rho^2)^7} + \frac{e_{\rho,2}^Te_{z}^T\eta_z^2}{(1-\rho^2)^5} + \min \left(T, \frac{1}{\eta_z} 
 \right)\right)  + 1\notag \\
       & \sim \mathcal{O}(1).
       \end{align}
    Combining this and replacing $e_y^T, e_{\rho,1}^T, e_{\rho,2}^T$ and $e_z^T$ with 1 and let $a_1 \sim \frac{\alpha\eta_y}{\eta_x} \sim \frac{\eta_z}{\eta_x}$ and $a_2 \sim \frac{\eta_z}{\eta_x^2\alpha^2} \sim \frac{\alpha\eta_y}{\eta_x^2\alpha^2}$, the inequality \eqref{ineq: original_long} can be simplified
\begin{align}\label{inequ:t1:}
 & \frac{1}{2S}\sum_{s=0}^{S-1}\norm{\nabla \Gamma^{\alpha}(\bar x_s)}^2 + \left(\frac{1}{2S} - \frac{\eta_{x}\ell_{\Gamma}}{2S}\right)\sum_{s=0}^{S-1}\norm{\E\left[\bar v_{s+1}\middle|\cF_s\right]}^2 \notag \\ 
 \lesssim & {\frac{\Delta_{x,0}}{\eta_x S}} + \frac{\alpha^2}{S}\left(\left({\frac{\eta_x^2}{\alpha^2\eta_y^2}} + \frac{\eta_x^2}{\eta_z^2}\right)  + \frac{\eta_{x}^2}{(1-\rho^2)^4}\left(\frac{\alpha^2 \eta_x^2}{\eta_z^2} + \frac{\eta_x^2}{\eta_y^2} \right)\right)\sum_{s=0}^{S-1}\E\left[\norm{\E\left[\bar v_{i+1}\middle|\cF_i\right]}^2\right] \notag \\
    + &   \frac{\alpha^2}{S} \left(\frac{\eta_x^2\alpha^2\eta_y^2}{(1-\rho^2)^6}  + \frac{\eta_x^2}{\alpha \eta_y} + \frac{\eta_x^4}{\eta_y^2}+ \frac{\alpha^2\eta_x^4}{\eta_z^2} \right) \sum_{s=0}^{S-1}\E\left[\norm{\E\left[\bar v_{i+1}\middle|\cF_i\right]}^2\right] \notag \\
    + &  \frac{\alpha^2}{S} \frac{\eta_{x}^4\alpha^2}{(1-\rho^2)^3}\left(\alpha^2  + \frac{1}{\alpha \eta_y} + \frac{\eta_x^2}{\eta_y^2}+ \frac{\alpha^2\eta_x^2}{\eta_z^2} \right)  \sum_{s=0}^{S-1}\E\left[\norm{\E\left[\bar v_{i+1}\middle|\cF_i\right]}^2\right]  \notag \\ 
   + &   \left( \left(\frac{\eta_x^2\alpha^2\eta_y^2}{(1-\rho^2)^6}  + {\frac{\eta_x^2}{\alpha \eta_y}} + \frac{\eta_x^4}{\eta_y^2}+ \frac{\alpha^2\eta_x^4}{\eta_z^2} \right) + \frac{\eta_{x}^4\alpha^2}{(1-\rho^2)^3}\left(\alpha^2  + \frac{1}{\alpha \eta_y} + \frac{\eta_x^2}{\eta_y^2}+ \frac{\alpha^2\eta_x^2}{\eta_z^2} \right)  \right)\frac{\alpha^4\sigma^2}{n}\notag \\
    + &  \frac{\eta_{x}\alpha^2\sigma^2}{n} +\frac{\alpha^2\sigma^2}{2}\left(\frac{\alpha\eta_y}{n} +  \frac{\eta_z}{n}\right)  + \alpha^2 \left(\frac{\eta_{x}^2\alpha^2\sigma^2 }{(1-\rho^2)^4}\left(\frac{\alpha\eta_y}{n} +  \frac{\eta_z}{n}\right) + \frac{\eta_{x}^2\alpha^2\sigma^2}{(1-\rho^2)^3}\right)  \notag \\
   + &  \frac{\alpha^2}{S}\frac{\eta_{x}^2}{(1-\rho^2)^3} + 
        \frac{\alpha^2}{S}\left(1 +  \frac{\alpha^2\eta_x^3}{(1-\rho^2)^4}\right) \left(\frac{1}{\alpha\eta_y}\Delta_{y_*,0} +  \frac{\Delta_{Y,0}}{(1-\rho^2)^2} +\frac{\eta_y^2\Delta_{U_y,0}}{(1-\rho^2)^4} \right) \notag \\
   + &  \frac{\alpha^2}{S}\left(1 + \frac{\eta_{x}^2\alpha^2}{(1-\rho^2)^3}\right)\left(\frac{1}{\eta_z}\Delta_{z_*,0}+  \frac{\Delta_{Z,0}}{1-\rho^2}  + \frac{\eta_z^2\Delta_{U_z,0}}{(1-\rho^2)^5}\right) 
\end{align}
where
\begin{align*}
C_{y,vs} & \sim \frac{\eta_x^2\alpha^2\eta_y^2}{(1-\rho^2)^6}  + \frac{\eta_x^2}{\alpha \eta_y} + \frac{\eta_x^4}{\eta_y^2}; \quad   
C_{z,vs} \sim \frac{\eta_x^2}{\eta_z} + \frac{\alpha^2\eta_x^4}{\eta_z^2}  \notag \\
C_{vz} & \sim \frac{\alpha^2}{1-\rho^2}; \quad C_{vy} \sim \frac{\alpha^2}{1-\rho^2};\quad 
C_{y,\sigma} \sim \frac{\eta_y}{n \alpha} + \eta_y^2; \quad 
C_{z,\sigma}  \sim \frac{\eta_z}{n} + \eta_z^2
\end{align*}
Now we choose the parameters $\alpha$ and stepsizes as follows.
\begin{align}\label{inequ:t1:para}
\alpha = \Theta\left((1-\rho^2)^{1/2}(nS)^{1/7}\right), \eta_x = \eta_y = \Theta \left( \frac{(1-\rho^2)n^{2/7}}{S^{5/7}}\right), \alpha \eta_y = \eta_z = \Theta\left( \frac{(1-\rho^2)^{3/2}n^{3/7}}{S^{4/7}} \right), 
\end{align}
and set  a warm-start for $y, z$ such that 
\begin{align}\label{inequ:t1:para:2}
\Delta_{y_*,0}, \Delta_{z_*, 0} \leq \mathcal{O}\left(1/\alpha\right).
\end{align}
By properly choosing the parameters as in \eqref{inequ:t1:para}, \eqref{inequ:t1:para:2} and $S \geq n^{4/3}$,  the coefficient of $\frac{1}{S}\sum_{s=0}^{S-1}\E\left[\norm{\E\left[\bar v_{s+1}\middle|\cF_s\right]}^2\right]$ of RHS in \eqref{inequ:t1:} is smaller than that of LHS. We thus conclude
    \begin{align}
  &\sum_{s=0}^{S-1}\frac{1}{2S}\E\left[\norm{\nabla \Gamma^{\alpha}(\bar x_s)}^2\right] + \frac{1}{2S}\sum_{s=0}^{S-1}\norm{\E\left[\bar v_{s+1}\middle|\cF_s\right]}^2 \notag\\
\lesssim  \, & \frac{1}{(1-\rho^2)}\frac{\Delta_{x,0}}{(nS)^{2/7}} + \frac{ 1}{n^{9/7} S^{6/7}} + \frac{\sigma^2}{(nS)^{2/7}} + \frac{1}{(1-\rho^2)(nS)^{2/7}} = \cO\left(\frac{1}{(1-\rho^2)(nS)^{2/7}} \right). \label{ineq: final_ineq}
    \end{align}
Note that the difference $\Delta_{x,0}=\Gamma^{\alpha}(x_0)- \Gamma^{\alpha}(x^{\ast})$ can be controlled by a constant which is independent with $\alpha$:
\begin{align}
\Delta_{x,0} =  & \Gamma^{\alpha}(x_0) - \Gamma^{\alpha}(x^{\ast}) = \cL^{\alpha}(x_0, y_*^{\alpha}(x_0), z_*(x_0)) - \cL^{\alpha}(x^{\ast}, y_*^{\alpha}(x^{\ast}), z_*(x^{\ast})) \notag \\
= &  f(x_0, y_*^{\alpha}(x_0)) - f(x^{\ast}, y_*^{\alpha}(x^{\ast})) + \alpha \left(g(x_0, y_*^{\alpha}(x_0)) - g(x_0, z_*(x_0)) \right) \notag \\
& + \alpha \left(g(x^{\ast}, y_*^{\alpha}(x^{\ast})) - g(x^{\ast}, z_*(x^{\ast})) \right) \notag \\
= &   f(x_0, y_*(x_0)) - f(x^{\ast}, y_*(x^{\ast})) + f(x_0, y_*^{\alpha}(x_0)) - f(x_0, y_*(x_0)) \notag \\ & + f(x^{\ast}, y_*(x^{\ast})) - f(x^{\ast}, y_*^{\alpha}(x^{\ast})) + \alpha \left(g(x_0, y_*^{\alpha}(x_0)) - g(x_0, z_*(x_0)) \right) \notag \\
& + \alpha \left(g(x^{\ast}, y_*^{\alpha}(x^{\ast})) - g(x^{\ast}, z_*(x^{\ast})) \right) \notag \\
\leq & \Phi(x_0) - \Phi(x^{\ast}) + \ell_{f,0}\left\| y_*^{\alpha}(x_0) - y_*(x_0)\right\| + \ell_{f,0}\left\| y_*^{\alpha}(x^{\ast}) - y_*(x^{\ast})\right\| \notag \\
& + \alpha \frac{\ell_{g,1}}{2}\left\|y_*^{\alpha}(x_0)-z_*(x_0)\right\|^2 +  \alpha \frac{\ell_{g,1}}{2}\left\|y_*^{\alpha}(x^{\ast})-z_*(x^{\ast})\right\|^2 \notag \\
\leq & \Phi(x_0) - \Phi(x^{\ast}) + \frac{2\ell_{f,0}C_0}{\alpha}  + 2\alpha \frac{\ell_{g,1}}{2} \frac{C_0^2}{\alpha^2} \notag \\
\leq &  \Phi(x_0) - \Phi(x^{\ast}) + \frac{2\ell_{f,0}C_0 \mu_g}{\ell_{f,1}}  + \frac{\ell_{g,1} C_0^2\mu_g}{\ell_{f,1}} = \Phi(x_0) - \Phi(x^{\ast}) + \cO\left( \kappa^2 \ell_{g,1}\right),
\end{align}
where the first inequality follows from the gradient-Lipschitz of $g$ and the Lipschitz continuity of $f$ in $y$, and the second inequality uses Lemma~\ref{lem:y:ast}.

We then recall the relationship of $\norm{\nabla \Phi(\bar x_s)}^2$ and $\norm{\nabla \Gamma^{\alpha}(\bar x_s)}^2$ in Lemma~\ref{lem: gap_grad_bil_minmax}, we have
\begin{align*}
\sum_{s=0}^{S-1}\frac{1}{2S}\E\left[\norm{\nabla \Phi(\bar x_s)}^2\right] \leq 2\sum_{s=0}^{S-1}\frac{1}{2S}\E\left[\norm{\nabla \Gamma^{\alpha}(\bar x_s)}^2\right] + \frac{2}{\alpha^2} \leq \cO\left(\frac{1}{(1-\rho^2)(nS)^{2/7}} \right).
\end{align*}
We then notice that 
\[
    \left(\frac{1}{S}\sum_{s=0}^{S-1}\E\left[\norm{\nabla \Phi(\bar x_s)}\right]\right)^2  \leq \frac{1}{S}\sum_{s=0}^{S-1}\E\left[\norm{\nabla \Phi(\bar x_s)}\right]^2\leq \frac{1}{S}\sum_{s=0}^{S-1}\E\left[\norm{\nabla \Phi(\bar x_s)}^2\right]
\]
where the first inequality uses Cauchy-Schwarz inequality and the second one uses Jensen's inequality. Hence we know
\[
    \frac{1}{S}\sum_{s=0}^{S-1}\E\left[\norm{\nabla \Phi(\bar x_s)}\right]\leq \cO\left(\frac{1}{\sqrt{1-\rho^2}(nS)^{1/7}} \right).
\]
Moreover, we notice that from Lemma \ref{lem: consensus}, \eqref{inequ:t1:para}, and \eqref{inequ:t1:para:2} we know
\begin{align*}
    & \frac{1}{n}\sum_{s=0}^{S}\E\left[\norm{\mX_{s+1} - \bar x_{s+1}\bfonet_n}^2 \right] \\ 
      & \lesssim\frac{\eta_{x}^2}{(1-\rho^2)^3}\left( C_{vz}C_{z,v} + C_{vy}C_{y,v} \right)\sum_{i=0}^{S}\E\left[\norm{\E\left[\bar v_{i+1}\middle|\cF_i\right]}^2\right] \\
   & + \frac{\eta_{x}^2}{(1-\rho^2)^3}\left(4C_{vz}C_{z,vs} + 4C_{vy}C_{y,vs}  + \frac{C_{vv}}{n}\right)\sum_{i=0}^{S-1}\E\left[\norm{\bar v_{i+1}}^2\right]  \\
    & +\frac{\eta_{x}^2 C_{vz}}{(1-\rho^2)^3} \left(C_{z_*,0}\Delta_{z_*,0} +  C_{Z,0}\Delta_{Z,0}  + C_{U_z,0}\Delta_{U_z,0}  \right)  \\
   & + \frac{C_{vy}\eta_x^2}{(1-\rho^2)^3} \left( C_{y_*,0}\Delta_{y_*,0} +  C_{Y,0}\Delta_{Y,0} + C_{U_y,0}\Delta_{U_y,0} \right)  \\
    & + \frac{\eta_{x}^2S}{(1-\rho^2)^3}\left(C_{vz}C_{z,\sigma} \sigma_z^2 +  C_{vy}C_{y,\sigma} \sigma_y^2 \right)  + \frac{S\eta_{x}^2\sigma_{x}^2}{(1-\rho^2)^3} + \frac{\eta_{x}^2\ell_{f,0}^2(1+\rho^2)}{(1-\rho^2)^3} \\
    &\lesssim \frac{\eta_{x}^2}{(1-\rho^2)^4}\left(\frac{\alpha^2 \eta_x^2}{\eta_z^2} + \frac{\eta_x^2}{\eta_y^2} \right)\sum_{i=0}^{S}\E\left[\norm{\E\left[\bar v_{i+1}\middle|\cF_i\right]}^2\right] \\
    & + \frac{\eta_{x}^4\alpha^2}{(1-\rho^2)^4}\left(\frac{\alpha^2\eta_y^2}{(1-\rho^2)^6}  + \frac{1}{\alpha \eta_y} + \frac{\eta_x^2}{\eta_y^2}+ \frac{1}{\eta_z} +  \frac{\alpha^2\eta_x^2}{\eta_z^2} + 1\right)\sum_{s=0}^{S-1}\E\left[\norm{\bar v_{i+1}}^2\right] \\
    & + \frac{\alpha^2\eta_x^3}{(1-\rho^2)^4}\left(\frac{1}{\alpha\eta_y}\Delta_{y_*,0} +  \frac{\Delta_{Y,0}}{(1-\rho^2)^2} +\frac{\eta_y^2\Delta_{U_y,0}}{(1-\rho^2)^4} \right) \\
    & + \frac{\eta_{x}^2\alpha^2}{(1-\rho^2)^3}\left(\frac{1}{\eta_z}\Delta_{z_*,0}+  \frac{\Delta_{Z,0}}{1-\rho^2}  + \frac{\eta_z^2\Delta_{U_z,0}}{(1-\rho^2)^5}\right) \\
    & + \frac{\alpha^2\eta_{x}^2S}{(1-\rho^2)^4}\left(\left(\frac{\eta_z}{n} + \eta_z^2\right) + \left(\frac{\eta_y}{n \alpha} + \eta_y^2\right)\alpha^2\right)\sigma^2  + \frac{S\alpha^2\eta_{x}^2\sigma^2}{(1-\rho^2)^3} + \frac{\eta_{x}^2\ell_{f,0}^2(1+\rho^2)}{(1-\rho^2)^3} \\
    &\lesssim \frac{\eta_x^2}{(1-\rho^2)^4}\sum_{i=0}^{S}\E\left[\norm{\E\left[\bar v_{i+1}\middle|\cF_i\right]}^2\right] + \frac{\eta_{x}^4\alpha^2}{(1-\rho^2)^4}\left(\frac{\alpha^2\eta_y^2}{(1-\rho^2)^6} + \frac{1}{\alpha\eta_x}\right)\sum_{s=0}^{S-1}\E\left[\norm{\bar v_{i+1}}^2\right] \\
    & + \frac{\eta_x^2}{(1-\rho^2)^4} + \frac{\eta_{x}}{(1-\rho^2)^3} + \frac{\alpha^3\eta_{x}^3S\sigma^2}{n(1-\rho^2)^4} + \frac{S\alpha^2\eta_{x}^2\sigma^2}{(1-\rho^2)^3} + \frac{\eta_{x}^2\ell_{f,0}^2(1+\rho^2)}{(1-\rho^2)^3}.
 \end{align*}
 Using \eqref{ineq: bdd_var} and \eqref{ineq: final_ineq} in the above inequality, we know
 \begin{align*}
    & \frac{1}{nS}\sum_{s=0}^{S}\E\left[\norm{\mX_{s+1} - \bar x_{s+1}\bfonet_n}^2 \right] \\ 
    = &\cO\left(\frac{n^{2/7}}{(1-\rho^2)^3S^{12/7}} \right) + \cO\left(\frac{n^{8/7}}{S^{20/7}}\left(\frac{n^{6/7}}{(1-\rho^2)^3S^{8/7}} + \frac{S^{4/7}}{(1-\rho^2)^{3/2}n^{3/7}}\right)\right) + \cO\left(\frac{n^{2/7}}{(1-\rho^2)^2S^{12/7}} + \frac{n^{6/7}}{S^{8/7}} + \frac{n^{4/7}}{(1-\rho^2)S^{17/7}}\right) \notag \\
    &=\cO\left(\frac{n^{6/7}}{S^{8/7}} + \frac{n^{2/7}}{(1-\rho^2)^3S^{12/7}} \right).
 \end{align*}
This indicates that
\begin{align*}
    &\frac{1}{n}\left(\min_{0 \leq s \leq S-1}\E\left[\norm{X_s - \bar{x}_s \bfone_n}\right]\right)^2\leq \frac{1}{n}\left(\frac{1}{S}\sum_{s=0}^{S}\E\left[\norm{\mX_{s+1} - \bar x_{s+1}\bfonet_n}\right]\right)^2\\
    \leq &\frac{1}{nS}\sum_{s=0}^{S}\E\left[\norm{\mX_{s+1} - \bar x_{s+1}\bfonet_n}\right]^2\leq\frac{1}{nS}\sum_{s=0}^{S}\E\left[\norm{\mX_{s+1} - \bar x_{s+1}\bfonet_n}^2 \right] = \cO\left(\frac{n^{6/7}}{S^{8/7}} + \frac{n^{2/7}}{(1-\rho^2)^3S^{12/7}}\right)
\end{align*}
which gives
\begin{align*}
    \min_{0 \leq s \leq S-1}\frac{1}{n}\E\left[\norm{X_s - \bar{x}_s \bfone_n}\right] = \cO\left(\frac{1}{n^{1/14}S^{4/7}} + \frac{n^{1/7}}{(1-\rho^2)^{3/2}S^{6/7}}\right).
\end{align*}

\end{proof}

\section{Details of the Experiments}\label{numerical:appendix}
In this part, we provide the details of the experiments. For hyperparameters in the algorithms, we followed the settings in the corresponding papers. For example, we chose stepsizes to be $a/\sqrt{K}$ for MA-DSBO as suggested in their paper \cite{chen2022decentralizeddsbo}, while for our experiments we picked stepsizes and multiplier $\alpha$ according to Theorem \ref{thm: warm_start_rate}. Iterates are randomly initialized from normal distributions. The network topologies are chosen as ring networks for all settings. In synthetic data experiments, we choose outer loop stepsizes: 0.03, 
inner loop stepsizes: 0.03, and $\alpha$: 1.0. In particular, for our algorithm DSGDA-GT, in the MNIST experiments: outer loop stepsize is 0.3; inner loop stepsize is  0.3; and $\alpha$ is 1.0. The number of samples in each agent:
training set: 7500; validation set: 7500; test set: 1000.

\end{document}